\documentclass[reqno]{amsart}

\usepackage{amssymb,mathtools,amsthm}
\usepackage{mathtools}
\usepackage[foot]{amsaddr}
\usepackage{paralist}
\usepackage{enumitem}
\usepackage{autobreak}
\usepackage[mathscr]{euscript}
\usepackage{booktabs}
\usepackage{caption}
\usepackage{subcaption}

\usepackage{placeins}

\usepackage{tikz}
\usetikzlibrary{decorations.pathreplacing}
\usetikzlibrary{arrows,arrows.meta,decorations,backgrounds,patterns,matrix,shapes,fit,calc,shadows,plotmarks}
\usepackage{pgfplots,pgfplotstable}
\pgfplotsset{compat=newest}

\usepgfplotslibrary{units}

\usetikzlibrary{external}
\pgfrealjobname{main}
\tikzsetexternalprefix{}
\tikzset{/tikz/external/only named=true}
\tikzset{external/system call={lualatex \tikzexternalcheckshellescape -halt-on-error -interaction=batchmode -jobname "\image" "\texsource"}}
\tikzset{/tikz/external/mode=list and make}
\pgfmathdeclarefunction{lg10}{1}{ \pgfmathparse{ln(#1)/ln(10)}}

\usepackage[T1]{fontenc}
\usepackage[utf8]{inputenc}
\usepackage[separate-uncertainty = true,multi-part-units=single,binary-units]{siunitx}

\usepackage[bibencoding=utf8, backend=biber, sorting=nyt, sortcites, citestyle=numeric-comp, doi=false, giveninits=true, url=false, isbn=false, date=year, maxbibnames=100]{biblatex}

\renewbibmacro{volume+number+eid}{\printfield{volume}}
\DeclareFieldFormat{pages}{#1}
\DeclareFieldFormat[article]{volume}{\textbf{#1}}
\DeclareFieldFormat[article]{title}{#1\isdot}
\DeclareFieldFormat[thesis]{title}{\emph{#1}\midsentence}
\DeclareFieldFormat[incollection]{title}{\emph{#1}\midsentence}
\renewrobustcmd*{\bibinitdelim}{}

\DeclareNameAlias{labelname}{given-family}

\DeclareFieldFormat[inproceedings]{citetitle}{\textit{#1}}
\DeclareFieldFormat[inproceedings]{title}{\textit{#1}}

\DeclarePunctuationPairs{dot}{?}

\DeclareBibliographyDriver{thesis}{%
  \usebibmacro{bibindex}%
  \usebibmacro{begentry}%
  \usebibmacro{author}%
  \setunit{\printdelim{nametitledelim}}\newblock
  \usebibmacro{title}%
  \newunit
  \printlist{language}%
  \newunit\newblock
  \usebibmacro{byauthor}%
  \newunit\newblock
  \printfield{note}%
  \newunit\newblock
  \printfield{type}%
  \newunit
  \usebibmacro{institution+location+date}%
  \newunit\newblock
  \usebibmacro{chapter+pages}%
  \newunit
  \printfield{pagetotal}%
  \newunit\newblock
  \iftoggle{bbx:isbn}
    {\printfield{isbn}}
    {}%
  \newunit\newblock
  \usebibmacro{url+urldate}%
  \newunit\newblock
  \usebibmacro{addendum+pubstate}%
  \setunit{\bibpagerefpunct}\newblock
  \usebibmacro{pageref}%
  \newunit\newblock
  \iftoggle{bbx:related}
    {\usebibmacro{related:init}%
     \usebibmacro{related}}
    {}%
  \usebibmacro{finentry}}
\DeclareBibliographyDriver{misc}{%
  \usebibmacro{bibindex}%
  \usebibmacro{begentry}%
  \usebibmacro{author/editor+others/translator+others}%
  \setunit{\printdelim{nametitledelim}}\newblock
  \usebibmacro{title}%
  \newunit
  \printlist{language}%
  \newunit\newblock
  \usebibmacro{byauthor}%
  \newunit\newblock
  \usebibmacro{byeditor+others}%
  \newunit\newblock
  \printfield{howpublished}%
  \newunit\newblock
  \printfield{type}%
  \newunit
  \printfield{version}%
  \newunit
  \printfield{note}%
  \newunit\newblock
  \usebibmacro{organization+location+date}%
  \newunit\newblock
  \usebibmacro{url+urldate}%
  \newunit\newblock
  \usebibmacro{addendum+pubstate}%
  \setunit{\bibpagerefpunct}\newblock
  \usebibmacro{pageref}%
  \newunit\newblock
  \iftoggle{bbx:related}
    {\usebibmacro{related:init}%
     \usebibmacro{related}}
    {}%
  \usebibmacro{finentry}}
\DeclareBibliographyDriver{manual}{%
  \usebibmacro{bibindex}%
  \usebibmacro{begentry}%
  \usebibmacro{author/editor}%
  \setunit{\printdelim{nametitledelim}}\newblock
  \usebibmacro{title}%
  \newunit
  \printlist{language}%
  \newunit\newblock
  \usebibmacro{byauthor}%
  \newunit\newblock
  \usebibmacro{byeditor}%
  \newunit\newblock
  \printfield{edition}%
  \newunit\newblock
  \usebibmacro{series+number}%
  \newunit\newblock
  \printfield{type}%
  \newunit
  \printfield{version}%
  \newunit
  \printfield{note}%
  \newunit\newblock
  \printlist{organization}%
  \newunit
  \usebibmacro{publisher+location+date}%
  \newunit\newblock
  \usebibmacro{chapter+pages}%
  \newunit
  \printfield{pagetotal}%
  \newunit\newblock
  \iftoggle{bbx:isbn}
    {\printfield{isbn}}
    {}%
  \newunit\newblock
  \usebibmacro{url+urldate}%
  \newunit\newblock
  \usebibmacro{addendum+pubstate}%
  \setunit{\bibpagerefpunct}\newblock
  \usebibmacro{pageref}%
  \newunit\newblock
  \iftoggle{bbx:related}
    {\usebibmacro{related:init}%
     \usebibmacro{related}}
    {}%
  \usebibmacro{finentry}}

\makeatletter
\renewbibmacro{in:}{%
  \ifboolexpr{%
     test {\ifentrytype{incollection}}%
     or
     test {\ifentrytype{inproceedings}}%
  }{\printtext{\bibstring{in}\intitlepunct}}{}%
}
\makeatother
\newtheorem{theorem}{Theorem}[section]

\newtheorem{lemma}[theorem]{Lemma}
\newtheorem{corollary}[theorem]{Corollary}

\theoremstyle{definition}
\newtheorem{definition}[theorem]{Definition}

\theoremstyle{remark}

\numberwithin{equation}{section}

\calclayout

\DeclareMathOperator*{\argmin}{arg\,min}

\usepackage{hyperref}

\begin{document}

\makeatletter
\newcommand{\abs}[1]{%
  \@ifnextchar^{\abs@sp{#1}}{\abs@nosp{#1}}%
}
\newcommand{\abs@nosp}[1]{\mathchoice{\left\lvert #1\right\rvert}{\lvert #1\rvert}{\lvert #1\rvert}{\lvert #1\rvert}}
\newcommand{\abs@sp}[3]{\mathchoice{\left\lvert #1\right\rvert^{#3}}{\lvert #1\rvert^{#3}}{\lvert #1\rvert^{#3}}{\lvert #1\rvert^{#3}}}
\newcommand{\norm}[2][]{%
  \@ifnextchar^{\norm@sp[{#1}]{#2}}{\norm@nosp[{#1}]{#2}}%
}
\newcommand{\norm@sp}[4][]{%
  \mathchoice{\left\lVert #2 \right\rVert_{#1}^{#4}}{\lVert #2 \rVert_{#1}^{#4}}{\lVert #2 \rVert_{#1}^{#4}}{\lVert #2 \rVert_{#1}^{#4}}%
}
\newcommand{\norm@nosp}[2][]{%
  \mathchoice{\left\lVert #2 \right\rVert_{#1}}{\lVert #2 \rVert_{#1}}{\lVert #2 \rVert_{#1}}{\lVert #2 \rVert_{#1}}%
}
\newcommand{\scalar}[3][]{%
  \@ifnextchar^{\scalar@sp[{#1}]{#2}{#3}}{\scalar@nosp[{#1}]{#2}{#3}}%
}
\newcommand{\scalar@nosp}[3][]{\mathchoice{{\left\langle #2 , #3 \right\rangle_{#1}}}{{\langle #2 , #3 \rangle_{#1}}}{{\langle #2 , #3 \rangle_{#1}}}{{\langle #2 , #3 \rangle_{#1}}}}
\newcommand{\scalar@sp}[5][]{\mathchoice{{\left\langle #2 , #3 \right\rangle_{#1}^{#5}}}{{\langle #2 , #3 \rangle_{#1}^{#5}}}{{\langle #2 , #3 \rangle_{#1}^{#5}}}{{\langle #2 , #3 \rangle_{#1}^{#5}}}}
\makeatother
\newcommand{\set}[2][]{\mathchoice{\left\{#2\right\}_{#1}}{\left.\{#2\}_{#1}\right.}{\left.\{#2\}_{#1}\right.}{\left.\{#2\}_{#1}\right.}}

\title[Vector-valued Splines for the MEG and EEG problem]{Vector-valued Spline Method for the Spherical Multiple-shell Electro-magnetoencephalography Problem}


\author{S. Leweke}
\author{O. Hauk}
\author{V. Michel}
\address{Geomathematics Group, Department of Mathematics, University of Siegen, Walter-Flex Stra{\ss}e 3, 57068 Siegen}
\address{MRC Cognition and Brain Sciences Unit, University of Cambridge, 15 Chaucer Road, Cambridge CB2 7EF, UK}
\address{Geomathematics Group, Department of Mathematics, University of Siegen, Walter-Flex Stra{\ss}e 3, 57068 Siegen}
\email{leweke@mathematik.uni-siegen.de}
\email{olaf.hauk@mrc-cbu.cam.ac.uk}
\email{michel@mathematik.uni-siegen.de}


\date{}

\begin{abstract}
  Human brain activity is based on electrochemical processes, which can only be measured invasively. 
  For this reason, quantities such as magnetic flux density (via MEG) or electric potential differences (via EEG) are 
  measured non-invasively in medicine and research. 
  The reconstruction of the neuronal current from the measurements is a severely ill-posed problem 
  though the visualization of the cerebral activity is one of the main research tools in cognitive neuroscience.  
  Here, using an isotropic multiple-shell model for the geometry of the human head 
  and a quasi-static approach for modeling the electro-magnetic processes, we derive a novel vector-valued spline method based on reproducing kernel Hilbert spaces (RKHS) in order to reconstruct the neuronal current from the measurements. 
  The presented vector spline method follows the path of former spline approaches and provides classical minimum norm properties.
  In addition, it minimizes the (infinite-dimensional) Tikhonov-Philips functional which handles the instability of the inverse problem. 
  This optimization problem reduces to solving a finite-dimensional system of linear equations without loss of information, due to its particular construction. 
  It results in a unique solution which takes into account that only the harmonic and solenoidal component of the neuronal current affects the measurements.
  In addition, we prove a convergence result: the solution achieved by the vector spline method converges to the generator of the data as the number of measurements increases.  
  The vector splines are applied to the inversion of several synthetic test cases, where the irregularly distributed data situation could be handled very well.
  Combined with several parameter choice methods, numerical results are shown for synthetic test cases with and without additional Gaussian white noise.
  Former approaches based on scalar splines are outperformed by the novel vector splines results with respect to the normalized root mean square error. 
  Finally, results for real data are demonstrated. 
  They can be computed quickly and are reasonable with respect to physiological expectations.
\end{abstract}

\maketitle


\frenchspacing

\section{Introduction}

Neurons in the cerebrum use electrochemical processes for signaling. 
These electrical activities are often referred to as cerebral or brain activity, respectively.
First of all, the signal processing and interplay of neuronal cells in the brain is responsible for human behavior and cognition. 
Accordingly, defects in the signal processing cause diseases, such as epilepsy or schizophrenia.
Second, localization of the brain activity can help to answer open questions in the field of cartography of the brain structures and to further understand the human brain functionality. 
Hence, an accurate visualization of the cerebral activity is one of the main tools in research and diagnosis, \cite{Kandel2013}.

Since non-invasive measurements outside the head are often the only available option, the physical fundamentals connecting the brain activity with measurable quantities need to be understood and modeled adequately.
The mathematical modeling of these fundamentals results in an electric current inside the cerebrum which can be split into the neuronal current $J$ and its induced Ohmic current, where only the neuronal current is of interest \cite{HE2005}.
The cerebral current induces a macroscopic electric potential $u_s$ and a magnetic field $B$, which are transmitted through the conductive brain tissues and can be measured outside the head non-invasively provided that the number of simultaneously active neurons is large enough. 

In order to formulate the relation between the neuronal current and the measured quantities, we follow the path of optimize-then-discretize approaches in \cite{Leweke2018,Leweke2020,Dassios2009c,Fokas2009,Fokas2012,deMunck1988}. 
Via this ansatz, the problem is solved analytically as far as possible. 
This has the advantage that structures like non-visible parts of the current are conserved up to the implementation.
Otherwise, model assumptions are required to enable the analysis. 
For our approach, it is sufficient to model the neuronal current as continuously distributed current, the head via the common multiple-shell model, \cite{Hamalainen1993,deMunck1988}, and the physics via quasi-static Maxwell's equations, \cite{Hamalainen1993,Plonsey69,Plonsey1967}.  

Within this framework, the problem of directly reconstructing the vector-valued neuronal current from magnetic flux density $\nu \cdot B$ measurements by the magnetoencephalograph (MEG), \cite{Cohen2009}, and electric potential differences values on the scalp obtained via the electroencephalograph (EEG), \cite{Cook2002}, was addressed in \cite{Leweke2018,Leweke2020}.
Therein, singular value decompositions (SVD) of the corresponding operators are derived based on a novel set of vector-valued orthonormal basis functions. 
First of all, the SVD paves the path for determining the parts of the neuronal current which can be reconstructed from the measurements.
This is closely related to the specification of the operator null space.
Moreover, asymptotic behavior of the singular values characterizes the ill-posedness of the inverse problems.
Related problems of reconstructing scalar parts of the cerebral current from these measurements have been discussed before in \cite{Dassios2009c,Fokas2009,Fokas2012,deMunck1988}.

Based on the insights achieved by SVDs, it is revealed that 
for the reconstruction of the neuronal current $J$ from the given measurements several challenges need to be tackled:
\begin{description}
  \item[Non-uniqueness] If a solution of the inverse problem exists, it is not unique, since only the harmonic part of the solenoidal directions of the neuronal current is not silent for the MEG and EEG, \cite{Leweke2020}.  
  \item[Instability] Due to the distance between the head and the MEG device and the distance between the cerebrum and the outer surface of the scalp, terms emerge in the formulae which are associated to a continuation of the magnetic/electric field down to a sphere with lower radius. 
  This is a well-known mathematical problem, which also occurs in the geosciences for the magnetic and the gravitational field. 
  It leads to exponentially diverging singular values for the inverse operator, which implies a severe ill-posedness of the inverse problem. 
  As a consequence, small changes in the data (e.g. by means of noise) have a huge impact on the reconstruction. 
  \item[Noise] The achieved data is noisy, due to physiological noise (e.g. head, eye, jaw, or neck movements, magnetic field generated by the heart) and non-physiological noise (e.g. technical noise level of the devices, particles attached or implanted to the patient's body, other outer fields and currents), \cite{Hamalainen1993,Leweke2020}. 
  \item[Data Distribution] There are only few sensor positions available (commonly a few dozens to a few hundreds in the case of only MEG or EEG measurements), which are irregularly distributed and have a major gap in the area of the face and the lower half of the head.
\end{description}

The aim of this paper is to reconstruct the vector-valued neuronal current directly by developing a robust and stable numerical optimize-then-discretize method for solving these functional inverse problems addressing these challenges.
Within this paper, we will mainly answer the following questions 
\begin{itemize}
  \item Does our method produce a reasonable and correct reconstruction?
  \item Is the method stable with respect to increasing noise level?
  \item Is there an advantage of reconstructing the vector-valued current directly instead of reconstructing scalar-valued components as an intermediate step?
\end{itemize}
First of all, regularization techniques are required to handle the ill-posedness: the best-approximate solution obtained from noisy data is not suitable for the reconstruction, since the generalized Moore-Penrose inverse is unbounded if it exists, \cite{Engl1996}. 
Popular regularizations include generalized Tikhonov-Philips functionals.
In the context of approximations of the neuronal current based on continuously distributed MEG and EEG models, a regularization method based on global orthonormal basis functions (related to spherical harmonics) and scalar spherical splines have been used before, \cite{Fokas2012}. 
In addition, for the EEG model a hybrid analytical-numerical algorithm using OpenMEG exists, \cite{Hashemzadeh2020}. 
Besides, discrete models (e.g. MNE, LORETA, FOCUSS) often yield in a finite-dimensional Tikhonov-regularized normal equation, which is solved by a variety of optimization algorithms, \cite{Grech2008,HE2005}.
Though such methods have also become established, their usage has to be seen critically from the mathematical point of view, since the discretization approach does not consider the known results on the large null spaces of the inverse problems. 
Hence, it is not possible to distinguish uniquely determined components of the solution from remaining degrees of freedom including possible artifacts.

Within this paper, we extend the idea of the regularization method based on global orthonormal basis functions and construct vector-valued as well as scalar-valued splines based on the SVD of the integral operators. 
In order to obtain a unique solution, we assume additionally that the neuronal current satisfies an $\mathrm{L}^2$-minimum-norm condition. 
This results (in the unregularized case) in the best-approximate solution of the inverse problems, \cite[Thm. 13.6]{Leweke2020}, which is entirely contained in the orthogonal complement of the operator null space.

The idea of using a reproducing kernel to regularize an inverse problem is not new. 
For instance, \cite{Nashed1974} presented a regularization theory for linear inverse problems in RKHSs. 
Regarding examples of approaches tailored for specific domains, (functional) spherical splines based on reproducing kernels (which will shortly be referred to \emph{spherical splines}) are a widely known tool for interpolation and approximation, which is frequently used for solving (inverse) problems in the geosciences and beyond.
This method goes back to \citeauthor{Freeden1981} and \citeauthor{Wahba1981} and provides us with a best-approximate solution, \cite{Freeden1981,Freeden1998,Freeden1981b,Wahba1981}, in the unregularized case. 
They are constructed via spatially localized reproducing kernels, hence local changes in the data have only mainly local effects in the reconstruction. 
In addition, spherical splines have the smallest (RKHS) norm among all functions fitting the data.
In addition, the formulation via splines allows, without loss of information, to transfer the infinite-dimensional optimization problem of minimizing the Tikhonov-Philips functional to a finite-dimensional system of linear equations.
Due to these nice properties, scalar spherical splines have been used in several applications on the unit sphere, \cite{Michel2015Geo,Michel2013,Freeden1993,Schneider1996}, and on the ball, \cite{Amirbekyan2008,Berkel2010b,Berkel2010a}. 
Also in the MEG and EEG setting they have been used before in \cite{Fokas2012}.

In this paper, we will construct mainly two types of spherical splines over the three-dimensional ball. 
For the first type, further decompositions of the neuronal current are required in order to achieve a reduction to a scalar-valued problem. 
Then, scalar-valued splines connected to tailor-made orthonormal basis functions are constructed to solve these problems. 
This follows the path of \cite{Fokas2012} but it takes into account the novel knowledge about those parts of the neuronal current which are silent to MEG and EEG.
Afterwards, the scalar solutions have to be transferred back to the neuronal current which we are interested in. 
The second and novel type of spherical splines on the ball is vector-valued and can be used to solve the vector-valued MEG and EEG problem directly. 
This novel approach conserves the excellent properties of former (scalar) spline methods like the minimum properties. 

From a theoretical point of view, a direct vector approach for the entire neuronal current has several advantages over the scalar approaches. 
First of all, one obtains a direct relation between the desired quantity and the measured one which reduces the effort of transforming the scalar quantities back to the current. 
This is an advantage, since in several cases, there exists no one-to-one relation between the scalar and the vector quantities. 
This takes effects especially when adding uniqueness constraints.
Besides, additional smoothness or boundary conditions or complementary gauges, which may not have a physical or medical meaning, are required for the current in order to derive or transfer the scalar solutions. 
Finally, from the vector approach and its characterization of the null space it became explicit that only the combined inversion of MEG and EEG data can depict a comprehensive image of the current. 
The vector approach yields the easiest possibility to combine these to inversions compared to the scalar approaches, see \cite{Leweke2020,Leweke2018}.

Within this paper, we will show that also from a numerical point of view, the vector approach has several advantages compared to the scalar ones which manifests in lower approximation errors and a higher stability with respect to the influence of noise. 
This is done by several synthetic numerical tests. 
Afterwards the vector spline method is applied to real data, where the corresponding approximation of the neuronal current answers all physiological expectations.

The rest of the paper is structured as follows: In Sec.~\ref{sec:Prelim} short introductions into (vector-valued) orthonormal basis functions with generalized Fourier analysis on the ball, the multiple-shell model for the MEG and EEG problem, and the scalar reproducing kernel based spline method for inverse problems are given. 
In Sec.~\ref{sec:RKHSball} vector-valued functional RKHS and the corresponding splines on the ball are constructed.
Afterwards, properties of the spline such as minimum properties, approximation property, and convergence results are proven. 
In Sec.~\ref{sec:Foundations} the foundations for the implementation of the particular MEG and EEG scalar and vector spline methods are summarized.
The particular numerical framework and the numerical results for the synthetic test cases as well as the inversion of the real data are stated in Sec.~\ref{sec:numerics}. 
Therein, the used methods are also checked against each other.
Eventually, conclusions are presented in Sec.~\ref{sec:conclude}.

\section{Preliminaries}\label{sec:Prelim}
\subsection{Construction of Suitable Orthonormal Basis Functions}

As stated before, an adequate set of orthonormal basis functions is required for the SVD of the integral operators. 
In addition, due to Mercers representation theorem, \cite{Mercer1909}, it is well-known that reproducing kernels are closely connected to orthonormal basis functions. 
Eventually, these particular basis functions build the foundation of the presented spline methods. 
Before we have a closer look at these basis functions, we give a short introduction to the used notations.

The Euclidean inner product is denoted by $\cdot$, the vector product by $\times$, and the tensor product by $\otimes$. 
In addition, we will use throughout the paper the abbreviations $x = r \xi$ and $y = s \eta$ with unit vectors $\xi$, $ \eta \in  S \coloneqq  S_{1} \subset \mathbb{R}^3$ and radii $r = |x|$, $s=|y|$, where $S_1$ is the unit sphere in $\mathbb{R}^3$. 
This decomposition is unique for all $x \in \mathbb{R}^3\setminus \{0\}$. 
Furthermore, $B_R \subset \mathbb{R}^3$ is the ball with radius $R$ and center $0$.
In our notation, $\nabla^\ast_{\xi}$ denotes the part of the gradient $\nabla$ containing the tangential derivatives divided by $r$, \cite[Eq. (2.136)]{Freeden2009}. $\nabla^\ast_{\xi}$ is often called the surface gradient. The differential operator $L^\ast$ (independent of the radius $r$) is defined by
\begin{equation}\label{defi:Last}
  L^\ast_{\xi} \coloneqq x \times \nabla_{x} = \xi \times \nabla^\ast_{\xi}
\end{equation}
and called the surface curl operator. 
The Beltrami operator is given by $\Delta^\ast_{\xi} = \nabla^\ast_{\xi} \cdot \nabla^\ast_{\xi}$, \cite[Eq. (2.140)]{Freeden2009}, and is the part of the Laplacian independent of the radius.
Note that variables as indices of operators indicate the dependence to which the operator is implied, since this is otherwise not always unique (e.g. \eqref{eq:LegPol} below).

The construction of our basis function is in all cases based on spherical harmonics.
Recall that a function of $\mathrm{Harm}_n( S)$ (i.e. the space of all homogeneous, harmonic polynomials of degree $n$ restricted to the unit sphere $ S$, \cite[Def. 3.22]{Freeden2009}) is called a spherical harmonic of degree $n \in\mathbb{N}_0\coloneqq \mathbb{N} \cup \set{0}$.
With $\set[j=1,\dots,2n+1]{Y_{n,j}}$ we denote an $\mathrm{L}^2( S)$-orthonormal set in $\mathrm{Harm}_n( S)$, \cite[Rem. 3.25]{Freeden2009}. 
For more details on scalar spherical harmonics and their properties, see for instance \cite{Freeden2009}. 
For an introduction to (vector-valued) Lebesgue spaces, see \cite{Bauer2001}.

By means of the scalar spherical harmonics, we can define a complete $\mathrm{L}^2( S,\mathbb{R}^3)$-orthonormal system of vector-valued spherical harmonics, \cite[Thm. 5.56]{Freeden2009}, which goes back to Edmonds, \cite{Edmonds57}, 
and is defined for example in \cite[Eq. (5.309)-(5.311)]{Freeden2009} by
\begin{equation}\label{eq:DefiVecSph}
  \tilde{{y}}_{n,j}^{(i)}(\xi) \coloneqq \left(\tilde{\mu}_n^{(i)}\right)^{-1/2} \tilde{o}^{(i)}_{n,\xi} Y_{n,j}(\xi), \qquad \xi \in  S,
\end{equation}
where 
\begin{equation}\label{eq:Defioi}
  \tilde{\mu}_n^{(i)} \coloneqq \begin{cases}
    (n+1)(2n+1), & \text{for } i=1, \\
    n(2n+1), & \text{for } i=2, \\
    n(n+1), & \text{for } i=3,
  \end{cases}
  \qquad 
    \tilde{o}^{(i)}_{n,\xi} \coloneqq \begin{cases}
      (n+1) \xi  - \nabla^\ast_{\xi}, & \text{for } i=1, \\
      n \xi  + \nabla^\ast_{\xi}, & \text{for } i=2, \\
      L^\ast_{\xi}, & \text{for } i=3, 
  \end{cases}
\end{equation}
for all $i\in\{1, 2, 3\}$, $n \in \mathbb{N}_{0_i}$, and $j=1,\dots,2n+1$. 
We use $\mathbb{N}_{0_i}$ as an abbreviation for $\mathbb{N}_0$ in the case of $i=1$ and for $\mathbb{N}$ in the case of $i\in\{2, 3\}$. 
Note that these vector-valued spherical harmonics are homogeneous harmonic polynomials. 
This and more information on vector spherical harmonics can be found in \cite{Freeden2009}.
In analogy, we can define, for all $i\in\{1, 2, 3\}$ and $n \in \mathbb{N}_{0_i}$, Edmonds-vector-Legendre polynomials by means of the scalar Legendre polynomials $P_n$, \cite[Eq. (3.165)]{Freeden2009}, of degree $n\in\mathbb{N}_0$ and type $i \in \{1,2,3\}$ as in \cite[Lem. 5.63]{Freeden2009},
\begin{equation}\label{eq:LegPol}
  \tilde{{p}}^{(i)}_n( {\xi},{\eta}) \coloneqq \left(\tilde{\mu}_n^{(i)}\right)^{-1/2}\tilde{o}^{(i)}_{n,\xi}P_n({\xi}\cdot {\eta}) , \qquad \xi,\, \eta \in  S. 
\end{equation}
Besides this, the Legendre polynomials and their vectorial counterpart enable  addition theorems, \cite[Thm. 3.26, Thm. 5.64]{Freeden2009}. These imply for all types and degrees the representations
\begin{equation}\label{eq:AdditionThm}
  \sum_{j=1}^{2n+1} Y_{n,j}({\xi}) Y_{n,j}({\eta}) = \frac{2n+1}{4\pi} P_n( {\xi}\cdot{\eta}), \qquad \sum_{j=1}^{2n+1} \tilde{{y}}^{(i)}_{n,j}({\xi}) Y_{n,j}({\eta}) = \frac{2n+1}{4\pi} \tilde{{p}}^{(i)}_n( {\xi},{\eta}), \qquad \xi,\, \eta \in  S. 
\end{equation} 
Furthermore, we also obtain for (pointwise) Euclidean norms
\begin{equation}\label{eq:AddThmAbs}
  \sum_{j=1}^{2n+1} \abs{\tilde{{y}}^{(i)}_{n,j}({\xi})}^2 = \frac{2n+1}{4\pi}, \qquad \xi \in  S. 
\end{equation}
Based on a separation ansatz, we combine the vector spherical harmonics with orthogonal Jacobi functions $P_m^{(\alpha,\beta)}$ for the radial part in order to achieve an orthonormal set of functions over the ball. 
For further details on Jacobi polynomials, see, for instance, \cite{Szego1975}.
\begin{theorem}\label{thm:ONBBall}
  The set of vector-valued functions $\tilde{{g}}^{(i)}_{m,n,j}(R;\cdot)$ for $i\in\{1, 2, 3\}$, $m\in\mathbb{N}_0$, $n\in\mathbb{N}_{0_i}$, and $j=1,\dots,2n+1$ with the parameter 
  \begin{equation*}
    t_n^{(i)} \coloneqq \begin{cases}
      n & i =1,\, 3,\\
      n-1 & i =2, \\
    \end{cases}
  \end{equation*}
  is defined via
  \begin{equation*}
    \tilde{{g}}^{(i)}_{m,n,j}(R;{x}) \coloneqq 
    \sqrt{\frac{4m+2t^{(i)}_n + 3}{R^3}} \left(\frac{r}{R}\right)^{t^{(i)}_n} P_m^{\left(0,t^{(i)}_n + 1/2\right)}\left(2\frac{r^2}{R^2} -1\right) \tilde{{y}}_{n,j}^{(i)}(\xi), \qquad x \in B_R.
  \end{equation*} 
  It is a complete orthonormal system in $\mathrm{L}^2(B_R, \mathbb{R}^3)$. Eventually, each $f \in \mathrm{L}^2(B_R, \mathbb{R}^3)$ has with the abbreviation $f^\wedge(i,m,n,j) \coloneqq \left\langle f, \tilde{{g}}^{(i)}_{m,n,j}(R;\cdot) \right\rangle_{\mathrm{L}^2(B_R, \mathbb{R}^3)}$ the representation
  \begin{equation}\label{eq:FourierSeiresVec}
    f = \sum_{i=1}^3 \sum_{m=0}^\infty \sum_{n=0_i}^\infty \sum_{j=1}^{2n+1} f^\wedge(i,m,n,j) \tilde{{g}}^{(i)}_{m,n,j}(R;\cdot),
  \end{equation}
  which converges unconditionally and strongly in the $\mathrm{L}^2(B_R, \mathbb{R}^3)$-sense.
\end{theorem}
Recall that other choices for the parameter sequence $(t_n^{(i)})_n$ are possible, as long as the condition \newline $\inf_{n\in\mathbb{N}_{0_i}} t_n^{(i)} \geq -\frac{3}{2}$  is satisfied, see \cite{Leweke2020}, also regarding a proof of Thm.~\ref{thm:ONBBall}. 
However, the stated sequence is the natural choice corresponding to our application, \cite[Rem.~5.2]{Leweke2018}. 
\begin{definition}\label{cor:Fnj}
  Let $f\in\mathrm{L}^2(B_R, \mathbb{R}^3)$ be a given function. 
  For almost all $r \in [0,R]$, $i\in\{1, 2, 3\}$, $n\in\mathbb{N}_{0_i}$, and $j=1,\dots,2n+1$ we define
  \begin{equation}\label{eq:GenFourExpans}
    f_{n,j}^{(i)}(r) \coloneqq \int_{ S} f(r\xi) \cdot \tilde{y}_{n,j}^{(i)}(\xi)\, \mathrm{d}\omega(\xi).
  \end{equation}
  An analogue can be obtained for scalar-valued functions $F\in \mathrm{L}^2(B_{R})$  with an appropriate sequence $(t_n)_{n\in\mathbb{N}_0}$ by 
  \begin{equation}\label{eq:FnrGenExpan}
    F_{n,j}(r)  \coloneqq \int_{ S} F(x) Y_{n,j}(\xi)\, \mathrm{d}\omega(\xi). 
  \end{equation}
\end{definition}

\subsection{The Multiple-shell Model for MEG and EEG}

Having the described scalar- and vector-valued orthonormal basis function at hand, we are able to recapitulate the SVD of the integral operators and related relations between the measured quantities (i.e., magnetic flux density and electric potential differences) and the neuronal current $J$.
This SVD is a central point for the construction of the (vector) functional splines for inverse problems.

Before, we state these results, we introduce the common \emph{multiple-shell model} used for modeling the head, \cite{deMunck1993}.
More precisely, in our setting, we assume that  
\begin{itemize}
  \item the cerebrum is a closed ball with radius $\varrho_0$, that is $B_{\varrho_0}$,
  \item surrounding the cerebrum, there are $L \geq 2$ spherical shells $S_{[\varrho_l,\varrho_{l+1}]} \coloneqq \overline{B_{\varrho_{l+1}} \setminus B_{\varrho_{l}}}$ for $l=1,\dots,L-1$ modeling the various head tissues,
  \item each tissue (i.e. each shell $S_{[\varrho_l,\varrho_{l+1}]}$) has a constant conductivity $\sigma_l$ for all $l=0,\dots,L-1$ and outside the head the conductivity is vanishing, $\sigma_L = 0$,
  \item the permeability is constant everywhere and equals the permeability of the vacuum $\mu_0$, 
  \item the relation between the neuronal current and the induced quantities can be modeled by means of quasi-static Maxwell's equations, \cite{Plonsey69}, and
  \item the continuously distributed neuronal current $J \in \mathrm{L}^2(B_{\varrho_0},\mathbb{R}^3)$ is non-vanishing only inside the cerebrum.
\end{itemize}
Note that we do not assume further smoothness or boundary conditions for the neuronal current.

In the MEG case, the functionals $\mathcal{A}_{\mathrm{M}}^{k} \colon \mathrm{L}^2(B_{\varrho_0},\mathbb{R}^3) \to \mathbb{R}$ map the neuronal current onto the magnetic flux density evaluated at the sensor positions $y_k \in \overline{\mathbb{R}^3 \setminus B_{\varrho_L}}$, $k=1,\dots,\ell_M$, outside the head. The flux density is the part of the magnetic field pointing towards the normal direction $\nu$ of the sensor surface. In the EEG case, the functionals $\mathcal{A}_{\mathrm{E}}^{k} \colon \mathrm{L}^2(B_{\varrho_0},\mathbb{R}^3) \to \mathbb{R}$, $k=1,\dots,\ell_{\mathrm{E}}$, map the current onto the electric potential difference measured at several positions on the scalp.

Summarizing the results, we obtain the following equations for functionals, where they are stated in previous publications, \cite[Eq.~(6.5)]{Leweke2020} in the MEG case and in \cite[Proof of Thm.~6.3]{Leweke2020} in the EEG case, respectively:
\begin{align}
   \mathcal{A}_{\mathrm{M}}^{k} J 
    &= \nu(y_k) \cdot B(y_k) \notag \\ 
    &=-\mu_0 \sum_{n=1}^\infty \sum_{j=1}^{2n+1} \sqrt{\frac{n \varrho_0}{(2n+1)(2n+3)}}  \scalar[\mathrm{L}^2(B_{\varrho_0},\mathbb{R}^3)]{J}{\tilde{g}^{(3)}_{0,n,j}(\varrho_0;\cdot)} \left(\frac{\varrho_0}{s_k}\right)^{n+1}\frac{1}{s_k} \nu(y_{k}) \cdot \tilde{y}_{n,j}^{(1)}(\eta_k), \label{eq:funcVector_MEG}\\
  \mathcal{A}_{\mathrm{E}}^k J &= u_s(y_k) \notag\\
  &= \sum_{n=1}^\infty \sum_{j=1}^{2n+1} \sqrt{\frac{n}{\varrho_0}}  \scalar[\mathrm{L}^2(B_{\varrho_0},\mathbb{R}^3)]{J}{\tilde{g}_{0,n,j}^{(2)}(\varrho_0;\cdot)} \left(\frac{(n+1)}{n} \left(\frac{s_k}{\varrho_L}\right)^{2n+1} + 1 \right) \left(\frac{\varrho_0}{s_k}\right)^{n+1} \beta^{(L)}_n Y_{n,j}(\eta_k) \label{eq:funcVector_EEG}.
\end{align}
Recall that $y_k = s_k \eta_k\in \mathbb{R}^3$ for all $k=1,\dots,\ell_{\mathrm{M}}$ or $k=1,\dots,\ell_{\mathrm{E}}$, respectively, denotes the sensor positions of the measurement devices. 
Note that the coefficients $(\beta^{(L)}_n)_n$ depend on the particular model geometry and the tissues conductivities. 
A recursive formula for determining these coefficients and an analysis of the corresponding asymptotic behavior are given in \cite{Leweke2020,Leweke2018}. 
Therein, the absolute and uniform convergence of the series are additionally proven.

Before introducing the corresponding scalar-valued problems, we shortly sum up the information which can be achieved by the SVD-based representation of the functionals. 
In the MEG case, \eqref{eq:funcVector_MEG} reveals that only the direction of the neuronal current which corresponds to the toroidal part (i.e $i=3$) is not silent for the MEG device.
In complementary, in the EEG case, \eqref{eq:funcVector_EEG} shows that only the orthonormal basis functions of type $i=2$ affect the measurements. 
Having the construction of Edmonds vector spherical harmonics in mind and the fact that only type $i=2$ and $i=3$ are divergence-free vector fields, one can deduce that only the solenoidal direction of the neuronal current is related to the two measurements. 
In addition, not the entire solenoidal part of the current can be reconstructed, since only the degree $m=0$ of the generalized Fourier expansions contribute to the data.
This coincides with the harmonic parts of the non-silent directions. 
Hence, the native non-uniqueness condition is to require a harmonic and solenoidal neuronal current. 
This also coincides with the minimum-norm condition, which will serve as our uniqueness constraint for the neuronal current. 
In addition, the severe ill-posedness of the two inverse problems can be seen in the exponentially fast decreasing of the singular values to zero, which mainly goes back to the sequence $((\varrho_0/s_k)^{n+1})_n$.
A more detailed discussion of the non-unique solution, additional uniqueness constraints, and the singular values can be found in \cite{Leweke2018,Leweke2020}.

Besides a decomposition of the neuronal current by means of this vector-valued orthonormal basis, several other decompositions exist with the aim of only considering relevant scalar-valued parts of the neuronal current. For example, the Helmholtz decomposition with the Coulomb gauge of the neuronal current can be used for the MEG as well as the EEG problem, \cite{Dassios2013,Fokas2012,Fokas2012a,Fokas2009}.
Note that for the MEG problem, in \cite{Fokas2012}, the Helmholtz decomposition is combined with a layer density constraint to achieve uniqueness of the solution. 
Unfortunately, this uniqueness assumption contradicts our minimum norm assumption,  \cite[Sec. 20.2]{Leweke2018}. Thus, we need to adapt the approach in order to fit into our setting. 

In general, the Helmholtz decomposition for a sufficiently smooth vector-valued function is given by
\begin{equation}\label{eq:Helmholtz}
  {J} = {\nabla} \Psi + {\nabla} \times {a}
\end{equation}
with the scalar potential $\Psi \in \mathrm{C}^2(B_{\varrho_0})$ and the vector potential $a \in \mathrm{C}^2(B_{\varrho_0}, \mathbb{R}^3)$. Note that this decomposition is not unique without an additional gauge, therefore, we use the Coulomb gauge (i.e. $\nabla \cdot a = 0$). Due to its smoothness, the vector potential $a$ is decomposable further by means of the spherical Helmholtz decomposition, \cite[Eq. (5.58)]{Freeden2009},
\begin{equation*}
    a(x) = \xi A^{(1)}(x) + \nabla^\ast_{\xi} A^{(2)}(x) + L^\ast_{\xi} A^{(3)}(x).
  \end{equation*}
Inserting the decompositions for the neuronal current into the functionals \eqref{eq:funcVector_MEG}, \eqref{eq:funcVector_EEG} and using several orthogonality properties of the orthonormal basis function, \cite{Leweke2018,Leweke2020}, we obtain new functionals derived from the previous ones.  
In the MEG case, we achieve a functional $\mathcal{A}^k_{\mathrm{m}} \colon \mathrm{C}^2(B_{\varrho_0}) \to \mathbb{R}$, $k=1,\dots,\ell_{\mathrm{M}}$, mapping the \emph{scalar-valued part} $A^{(1)}$ of the neuronal current onto the magnetic flux density evaluated at the sensor positions \cite[Lem.~7.3]{Leweke2020}, that is
\begin{equation}\label{eq:ScalarFunctionalMEG}
    \begin{multlined}
    \mathcal{A}^k_{\mathrm{m}} A^{(1)} \coloneqq \mu_0 \nu(y_k) \cdot \left( \int_{B_{\varrho_0}} (\Delta_x (|x| A^{(1)}(x))) \nabla_y K_{\mathrm{m}}(x,y) \,\mathrm{d}x \middle)\right|_{y = y_k} \\ 
    =-\mu_0 \sum_{n=1}^\infty \sum_{j=1}^{2n+1} \sqrt{\frac{1}{(2n+1)(n+1)}} \left(\frac{\mathrm{d} A_{n,j}^{(1)}}{\mathrm{d}r}(\varrho_0)\varrho_0 - (n-1)A_{n,j}^{(1)}(\varrho_0)\right) \left(\frac{\varrho_0}{s_k}\right)^{n+2} \nu(y_k) \cdot \tilde{y}_{n,j}^{(1)}(\eta_k).
    \end{multlined}
\end{equation}
The precise representation of the integral kernel $K_{\mathrm{m}}$ can be found in \eqref{eq:repr_kernel}.

Apart from that, inserting the Helmholtz decomposition into the functionals for the EEG problem, two scalar parts of the neuronal current affect the measured quantity. On the one hand, we have the scalar potential $\Psi$ of the Helmholtz decomposition and on the other hand, we have the scalar part $A^{(3)}$ of the vector potential. The functionals $\mathcal{A}^k_{\mathrm{e}} \colon  \mathrm{C}^2(B_{\varrho_0}) \times \mathrm{C}^2(B_{\varrho_0}) \to \mathbb{R}$, $k=1,\dots,\ell_{\mathrm{e}}$ are given by \cite[Sec.~7.6]{Leweke2020} 
\begin{equation*}
    \mathcal{A}^k_{\mathrm{e}} \left(\Psi,A^{(3)}\right) = \sum_{n=1}^\infty \sum_{j=1}^{2n+1} \left(n \Psi_{n,j}(\varrho_0) - n(n+1) A^{(3)}_{n,j}(\varrho_0)\right) \left(\frac{n+1}{n} \left(\frac{s_k}{\varrho_L}\right)^{2n+1} +1\right) \left(\frac{\varrho_L}{s_k}\right)^{n+1} \beta_n^{(L)}  Y_{n,j}(\eta_k).
\end{equation*}
Note that no gauge is required to derive this representation. 
In the case of the Coulomb gauge, only a relation between the functions $A^{(1)}$ and $A^{(2)}$ can be achieved \cite[Thm.~15.14]{Leweke2018}. 
In the case of the Poincaré gauge (i.e. $x\cdot a= 0$), only conditions for the function $A^{(1)}$ are gained. 
In order to get rid of the function $A^{(3)}$ in the upper relation, an additional boundary condition for the neuronal current, $(x\cdot J)|_{S_{\varrho_0}} = 0$ is often used, to connect the two scalar parts: $0=\frac{\mathrm{d}}{\mathrm{d}r}\Psi(x) + r^{-1}\Delta^{\ast}_{\xi} A^{(3)}(x)$, \cite{Dassios2013,Fokas2009,Fokas2012a}. 
However, it can be proven that every neuronal current $J \in \mathrm{L}^2(B_{\varrho_0},\mathbb{R}^3)$ satisfying this boundary condition in addition to the minimum norm condition, must be equal to the zero function, \cite[Thm. 15.25]{Leweke2020}.
Hence, the Helmholtz decomposition is not suitable for combining the MEG and EEG inversion under the minimum-norm condition. 

\subsection{Introduction to Reproducing Kernel Based Spline Methods for Inverse Problems}

The spline method presented in this paper is based on the construction of reproducing kernels and their generated Hilbert spaces. 
An overview of properties of the scalar as well as the vector-valued spherical interpolating splines based on RKHS on the sphere is summarized in \cite{Freeden1998}.

In the work of \citeauthor{Amirbekyan2008}, the scalar spherical splines are extended to arbitrary RKHS over compact domains and combined with functionals for the application to (ill-posed) inverse problems, see \cite{Amirbekyan2008,Michel2013}. 
These results are also summarized for scalar Hilbert spaces over the ball in the particular context of the inverse MEG and EEG problem in \cite{Fokas2012}.

Before we extend this approach, we briefly summarize the properties and statements relevant for our application in the following. 
More information can be found in the above references and the references therein.
Now, let the system of functions
\begin{equation*}
  H_{n,j}(x) \coloneqq G_n(r)Y_{n,j}(\xi), \qquad n \in \mathbb{N}_0,\, j=1,\dots,2n+1
\end{equation*}
be a linearly independent, orthonormal system in $\mathrm{L}^2(B_R)$. 
This is suitable for isotropic applications like the inverse MEG and EEG problem.

Let $\mathscr{H} \subset \mathrm{L}^2(B_R)$ be an RKHS over the ball $B_R$ generated by the sequence $(\kappa_n)_{n}$ and equipped with the inner product
\begin{equation*}
  \scalar[\mathscr{H}]{F}{G} \coloneqq \sum_{n\in\mathbb{N}_0}\sum_{j=1}^{2n+1} \kappa_n^2 \scalar[{\mathrm{L}^2(B_R)}]{F}{H_{n,j}}\scalar[{\mathrm{L}^2(B_R)}]{G}{H_{n,j}}.
\end{equation*}
This construction goes back to \cite{Freeden1981b}.
The reproducing kernel is uniquely given via Mercers representation theorem \cite{Mercer1909} by 
\begin{equation*}
  K(x, z) = \sum_{\substack{n\in\mathbb{N}_0 \\ \kappa_n \neq 0}} \sum_{j=1}^{2n+1} \kappa_n^{-2} H_{n,j}(x) H_{n,j}(z), \qquad x,\, z \in \mathrm{B}_R.
\end{equation*}
In order to make sense of the latter series expressions, the sequence $(\kappa_n)_{n}$ needs to satisfy a summability condition, that is
\begin{equation*}
  \sum_{\substack{n\in \mathbb{N}_0 \\ \kappa_n \neq 0}} \kappa_n^{-2} \sup_{x \in B_{R}} \abs{\sum_{j=1}^{2n+1}  H_{n,j}(x)}^2 = \sum_{\substack{n\in \mathbb{N}_0 \\ \kappa_n \neq 0}} \frac{2n+1}{4\pi} \kappa_n^{-2} \sup_{r \in [0,R]} \abs{G_{n}(r)}^2 < \infty.
\end{equation*} 
According to \cite[Eq.\ (41)]{Fokas2012}, the corresponding \emph{(scalar) spline function $S$} for the interpolation problem
\begin{equation*}
  g_k = \mathcal{A}^k F, \qquad k=1, \dots, \ell,
\end{equation*}
where each linear functional $\mathcal{A}^k$ for $k=1,\dots,\ell$ maps from $\mathscr{H} $ continuously to $\mathbb{R}$, is of the form
\begin{equation}\label{eq:DefiScalarSpline}
  S \coloneqq \sum_{k=1}^\ell \alpha_k \mathcal{A}^k_{z} K(\cdot, z)= \sum_{k=1}^\ell \alpha_k \sum_{\substack{n\in\mathbb{N}_0 \\ \kappa_n \neq 0}} \sum_{j=1}^{2n+1} \kappa_n^{-2} \left(\mathcal{A}^k H_{n,j}\right) H_{n,j}
\end{equation}
with arbitrary but real coefficients $\alpha = (\alpha_k)_k$.
It is well-known that a scalar spherical spline function has the following properties:
\begin{enumerate}
  \item For given data $g\in\mathbb{R}^\ell$ and continuous linear functionals $\mathcal{A}^k$ for $k=1,\dots,\ell$, the interpolation problem is uniquely solvable if and only if the functionals $\mathcal{A}^k$ are linearly independent, see \cite[Thm.\ 5.10]{Amirbekyan2008}.
  \item Among all solutions $F \in \mathscr{H}$ interpolating the data, that is $\mathcal{A}^k F = g_k$ for $k=1,\dots,\ell$, the spline function is the only solution with minimal $\mathscr{H}$-norm, see \cite[Thm.\ 10.14]{Amirbekyan2008}.
  \item The spline function satisfies a best-approximation property, see \cite[Thm.\ 10.16]{Amirbekyan2008}.
\end{enumerate}
In addition, in the regularized case the following theorem holds true. 
\begin{theorem}[Spline Approximation, {\cite[Thm.\ 10.16]{Michel2013}}]\label{thm:SplineApproxReg}
  Let $g \in\mathbb{R}^\ell$ and a regularization parameter $\lambda > 0$ be given. If the vector $\alpha = (\alpha_k)_{k=1,\dots,\ell} \in \mathbb{R}^\ell$ is the solution of
  \begin{equation}\label{eq:TikReg}
    \left(\left(\mathcal{A}_{x}^l \mathcal{A}_{z}^k \left(K(x, z)\right)\right)_{l,k =1,\dots,\ell} + \lambda \mathcal{I}_{\mathbb{R}^{\ell\times\ell}} \right) \alpha = g,
  \end{equation}
  then the scalar spherical spline function corresponding to the coefficient vector $\alpha$ is the unique minimizer of the corresponding Tikhonov functional, that is
  \begin{equation*}
    S = \argmin_{F \in{\mathscr{H}}} \left(\norm[2]{g - \mathcal{A}F}^2 + \lambda \norm[{\mathscr{H}}]{F}^2\right).
  \end{equation*}
\end{theorem}
Note that $\mathcal{I}_{\mathbb{R}^{\ell\times\ell}}$ stands for the $\ell\times\ell$-identity matrix.

\section{Functional Vector-valued Reproducing Kernel Hilbert Space Splines on the Ball}\label{sec:RKHSball}
In the previous section, we recapitulated that the scalar reproducing kernel based spline approximation has several advantages.
On the one hand, we have the interpolation and best-approximation properties. 
On the other hand, by means of the spline approximation the problem of minimizing the regularized Tikhonov-Philips functional over an infinite-dimensional Hilbert space reduces to solving a finite dimensional system of linear equations without loss of information.
We want to conserve these two outstanding properties for our vector-valued RKHS splines.
In order to do so, we start with the construction of vector-valued Sobolev spaces over the ball, which is based on vector Sobolev spaces on the sphere, \cite{Freeden1993}, and scalar ones on the ball, \cite{Michel2013,Akram2011,Michel2015}.

\subsection{Vector Sobolev Spaces on the Ball}

As in the scalar-valued case, we want to construct vector splines based on reproducing kernels over the ball. 
This method can easily be transferred to arbitrary Hilbert spaces if a complete orthonormal system therein is known. 

Based on \cite[Ch.\ 12.4]{Freeden1998} and the idea of the spherical Helmholtz decomposition, we split the space of all vector-valued arbitrarily often continuously differentiable functions on the ball into three $\scalar[{\mathrm{L}^2(B_R,\mathbb{R}^3)}]{\cdot}{\cdot}$-orthogonal subspaces. 
Due to knowledge of the null spaces of the operators related to the functional inverse MEG and EEG problem, this decomposition reflects the structure of the neuronal current. 
They are based on the spaces
\begin{equation*}  
  \mathrm{C}^{(i),\infty}(B_R,\mathbb{R}^3) \coloneqq \set{f \in \mathrm{C}^{\infty}(B_R,\mathbb{R}^3) \;\middle|\; \tilde{O}^{(\iota)}_{\xi}f(x) = 0\ \text{ if } \{1,2,3\} \ni \iota \neq i}, \qquad i \in \{1,2,3\},
\end{equation*}
where $\tilde{O}^{(i)}$ is the adjoint operator of $\tilde{o}^{(i)}$ (with respect to the $\mathrm{L}^2(B_R,\mathbb{R}^3)$-norm).
However, for the construction of vector-valued splines, it is only necessary that the three directions are $\mathrm{L}^2(B_R,\mathbb{R}^3)$-orthogonal.
Now, let one direction indicated by the superscript $i \in \{1,2,3\}$ be arbitrary but fix. 
For each subspace, we construct a family of orthonormal basis function $\{h_{m,n,j}^{(i)}\}_{m,n,j}$ which can be separated into a radial and an angular part.
Note that for our particular application the orthonormal basis functions are given in Thm.~\ref{thm:ONBBall}. 
In addition, we assume that the estimate
\begin{equation*}
  \sup_{x \in B_R}\sum_{j=1}^{2n+1} \abs{h_{m,n,j}^{(i)}(x)}^2 \eqqcolon B_{m,n}^{(i)} < \infty
\end{equation*}
holds true for every $i\in\{1,2,3\}$, $m \in \mathbb{N}_0$, and $n \in \mathbb{N}_{0_i}$.

\begin{definition}[Vector Sobolev Space]\label{defi:SobolevSpaceBall}
  Let $R>0$ be a given radius, $i\in\{1,2,3\}$ be fixed, and let $a^{(i)} \coloneqq (a_{m,n}^{(i)})_{m,n}$ be a given real sequence. We define a functional $\mathscr{E}^{(i)}\colon \mathrm{C}^{(i),\infty}(B_R,\mathbb{R}^3) \to \mathbb{R}$ by
  \begin{equation*}
    \mathscr{E}^{(i)}(f) \coloneqq \sum_{(m,n) \in\mathbb{N}_0\times\mathbb{N}_{0_i} } \sum_{j=1}^{2n+1} \left(a_{m,n}^{(i)}\right)^2 \scalar[{\mathrm{L}^2(B_R,\mathbb{R}^3)}]{f}{{h}_{m,n,j}^{(i)}}^2
  \end{equation*}
  and, consequently, the space $\mathscr{E}^{(i)}(a^{(i)}, B_R)$ is given by
  \begin{equation*}
    \mathscr{E}^{(i)}(a^{(i)}, B_R) \coloneqq \set{f \in \mathrm{C}^{(i),\infty}(B_R,\mathbb{R}^3) \;\middle|\; \mathscr{E}^{(i)}(f) < \infty \text{ and } \scalar[{\mathrm{L}^2(B_R,\mathbb{R}^3)}]{f}{{h}_{m,n,j}^{(i)}} = 0 \text{ if } a_{m,n}^{(i)} = 0}.
  \end{equation*}
  The space is equipped with the inner product defined for all $f$, $g \in \mathscr{E}^{(i)}(a^{(i)}, B_R)$ by
  \begin{equation*}
    \scalar[\mathscr{H}^{(i)}]{f}{g} \coloneqq \sum_{(m,n) \in\mathbb{N}_0\times\mathbb{N}_{0_i} } \sum_{j=1}^{2n+1} {\left(a_{m,n}^{(i)}\right)}^2 \scalar[{\mathrm{L}^2(B_R,\mathbb{R}^3)}]{f}{h^{(i)}_{m,n,j}}\scalar[{\mathrm{L}^2(B_R,\mathbb{R}^3)}]{g}{h^{(i)}_{m,n,j}}.
  \end{equation*}
  We call the completion of $\mathscr{E}^{(i)}(a^{(i)}, B_R)$ with respect to $\norm[\mathscr{H}^{(i)}]{\cdot}$ the \emph{Sobolev space} $\mathscr{H}^{(i)}(a^{(i)}, B_R)$.
  In addition, we define with $a \coloneqq (a^{(1)}, a^{(2)}, a^{(3)})$ the space
  \begin{equation*}
    \mathscr{H} \coloneqq \mathscr{H}(a, B_R,\mathbb{R}^3) \coloneqq \bigoplus_{i=1}^3\mathscr{H}^{(i)}(a^{(i)}, B_R).
  \end{equation*}
\end{definition}
As in the scalar case, the smoothness of the functions $f \in \mathscr{H}$ is closely related to the behavior of the generating sequence $a$. 
The increasing behavior of the sequence needs to be compensated by a proportionate decay of the generalized Fourier coefficients of the functions.
This has consequences for the smoothness of the functions. 
Accordingly, there exists an inclusion of two different Hilbert spaces. 
\begin{corollary}\label{cor:InclusionSobolev}
  Let $i\in\{1, 2, 3\}$ be arbitrary. Let $a^{(i)} \coloneqq (a_{m,n}^{(i)})_{m,n}$ and $b^{(i)} \coloneqq (b_{m,n}^{(i)})_{m,n}$ be two real sequences with 
  $|a_{m,n}^{(i)}| \leq |b_{m,n}^{(i)}|$ for all $m\in\mathbb{N}_0$ and $n\in\mathbb{N}_{0_i}$. 
  Then
  \begin{equation*}
    \mathscr{H}^{(i)}\left(b^{(i)}, {B_R}\right) \subset \mathscr{H}^{(i)}\left(a^{(i)}, B_R\right).
  \end{equation*}
\end{corollary}

\begin{definition}[Summability]\label{defi:sumSequence}
  Let $i\in\{1, 2, 3\}$ be arbitrary. 
  Let $a^{(i)} \coloneqq (a_{m,n}^{(i)})_{m,n}$ and $b^{(i)} \coloneqq (b_{m,n}^{(i)})_{m,n}$ be two given real sequences. The sequence $a^{(i)}$ is said to be  \emph{$b^{(i)}$-summable}
  if
  \begin{equation*}
    \sum_{\substack{ (m,n) \in\mathbb{N}_0\times \mathbb{N}_{0_i}\\ a_{m,n}^{(i)} \neq 0}} \left(\frac{b_{m,n}^{(i)}}{a_{m,n}^{(i)}}\right)^{2} \left(B_{m,n}^{(i)}\right)^2 < \infty.
  \end{equation*}
  If each element of the sequence $b^{(i)}$ is equal to one, $a^{(i)}$ is said to be \emph{summable}.
\end{definition}
Note that the summability condition is closely related to a bound for the orthonormal basis functions. Thus, if another orthonormal system is used for the construction of the vector Sobolev space, then this summability condition needs to be adapted.

\begin{theorem}\label{thm:SobolevEmbedding}
  Let for each $i\in\{1, 2, 3\}$ the sequence $a^{(i)}$ be $b^{(i)}$-summable, where $b_{m,n}^{(i)} \neq 0$ for all $m\in\mathbb{N}_0$ and $n\in\mathbb{N}_{0_i}$. 
  Then each function  $f \in \mathscr{H}(a/b, B_R,\mathbb{R}^3)$ is also continuous and has a uniformly convergent expansions in the basis $\{h_{m,n,j}^{(i)}\}$. 
\end{theorem}
In this context, the quotient $a^{(i)}/b^{(i)}$ is understood as the element-wise division of the sequences $a^{(i)}$ and $b^{(i)}$, that is $a^{(i)}/b^{(i)} \coloneqq (a_{m,n}^{(i)}/b_{m,n}^{(i)})_{m,n}$.
\begin{proof}[Proof of Thm.~\ref{thm:SobolevEmbedding}]
  Each function $f \in \mathscr{H}(a/b, B_R,\mathbb{R}^3)$ can be represented by a Fourier series converging with respect to the corresponding $\norm[\mathscr{H}]{\cdot}$-norm.
  The next estimate proves that the iterated Fourier series also converges uniformly, due to the Cauchy-Schwarz inequality for series and the bound of the orthonormal basis functions. 
  Thus, the following estimate holds true for each $i\in\{1, 2, 3\}$ and all $x\in B_R$:
  \begingroup
  \allowdisplaybreaks
  \begin{multline*}
    \abs{\sum_{\substack{(m,n)\in\mathbb{N}_0\times\mathbb{N}_{0_i} \\m+n\geq M}}\sum_{j=1}^{2n+1} \scalar[{\mathrm{L}^2(B_R,\mathbb{R}^3)}]{f}{h_{m,n,j}^{(i)}} h_{m,n,j}^{(i)}(x)}^2 \\
    \begin{aligned}
    &\leq \left(\sum_{\substack{(m,n)\in\mathbb{N}_0\times\mathbb{N}_{0_i} \\m+n\geq M}}\sum_{j=1}^{2n+1} \left(\frac{a_{m,n}^{(i)}}{b_{m,n}^{(i)}}\right)^2 \scalar[{\mathrm{L}^2(B_R,\mathbb{R}^3)}]{f}{h_{m,n,j}^{(i)}}^2\right)\left(\sum_{\substack{(m,n)\in\mathbb{N}_0\times\mathbb{N}_{0_i} \\m+n\geq M}}\left(\frac{b_{m,n}^{(i)}}{a_{m,n}^{(i)}}\right)^{2} \sum_{j=1}^{2n+1} \abs{ h_{m,n,j}^{(i)}(x)}^2\right) \\
    &\leq \norm[\mathscr{H}]{f}^2 \left(\sum_{\substack{(m,n)\in\mathbb{N}_0\times\mathbb{N}_{0_i} \\m+n\geq M}}\left(\frac{b_{m,n}^{(i)}}{a_{m,n}^{(i)}}\right)^{2} \left(B_{m,n}^{(i)}\right)^2 \right) 
    \end{aligned}
  \end{multline*}%
  \endgroup
  The right-hand side converges to zero as $M\to \infty$ due to the summability condition. 
  Hence, this iterated series converges uniformly and each summand is continuous in $B_R$ as assumed. 
  This estimate also implies the absolute convergence of the iterated series. 
  Due to Cauchy's theorem of double series, the absolute convergence of the iterated series suffices for the absolute convergence of the double series and their limits coincide.
\end{proof}

\subsection{Vector Splines on the Ball}

For the construction of vector-valued reproducing kernel based splines, we first need to construct appropriate reproducing kernels.
For this method, they are constructed via the tensor product of vector-valued orthonormal basis functions. This approach has already been used in \cite{Freeden1998,Freeden1993} for vector-valued splines on the unit sphere. 
Thus, the reproducing kernel $\mathfrak{k}^{(i)} \colon B_{R} \times B_{R} \to \mathbb{R}^{3 \times 3}$ considered in this section depends on $i\in\{1, 2, 3\}$ and is defined by
\begin{equation}\label{eq:TensRepKer}
  \mathfrak{k}^{(i)}(x, y) \coloneqq \sum_{\substack{(m,n)\in\mathbb{N}_0\times\mathbb{N}_{0_i} \\ \kappa_{m,n}^{(i)} \neq 0}} \sum_{j=1}^{2n+1} \left(\kappa_{m,n}^{(i)}\right)^{-2} h_{m,n,j}^{(i)}(x) \otimes h_{m,n,j}^{(i)}(y), \qquad x,\, y \in B_R.
\end{equation}
Reproducing kernels of this kind are also introduced in \cite{Bayer2000,Beth2000} and used for the construction of interpolating vector splines and vector-valued wavelets on the sphere. 
The Sobolev space $\mathscr{H}^{(i)}((\kappa_{m,n}^{(i)})_{m,n}, B_R, \mathbb{R}^3)$ is the natural Sobolev space containing the reproducing kernel with argument fixed, that is $\mathfrak{k}^{(i)}(x, \cdot) \in \mathscr{H}^{(i)}((\kappa_{m,n}^{(i)})_{m,n}, B_R, \mathbb{R}^3)$ and $\mathfrak{k}^{(i)}(\cdot, y) \in \mathscr{H}^{(i)}((\kappa_{m,n}^{(i)})_{m,n}, B_R, \mathbb{R}^3)$ for all $x,\ y \in B_R$.
\begin{lemma}
	Let $i\in\{1,2,3\}$ be fixed. If the sequence $(\kappa_{m,n}^{(i)})_{m,n}$ satisfies the summability condition, Def.~\ref{defi:sumSequence}, the series in \eqref{eq:TensRepKer} converges pointwise with respect to the $2$-norm of matrices.
\end{lemma}
\begin{proof}
	The $2$-norm of rank-one outer products of two vectors $u$, $v\in \mathbb{R}^3$ fulfills the estimate $|u \otimes v| \coloneqq \| u \otimes v \|_2 \leq \| u \|_2 \|v \|_2$.
	Translated in our notation we achieve for the tensor product of the orthonormal basis functions
	\begin{equation*}
		\left|h_{m,n,j}^{(i)}(x) \otimes h_{m,n,j}^{(i)}(y)\right| \leq \left|h_{m,n,j}^{(i)}(x)\right|\left|h_{m,n,j}^{(i)}(y)\right|  \qquad \text{ for all } x,\ y \in B_R . 
	\end{equation*}
	Hence, for the reproducing kernel we have a convergent majorant which is given by
	\begin{align*}
	 |\mathfrak{k}^{(i)}(x, y)| 
    &\leq \sum_{\substack{(m,n)\in\mathbb{N}_0\times\mathbb{N}_{0_i} \\ \kappa_{m,n}^{(i)} \neq 0}} \left(\kappa_{m,n}^{(i)}\right)^{-2}  \sum_{j=1}^{2n+1} |h_{m,n,j}^{(i)}(x)||h_{m,n,j}^{(i)}(y)| \\
    &\leq \sum_{\substack{(m,n)\in\mathbb{N}_0\times\mathbb{N}_{0_i} \\ \kappa_{m,n}^{(i)} \neq 0}} \left(\kappa_{m,n}^{(i)}\right)^{-2}  \left(\sum_{j=1}^{2n+1} |h_{m,n,j}^{(i)}(x)|^2\right)^{1/2}\left(\sum_{j=1}^{2n+1} |h_{m,n,j}^{(i)}(y)|^2\right)^{1/2} \\
    &\leq \sum_{\substack{(m,n)\in\mathbb{N}_0\times\mathbb{N}_{0_i} \\ \kappa_{m,n}^{(i)} \neq 0}} \left(\kappa_{m,n}^{(i)}\right)^{-2} \sup_{x\in B_R} \sum_{j=1}^{2n+1} |h_{m,n,j}^{(i)}(x)|^2 
    \leq \sum_{\substack{(m,n)\in\mathbb{N}_0\times\mathbb{N}_{0_i} \\ \kappa_{m,n}^{(i)} \neq 0}}  \left(\kappa_{m,n}^{(i)}\right)^{-2} \left(B_{m,n}^{(i)}\right)^2 < \infty. \tag*{\qedhere}
  \end{align*}
\end{proof}
For more details on vector-valued reproducing kernels, such as the reproducing property, we refer to \cite{Freeden1993} since the conversion from the spherical case to the ball case is straightforward.
Furthermore, we assume that the sequence $(\kappa_{m,n}^{(i)})_{m,n}$ is given in such a way that $\mathscr{H}^{(i)} \coloneqq \mathscr{H}^{(i)}((\kappa_{m,n}^{(i)})_{m,n}, B_R) \subset \mathrm{L}^2(B_{R},\mathbb{R}^3)$.

The spline constructed in \cite{Freeden1993,Bayer2000} is used for an interpolation problem.
In contrast, we want to use the spline for approximating the solution of a (functional) inverse problem. 
Having the vector-valued reproducing kernels at hand, we develop in the following a novel approach to vector-valued splines for inverse problems.

For the construction of approximation splines for functional inverse problems, the data $g = (g_1, \dots, g_\ell) \in\mathbb{R}^\ell$ needs to be of the form
\begin{equation*}
  g = \mathcal{A} f \qquad \Leftrightarrow \qquad  g_k = \mathcal{A}^k f, \qquad k=1, \dots, \ell
\end{equation*}
with the linear and continuous operator $\mathcal{A} \coloneqq (\mathcal{A}^1, \dots, \mathcal{A}^\ell)^{\mathrm{T}}$, the functionals $\mathcal{A}^k \colon \mathscr{H}^{(i)} \to \mathbb{R}$ for all $k=1, \dots, \ell$, and the (sought) quantity $f \in \mathscr{H}^{(i)}$. 
In \cite[Ch.\ 6.4., Ch.\ 10]{Michel2013}, this method is extended to a scalar spline approximation on the ball.
Thus, we call a function of the form 
\begin{equation}\label{eq:VectorialSpline}
  s(x)= \sum_{k=1}^{\ell} \alpha_k \mathcal{A}_{z}^k \left(\mathfrak{k}^{(i)}(z, x)\right), \qquad x \in B_{R}
\end{equation}
with the coefficients $\alpha = ( \alpha_k)_{k=1,\dots,\ell} \in \mathbb{R}^\ell$ a \emph{spline function in $\mathscr{H}^{(i)}$ subject to $\mathcal{A}$}. 
The set of all these spline functions is denoted by $\mathrm{Spline}((\kappa_{m,n}^{(i)})_{m,n}, \mathcal{A})$ which is an $\ell$-dimensional space. 
For this definition to make sense, we need to define what we understand by applying the functional $\mathcal{A}^k$ to a tensor product of two vector-valued functions. 
For a tensorial function $\mathfrak{t}(z,x) = \sum_{i,j=1}^3 \mathfrak{t}_{i,j}(z,x) \varepsilon^i \otimes \varepsilon^j$ with scalar component functions $\mathfrak{t}_{i,j}$, we set
\begin{equation}\label{eq:Func2Tensor}
  \mathcal{A}_z^k \mathfrak{t}(z,x) \coloneqq \sum_{j=1}^3 \mathcal{A}_z^k \left(\sum_{i=1}^3 \mathfrak{t}_{i,j}(z,x) \varepsilon^i \right) \varepsilon^j,
\end{equation}
where $\{\varepsilon^i\}_{i=1,2,3}$ is the standard basis of $\mathbb{R}^3$.
Eventually, the vector-valued spline function has the representation 
\begin{equation}\label{eq:defiSpline}
  s(x) = \sum_{k=1}^{\ell} \alpha_k \sum_{\substack{(m,n)\in\mathbb{N}_0\times\mathbb{N}_{0_i} \\ \kappa_{m,n}^{(i)} \neq 0}} \sum_{j=1}^{2n+1} \left(\kappa_{m,n}^{(i)}\right)^{-2} h_{m,n,j}^{(i)}(x)  \mathcal{A}^k h_{m,n,j}^{(i)}, \qquad x \in B_{R}.
\end{equation}
The convergence of the stated series is implied by the summability condition, the convergence of the reproducing kernel series, and the linearity and continuity of the functionals.

Several useful properties of (scalar) splines (over the ball) have already been known. For example, in \cite[Thm.\ 10.13-14]{Michel2013}, two minimum properties of scalar splines over the ball are proved. These statements also hold true in the vector-valued case, which is proved for the particular setting of $\mathscr{H}^{(i)}$ being the product space of two scalar Sobolev spaces in \cite{Berkel2010a,Berkel2009} using tensor-valued reproducing kernels. 
Now we adapt these statements to our setting. 
For this purpose, three central properties need to be verified.
\begin{lemma}\label{lem:VecSpline}
  Let the reproducing kernel be given as in \eqref{eq:TensRepKer}, then
  \begin{enumerate}
    \item a kind of reproducing property holds true for all $k=1,\dots,\ell$, that is
    \begin{equation*}
       \scalar[\mathscr{H}^{(i)}]{\mathcal{A}_{x}^k \left(\mathfrak{k}^{(i)}(x, \cdot)\right)}{f} = \mathcal{A}^k f,
    \end{equation*}
    for all $f \in \mathscr{H}^{(i)}$,
    \item for all $k,\, l=1,\dots,\ell$ the following relation holds true:
    \begin{equation*}
      \mathcal{A}^l_{x} \mathcal{A}^k_{z} \mathfrak{k}^{(i)}(x, z) = \scalar[\mathscr{H}^{(i)}]{\mathcal{A}^k_{x} \mathfrak{k}^{(i)}(\cdot,x)}{\mathcal{A}^k_{z} \mathfrak{k}^{(i)}(\cdot,z)},
    \end{equation*}
    and
    \item every spline function $s \in \mathrm{Spline}((\kappa_{m,n}^{(i)})_{m,n}, \mathcal{A})$ satisfies the relation
    \begin{equation}\label{eq:ReproPropSpline}
      \scalar[\mathscr{H}^{(i)}]{s}{f} = \sum_{k=1}^{\ell} \alpha_k \mathcal{A}^k f
    \end{equation}
    for all $f \in \mathscr{H}^{(i)}$.
  \end{enumerate}
\end{lemma}
\begin{proof}
  We start with the proof of the first item. 
  Then, we obtain with the definition of the $\mathscr{H}^{(i)}$-inner product, \eqref{eq:Func2Tensor},  \eqref{eq:defiSpline}, and Thm.~\ref{thm:SobolevEmbedding} as well as the linearity and continuity of the $\mathcal{A}^k$
  \begin{equation}
    \scalar[\mathscr{H}^{(i)}]{\mathcal{A}_{x}^k \left(\mathfrak{k}^{(i)}(x, \cdot)\right)}{f}
    \begin{aligned}[t]
        &= \sum_{\substack{(m,n)\in\mathbb{N}_0\times\mathbb{N}_{0_i} \\ \kappa_{m,n}^{(i)} \neq 0}} \sum_{j=1}^{2n+1} \left(\kappa_{m,n}^{(i)}\right)^{2}\left(\kappa_{m,n}^{(i)}\right)^{-2} \mathcal{A}^k \left(h_{m,n,j}^{(i)}\right) \scalar[\mathrm{L}^2(B_{R},\mathbb{R}^3)]{f}{h_{m,n,j}^{(i)} } \\
        &= \mathcal{A}^k \left(\sum_{\substack{(m,n)\in\mathbb{N}_0\times\mathbb{N}_{0_i} \\ \kappa_{m,n}^{(i)} \neq 0}} \sum_{j=1}^{2n+1} \scalar[\mathrm{L}^2(B_{R},\mathbb{R}^3)]{f}{h_{m,n,j}^{(i)} } h_{m,n,j}^{(i)}\right) \\
        &= \mathcal{A}^k f%
    \end{aligned}
  \end{equation}
  The last step is valid since $\mathscr{H}^{(i)} \subset \mathrm{L}^2(B_{R},\mathbb{R}^3)$ and, hence, $f$ can be represented by the Fourier series. 
  The second item is a particular case of the first with $f = \mathcal{A}^k_{z} \mathfrak{k}^{(i)}(\cdot, z)$.

  For the proof of the last statement, we only need to use the representation of the vector-valued spline from \eqref{eq:VectorialSpline}, the linearity of the inner-product, and the first item of this lemma. 
\end{proof}
  
With this preliminary work, the ideas of the original proofs, which can for instance be found in \cite{Michel2013}, of the next statements are still valid in the vector-valued case. 
\begin{theorem}\label{thm:VecSplineInterpolationProbl}
  Let $g \in\mathbb{R}^\ell$ be the given data and the spline function $s \in \mathscr{H}^{(i)}$ be unknown. Then the spline interpolation problem $\mathcal{A}^k s = g_k$ for all $k=1,\dots, \ell$ is uniquely solvable if and only if the functionals $\set[k=1,\dots,\ell]{\mathcal{A}^k}$ are linearly independent. 
\end{theorem}
\begin{proof}
  Via \eqref{eq:VectorialSpline}, we see that the interpolation problem is equivalent to solving a linear system of equations with the given data as the right-hand side and the matrix
  \begin{equation*}
    \left(\mathcal{A}_{x}^l \mathcal{A}_{z}^k \left(\mathfrak{k}^{(i)}(z, x)\right)\right)_{l,k =1,\dots,\ell},
  \end{equation*}
  which is uniquely solvable if and only if the matrix is regular. Via item (2) of Lem.~\ref{lem:VecSpline}, we get that this matrix is a Gramian matrix. Item (1) of this lemma provides us with the property that $\set[k=1,\dots,\ell]{\mathcal{A}^k}$ is linearly independent if and only if $\set[k=1,\dots,\ell]{\mathcal{A}_{x}^k \mathfrak{k}^{(i)}(x, \cdot)}$ is linearly independent.
\end{proof}
Again with \eqref{eq:ReproPropSpline}, the proofs of the following theorems are immediate consequences of the ones in {\cite[Thm.\ 10.13-10.16]{Michel2013}} and therefore skipped for the sake of brevity.

\begin{theorem}[Minimum Properties]\label{thm:MNP}
  Let $\mathscr{H}^{(i)} \subset \mathrm{L}^2(B_{R},\mathbb{R}^3)$ be a given Sobolev space and $\mathcal{A}^k \colon \mathscr{H}^{(i)} \to \mathbb{R}$ be bounded linear functionals for all $k=1,\dots,\ell$ that are linearly independent. Then the following properties hold true:
  \begin{enumerate}
    \item If $g  \in\mathbb{R}^\ell$ is a given vector and the spline $s$ is given by $\mathcal{A}^k s = g_k$ for all $k=1,\dots,\ell$, then $s$ is the unique minimizer of 
    \begin{equation*}
      \norm[{\mathscr{H}^{(i)}}]{s} = \min \set{ \norm[{\mathscr{H}^{(i)}}]{f} \;\middle|\; f \in \mathscr{H}^{(i)} \text{ with } \mathcal{A}^k f = g_k \text{ for all } k=1,\dots,\ell}.
    \end{equation*}
    \item If $f \in \mathscr{H}^{(i)}$ is a given function and the spline $s$ is defined by $\mathcal{A}^k s = \mathcal{A}^k f$ for all $k=1,\dots,\ell$, then $s$ is the unique minimizer of
    \begin{equation*}
      \norm[{\mathscr{H}^{(i)}}]{f- s} = \min \set{\norm[{\mathscr{H}^{(i)}}]{f - \bar{s}} \;\middle|\; \bar{s} \in \mathrm{Spline}\left((\kappa_{m,n}^{(i)})_{m,n}, \mathcal{A}\right) }.
    \end{equation*}
  \end{enumerate}
\end{theorem}

\begin{theorem}[Shannon Sampling Theorem]
  Any spline function $s \in \mathrm{Spline}\left((\kappa_{m,n}^{(i)})_{m,n}, \mathcal{A}\right)$ is representable by its samples $\mathcal{A}^k s$ as
  \begin{equation*}
    s(x) = \sum_{k=1}^\ell (\mathcal{A}^k s) L_k(x), \qquad x \in B_R,
  \end{equation*}
  where
  \begin{equation*}
    L_k(x) = \sum_{j=1}^\ell \alpha_j^{(k)} \mathcal{A}^j \mathfrak{k}^{(i)}(\cdot, x), \qquad x \in B_R,
  \end{equation*}
  with $\alpha_j^{(k)}$ given as solutions of the linear equation systems
  \begin{equation*}
    \sum_{j=1}^\ell \alpha_j^{(k)} \mathcal{A}^l \mathcal{A}^j \mathfrak{k}^{(i)}(\cdot, \cdot) = \delta_{k,l} \qquad \text{ for all } k,\, l = 1,\dots,N.
  \end{equation*}
\end{theorem} 

\begin{theorem}[Spline Approximation]\label{thm:RegSolSplineVec}
  Let $g \in\mathbb{R}^\ell$ and a regularization parameter $\lambda > 0$ be given. Let the bounded linear functionals $\mathcal{A}^k \colon \mathscr{H}^{(i)} \to \mathbb{R}$, with $k=1,\dots,\ell$, be linearly independent. If the vector $\alpha = ( \alpha_k)_{k=1,\dots,\ell}$ is the solution of
  \begin{equation}\label{eq:SplineMatrixVec}
    \left(\left(\mathcal{A}_{x}^l \mathcal{A}_{z}^k \left(\mathfrak{k}^{(i)}(z, x)\right)\right)_{l,k =1,\dots,\ell} + \lambda \mathcal{I}_{\mathbb{R}^{\ell\times\ell}} \right) \alpha = g,
  \end{equation}
  then the spline function given by \eqref{eq:VectorialSpline} is the unique minimizer of the corresponding Tikhonov-functional, that is
  \begin{equation*}
    s = \argmin_{f \in\mathscr{H}^{(i)}} \left(\norm[\mathbb{R}^\ell]{g - \mathcal{A}f}^2 + \lambda \norm[\mathscr{H}^{(i)}]{f}^2\right).
  \end{equation*}
\end{theorem}
\begin{proof}
  Let $s \in \mathrm{Spline}\left((\kappa_{m,n}^{(i)})_{m,n}, \mathcal{A}\right)$ be the unique solution of \eqref{eq:SplineMatrixVec} which exists since the occurring matrix is positive definite, see Lem.~\ref{lem:VecSpline}.
  We first prove the auxiliary statement
  \begin{equation*}
    \scalar[\mathscr{H}^{(i)}]{\mathcal{A}^\ast g - \mathcal{A}^\ast\mathcal{A}s- \lambda s }{s - f} = 0
  \end{equation*}
  for all $f \in \mathscr{H}^{(i)}$. 
  For this purpose, we notice by means of the spline definition, \eqref{eq:VectorialSpline}, and the canonical matrix-vector product that
  \begin{equation*}
    \mathcal{A}_x s(x) = \left(\mathcal{A}_{x}^l \mathcal{A}_{z}^k \left(\mathfrak{k}^{(i)}(z, x)\right)\right)_{l,k =1,\dots,\ell} \alpha.
  \end{equation*}
  In addition, $\mathcal{A}^\ast$ can be applied to \eqref{eq:SplineMatrixVec} and the resulting equation can be  equivalently reformulated to
  \begin{equation*}
    \mathcal{A}^\ast \mathcal{A} s + \lambda \mathcal{A}^\ast \alpha = \mathcal{A}^\ast g.
  \end{equation*}
  Thus,
  \begin{equation*}
    \mathcal{A}^\ast g - \mathcal{A}^\ast\mathcal{A}s- \lambda s = \mathcal{A}^\ast \mathcal{A} s + \lambda \mathcal{A}^\ast \alpha - \mathcal{A}^\ast\mathcal{A}s- \lambda s = \lambda \left(\mathcal{A}^\ast \alpha - s \right).
  \end{equation*}
  It remains to prove that the right-hand side of the last equation vanishes. For all $f \in \mathscr{H}^{(i)}$ we have with item (3) of Lem.~\ref{lem:VecSpline} the relation
  \begin{equation*}
    \scalar[\mathscr{H}^{(i)}]{\mathcal{A}^\ast \alpha - s}{f} = \scalar[\mathbb{R}^\ell]{\alpha}{\mathcal{A} f}  - \scalar[\mathscr{H}^{(i)}]{s}{f} = \scalar[\mathbb{R}^\ell]{\alpha}{\mathcal{A} f}  - \alpha \cdot \mathcal{A}f = 0. 
  \end{equation*}
  For the regularized Tikhonov-functional, we obtain for all $f \in \mathscr{H}^{(i)}$ the estimate
  \begin{align*}
    &\phantom{=\ }\norm[\mathbb{R}^\ell]{g - \mathcal{A}f}^2 + \lambda \norm[\mathscr{H}^{(i)}]{f}^2 \\
    &= \norm[\mathbb{R}^\ell]{g - \mathcal{A}s + \mathcal{A}s - \mathcal{A}f}^2 + \lambda \norm[\mathscr{H}^{(i)}]{f-s + s}^2 \\
    &= \norm[\mathbb{R}^\ell]{g - \mathcal{A}s}^2 + \norm[\mathbb{R}^\ell]{\mathcal{A}s - \mathcal{A}f}^2 + 2 \scalar[\mathbb{R}^\ell]{g - \mathcal{A}s}{\mathcal{A}s - \mathcal{A}f} + \lambda \left(\norm[\mathscr{H}^{(i)}]{f-s}^2 +2\scalar[\mathscr{H}^{(i)}]{f-s}{s} +  \norm[\mathscr{H}^{(i)}]{s}^2\right) \\
    & \geq \norm[\mathbb{R}^\ell]{g - \mathcal{A}s}^2 + \lambda \norm[\mathscr{H}^{(i)}]{s}^2 +  2 \scalar[\mathbb{R}^\ell]{g - \mathcal{A}s}{\mathcal{A}s - \mathcal{A}f} + 2 \lambda \scalar[\mathscr{H}^{(i)}]{f-s}{s} \\
    &= \norm[\mathbb{R}^\ell]{g - \mathcal{A}s}^2 + \lambda \norm[\mathscr{H}^{(i)}]{s}^2 +  2 \left(\scalar[\mathscr{H}^{(i)}]{\mathcal{A}^\ast(g - \mathcal{A}s)}{s - f} - \scalar[\mathscr{H}^{(i)}]{\lambda s}{s-f}\right) \\
    &= \norm[\mathbb{R}^\ell]{g - \mathcal{A}s}^2 + \lambda \norm[\mathscr{H}^{(i)}]{s}^2 +  2 \scalar[\mathscr{H}^{(i)}]{\mathcal{A}^\ast g - \mathcal{A}^\ast\mathcal{A}s- \lambda s }{s - f} \\
    &= \norm[\mathbb{R}^\ell]{g - \mathcal{A}s}^2 + \lambda \norm[\mathscr{H}^{(i)}]{s}^2 \tag*{\qedhere}.
  \end{align*}
  Note that equality only holds if $s = f$.
\end{proof}

\begin{theorem}[Convergence Result]
  Let $\set{\mathcal{A}^k \mid k \in \mathbb{N}}$ be a countable, infinite, and linearly independent system of linear and continuous functionals. 
  For a fixed $f \in \mathscr{H}^{(i)}$ and every $N \in \mathbb{N}$, let the spline function $s_N$ be given by
  \begin{equation*}
    s_N \in \mathrm{Spline}\left(\left(\kappa_{m,n}^{(i)}\right)_{m,n}, \left(\mathcal{A}^1,\dots, \mathcal{A}^N\right)\right) \qquad \text{ with } \mathcal{A}^k s_N = \mathcal{A}^k f \quad \text{ for all } k=1,\dots, N.
  \end{equation*}
  If $\operatorname{span}\set{\mathcal{A}^k \mid k \in \mathbb{N}}$ is dense in the dual space $(\mathscr{H}^{(i)})^\ast$ of $\mathscr{H}^{(i)}$ then
  \begin{equation*}
    \lim_{N \to \infty} \norm[{\mathscr{H}^{(i)}}]{f - s_N} = 0.
  \end{equation*}
\end{theorem}
\begin{proof}
  We prove the convergence result in two steps. First, we prove the weak convergence of the splines. 
  Second, the convergence of $\norm[{\mathscr{H}^{(i)}}]{s_N}$ to $\norm[{\mathscr{H}^{(i)}}]{f}$ is proven. 
  Combined we obtain the desired result.

  For the weak convergence we choose an arbitrary element $\mathcal{T} \in (\mathscr{H}^{(i)})^\ast$. 
  Since $\operatorname{span}\set{\mathcal{A}^k \mid k \in \mathbb{N}}$ is dense in the dual space, we find for all $\varepsilon > 0$ an $\tilde{N}\in \mathbb{N}$ and coefficients $(\alpha_k)_k$ such that 
  \begin{equation*}
    \tilde{\mathcal{T}} \coloneqq \sum_{k=1}^{\tilde{N}} \alpha_k \mathcal{A}^k
  \end{equation*}
  satisfies $\norm[{(\mathscr{H}^{(i)})^\ast}]{\mathcal{T} - \tilde{\mathcal{T}}} \leq \varepsilon$.
  In addition, we obtain with the interpolation condition of the spline function for all $N \geq \tilde{N}$ the relation
  \begin{equation*}
    \tilde{\mathcal{T}}f = \sum_{k=1}^{\tilde{N}} \alpha_k \mathcal{A}^k f = \sum_{k=1}^{\tilde{N}} \alpha_k \mathcal{A}^k s_{N} = \tilde{\mathcal{T}}s_{N}.
  \end{equation*}
  This yields together with the first minimum norm property of Thm.~\ref{thm:MNP}
  \begin{align*}
    \abs{\mathcal{T} f - \mathcal{T}s_{N}} &= \abs{\mathcal{T} f - \tilde{\mathcal{T}} f + \tilde{\mathcal{T}}{s_{N}} - \mathcal{T}s_{N}} \\
    &\leq \norm[{(\mathscr{H}^{(i)})^\ast}]{\mathcal{T}-\tilde{\mathcal{T}}}(\norm[{\mathscr{H}^{(i)}}]{f} + \norm[{\mathscr{H}^{(i)}}]{s_{N}}) \\
    &\leq 2 \varepsilon \norm[{\mathscr{H}^{(i)}}]{f},
  \end{align*}
  which implies the weak convergence.

  For the convergence of the norms, we use the representation, \cite[p. 91]{Yosida1988},
  \begin{equation*}
    \norm[\mathscr{H}^{(i)}]{f} = \sup_{\substack{\mathcal{T} \in \operatorname{span}\set{\mathcal{A}^k \mid k \in \mathbb{N}} \\ \norm[(\mathscr{H}^{(i)})^\ast]{\mathcal{T}} \leq 1}} \abs{\mathcal{T} f}.
  \end{equation*}
  Thus, for all $\varepsilon > 0$, we can find an $\tilde{N} \in \mathbb{N}$ and coefficients $(\alpha_k)_k$ such that the operator 
  \begin{equation*}
    \tilde{\mathcal{T}} \coloneqq \sum_{k=1}^{\tilde{N}} \alpha_k \mathcal{A}^k \quad \text{ with } \quad \norm[(\mathscr{H}^{(i)})^\ast]{\tilde{\mathcal{T}}} \leq 1
  \end{equation*}
  satisfies
  \begin{equation*}
    \norm[\mathscr{H}^{(i)}]{f} \leq \abs{\tilde{\mathcal{T}}f} + \varepsilon.
  \end{equation*}
  For all $N \geq \tilde{N}$ we obtain again with the first minimum norm property the estimate
  \begin{equation*}
    \norm[\mathscr{H}^{(i)}]{f} - \varepsilon \leq \abs{\tilde{\mathcal{T}}f} = \abs{\tilde{\mathcal{T}}s_N} \leq \norm[{(\mathscr{H}^{(i)})^\ast}]{\tilde{\mathcal{T}}} \norm[\mathscr{H}^{(i)}]{s_N} \leq \norm[\mathscr{H}^{(i)}]{s_N} \leq \norm[\mathscr{H}^{(i)}]{f},
  \end{equation*}
  which implies the convergence of the norms.
\end{proof}
Now, we have the theoretical foundations at hand in order to solve the vector-valued inverse MEG and EEG problems by means of RKHS splines. 

\section{Foundations for Implementation}\label{sec:Foundations}
\subsection{Scalar Splines for the Spherical MEG Problem} \label{sec:ScalarSpline}

In \cite{Fokas2012}, the presented scalar spline method is used for solving the scalar inverse MEG problem. 
However, therein a layer density constraint is used to achieve uniqueness of the inverse problem. 
Now, we want to use the minimum norm condition instead. 
For this purpose, the method needs to be adapted. 
Recall that the minimum norm condition is meant with respect to the entire neuronal current and not with respect to a scalar-valued part of it. 
Since the focus of this paper lies on the novel vector-valued spline method, we only give a short summary of the formulae required for the scalar spline implementation.
We use the following set of functions for the radial part of our orthonormal system
\begin{equation}\label{eq:GnBasis}
  G_n(r) \coloneqq \sqrt{\frac{2n+5}{\varrho_0^{3}}} \left(\frac{r}{\varrho_0}\right)^{n+1}, \qquad r \in [0, \varrho_0],\, n \in \mathbb{N}.
\end{equation}
Note that these functions fulfill the relation
\begin{equation*}
  G_n^\prime(\varrho_0)\varrho_0 - (n-1)G_n(\varrho_0) = 2\sqrt{\frac{2n+5}{\varrho_0^3}}, \qquad n \in \mathbb{N}.
\end{equation*}
In addition, due to the null space of our operator we choose $\kappa_0 \coloneqq 0$.
Eventually, we can conclude that the linear functionals stated in \eqref{eq:ScalarFunctionalMEG} applied to the corresponding reproducing kernel yield
\begin{equation}\label{eq:MEGscalarFunc2Kern}
  \mathcal{A}^k_{\mathrm{m}} K(x, \cdot) = \frac{-\mu_0}{2\pi\varrho_0^{2}}\nu(y_k)\cdot \left(\sum_{\substack{n\in\mathbb{N}\\\kappa_n \neq 0}} \sqrt{\frac{(2n+1)(2n+5)^2}{n+1}} \kappa_n^{-2} \frac{r^{n+1}}{s_k^{n+2}}  \tilde{p}_{n}^{(1)}(\eta_k,\xi)\right).
\end{equation}
Each entry of the spline matrix is given for all $l$, $k=1,\dots,{\ell_{\mathrm{M}}}$ by
\begin{equation}\label{eq:scalarSplineMatrixEntry}
  \mathcal{A}^l_{\mathrm{m},x} \mathcal{A}^k_{\mathrm{m},z} K(x, z) = \frac{\mu_0^2}{\varrho_0^3} \sum_{\substack{n\in\mathbb{N}\\\kappa_n \neq 0}} \sum_{j=1}^{2n+1} \frac{4(2n+5)}{{(n+1)(2n+1)}\kappa_n^{2}} \left(\frac{\varrho_0^2}{s_ks_l}\right)^{n+2} \left(\nu(y_k)\cdot\tilde{y}_{n,j}^{(1)}(\eta_k)\right) \left(\nu(y_l)\cdot \tilde{y}_{n,j}^{(1)}(\eta_l)\right). 
\end{equation}
An alternative representation, which is more suitable for implementation, is stated in \eqref{eq:scalarSplineMatrixEntryImpl}.

\subsection{Vector Splines for the Spherical MEG and EEG problem}

For the implementation of the vector spline method for the spherical MEG and EEG problem, we have to determine the occurring functionals and reproducing kernels with their corresponding vector Sobolev spaces. 

We do not want to go into details with the precise vector-valued integral representation of the forward operator and consider the SVD-based series representation instead, derived in \cite{Leweke2018,Leweke2020}.
Therein, it is verified that the series representations converge absolutely and uniformly on the particular domains. That is
\begin{align}
  \mathcal{A}_{\mathrm{M}}^{k} J 
    &= -\mu_0 \sum_{n=1}^\infty \sum_{j=1}^{2n+1} \sqrt{\frac{n \varrho_0}{(2n+1)(2n+3)}}  \scalar[\mathrm{L}^2(B_{\varrho_0},\mathbb{R}^3)]{J}{\tilde{g}^{(3)}_{0,n,j}(\varrho_0;\cdot)} \left(\frac{\varrho_0}{s_k}\right)^{n+1}\frac{1}{s_k} \nu(y_{k}) \cdot \tilde{y}_{n,j}^{(1)}(\eta_k), \label{eq:ForwardRMFPOp}\\
  \mathcal{A}_{\mathrm{E}}^k J &= \sum_{n=1}^\infty \sum_{j=1}^{2n+1} \frac{1}{\sqrt{n\varrho_0}} \beta^{(L)}_n \scalar[\mathrm{L}^2(B_{\varrho_0},\mathbb{R}^3)]{J}{\tilde{g}_{0,n,j}^{(i)}(\varrho_0;\cdot)} \left({(n+1)} \left(\frac{s_k}{\varrho_L}\right)^{2n+1} + {n} \right) \left(\frac{\varrho_0}{s_k}\right)^{n+1} Y_{n,j}(\eta_k), \label{eq:ForwardRMFPOp2}
\end{align}
where $y_k = s_k\eta_k$ for all $k=1,\dots,\ell_{\mathrm{M}}$ or $k=1,\dots,\ell_{\mathrm{E}}$, respectively, denotes the sensor positions of the measurement devices.
Recall that a definition of the used orthonormal basis function is stated in Thm.~\ref{thm:ONBBall}.
Due to properties of the orthonormal basis functions, we are able to calculate for all $m\in \mathbb{N}_0,\ n \in \mathbb{N}$ the value of $B_{m,n}^{(i)}$, see \eqref{app:Bmn} for more steps,
\begin{equation}
  B_{m,n}^{(i)} = \sup_{x \in B_{\varrho_L}} \sum_{j=1}^{2n+1} \abs{h_{m,n,j}^{(i)}} = \frac{\left(4m+2t_n^{(i)}+3\right)(2n+1)}{4\pi\varrho_L^3} \binom{m+t_n^{(i)}+1/2}{m}^2
\end{equation}

Now we need to determine appropriate RKHS with reproducing kernels. 
For this purpose, we set the sequences $\kappa^{(2)} \coloneqq (\kappa_{m,n}^{(2)})_{m,n}$ and $\kappa^{(3)} \coloneqq  (\kappa_{m,n}^{(3)})_{m,n}$ by 
\begin{equation*}\begin{aligned}
  \kappa_{m,n}^{(2)} &\coloneqq \kappa_n^{(2)} \delta_{m,0}, & 
  \kappa_{m,n}^{(3)} &\coloneqq \kappa_n^{(3)} \delta_{m,0} \qquad \text{for all } n\in\mathbb{N},\end{aligned}
\end{equation*}
where we furthermore assume that $\kappa_n^{(i)} \neq 0$ and $\kappa_0^{(i)} = 0$ for all $i\in\{2, 3\}$ and all $n\in\mathbb{N}$.
This reflects the structure of the integral operator null spaces.
We define accordingly the tensor-valued kernels
\begin{equation*}
  \mathfrak{k}^{(i)}(x, y) \coloneqq \sum_{n=1}^\infty \sum_{j=1}^{2n+1} \left(\kappa_{n}^{(i)}\right)^{-2} \tilde{g}_{0,n,j}^{(i)}(\varrho_0,x) \otimes \tilde{g}_{0,n,j}^{(i)}(\varrho_0,y), \qquad i\in\{2, 3\}.
\end{equation*}
Due to the choice of the sequence and the used orthonormal basis function, it is guaranteed that the corresponding spline function does not violate the minimum-norm condition of the neuronal current. 

The summability conditions are given by 
\begin{equation*}\begin{aligned}
  \sum_{n=1}^\infty \left(\kappa_{n}^{(3)}\right)^{-2} \frac{n}{2n+3} \left(\frac{\varrho_0}{\varrho_L}\right)^{2n+2} &< \infty, &
  \sum_{n=1}^\infty \left(\kappa_{n}^{(2)}\right)^{-2} \frac{(2n+1)^2}{n} \left(\frac{\varrho_0}{\varrho_{L-1}}\right)^{2n+2} &< \infty.\end{aligned}
\end{equation*}
According to Def.~\ref{defi:sumSequence}, they consist of the RKHS sequence $\kappa_{n}^{(i)}$, the bound of the orthonormal basis functions $B_{0,n}^{(i)}$, and a sequence $b_n^{(i)}$ which is closely related to the singular values of the operators, \cite[Thm.~6.1, Thm.~6.3]{Leweke2020}.
In order to use the presented vector-valued spline method, we need to verify that the functionals of \eqref{eq:ForwardRMFPOp} and \eqref{eq:ForwardRMFPOp2} are continuous linear functionals mapping from $\mathscr{H}^{(i)}(\kappa^{(i)}, B_{\varrho_0}, \mathbb{R}^3)$ to $\mathbb{R}$ for $i\in\{2,3\}$.  

For proving the boundedness of the functionals, we mainly use the summability conditions.
\begin{theorem}[Boundedness of the Vector Spline Functionals]\label{thm:continuousVecFunctionals}
  Let the functionals mapping the vector-valued current $J$ onto the measured quantities of the MEG and EEG device, be given by  \eqref{eq:ForwardRMFPOp} and \eqref{eq:ForwardRMFPOp2}. Let the Hilbert spaces $\mathscr{H}^{(i)}(\kappa^{(i)}, B_{\varrho_0}, \mathbb{R}^3)$ to $\mathbb{R}$ for $i\in\{2,3\}$ of vector valued functions on the ball be given. 
  Then $\mathcal{A}_{\mathrm{M}}^{k}\colon \mathscr{H}^{(3)}(\kappa^{(3)}, B_{\varrho_0}, \mathbb{R}^3)\to \mathbb{R}$ and $\mathcal{A}_{\mathrm{E}}^k \colon \mathscr{H}^{(2)}(\kappa^{(2)}, B_{\varrho_0}, \mathbb{R}^3)\to \mathbb{R}$ are linear and continuous functionals.
\end{theorem}
\begin{proof}
  The proof can be found in Appendix~\ref{sec:AppSupplImplement}. 
\end{proof}

Now we can determine the following expression required for the representation of the corresponding MEG and EEG spline.
\begin{lemma}\label{lem:reprVecRKHSKernels}
The MEG and EEG functionals applied to the tensor-valued reproducing kernels, have the representations
\begin{align*}
  \mathcal{A}_{\mathrm{M}}^{k}  \mathfrak{k}^{(3)}(\cdot,x)
  &= -\frac{\mu_0}{4\pi} \sum_{n=1}^\infty \left(\kappa_{n}^{(3)}\right)^{-2} \frac{r^n}{s_k^{n+2}} \left((\nu(y_k) \cdot \eta_k) (\xi \wedge \eta_k) P_n^\prime(\xi \cdot \eta_k) \phantom{\frac{1}{{n+1}}}\right.\\
  &\times \left. -\frac{1}{{n+1}}\left( (\xi \wedge \eta_k) P_n^{\prime\prime}(\xi \cdot \eta_k)\left(\xi\cdot(\nu(y_k)-(\nu(y_k)\cdot\eta_k)\eta_k)\right) + P_n^\prime(\xi \cdot \eta_k) \xi \wedge (\nu(y_k) - (\nu(y_k)\cdot\eta_k)\eta_k) \right)\right) \\
   \mathcal{A}_{\mathrm{E}}^k \mathfrak{k}^{(2)}(\cdot,x) 
  &= \frac{1}{4\pi}\sum_{n=1}^\infty \left(\kappa_{n}^{(2)}\right)^{-2} \frac{(2n+1)^{3/2}}{\sqrt{n}} \frac{r^{n-1}}{s_k^{n+1}} \beta^{(L)}_n \left({(n+1)} \left(\frac{s_k}{\varrho_L}\right)^{2n+1} + {n} \right) \tilde{p}_{n}^{(2)}(\xi;\eta_k).
\end{align*}
\end{lemma}
\begin{proof}
  The proof can be found in Appendix~\ref{sec:AppSupplImplement}. 
\end{proof}

\begin{theorem}\label{thm:splineMatrixVector}
  The entries of the vector spline matrices for the MEG and EEG case are given by
  \begin{align*}
    \mathcal{A}_{\mathrm{M}}^l \mathcal{A}_{\mathrm{M}}^k \left(\mathfrak{k}^{(3)}(\cdot, \cdot)\right)
    &= \frac{\mu_0^2}{\varrho_0} \sum_{n=1}^\infty \sum_{j=1}^{2n+1} 
    {\frac{n\, (\kappa_{n}^{(3)})^{-2} }{(2n+1)(2n+3)}} \left(\frac{\varrho_0^2}{s_ls_k}\right)^{n+2} \left(\nu(y_{l}) \cdot \tilde{y}_{n,j}^{(1)}(\eta_l)\right)\left( \nu(y_{k}) \cdot \tilde{y}_{n,j}^{(1)}(\eta_k)\right), \\
    \mathcal{A}_{\mathrm{E}}^l \mathcal{A}_{\mathrm{E}}^k \left(\mathfrak{k}^{(2)}(\cdot, \cdot)\right)
    &= \frac{1}{4\pi}\sum_{n=1}^\infty \left(\kappa_{n}^{(2)}\right)^{-2}  \frac{2n+1}{{n\varrho_0}} \left(\beta^{(L)}_n\right)^2 \left(\frac{\varrho_0^2}{s_ls_k}\right)^{n+1}   \left({(n+1)} \left(\frac{s_l}{\varrho_L}\right)^{2n+1} + {n} \right) \\
    &\phantom{=} \times \left({(n+1)} \left(\frac{s_k}{\varrho_L}\right)^{2n+1} + {n} \right) P_{n}(\eta_l\cdot\eta_k) . 
  \end{align*}
\end{theorem}
\begin{proof}
  The proof can be found in Appendix~\ref{sec:AppSupplImplement}. 
\end{proof}

\section{Numerics}\label{sec:numerics}
In this section, we present the numerical results achieved by the scalar and vector based RKHS spline method. 
First we discuss the MEG case in more detail. 
As stated before, there exist several approaches to solve this inverse problem. 
On the one hand, the scalar approach is discussed, where the neuronal current $J$ is decomposed by means of the Helmholtz decomposition. 
This approach results in a relation between the scalar quantity $A^{(1)}$ and the measured data.
Via the scalar spline method the quantity $A^{(1)}$ can be reconstructed from the data. 
Afterwards, this result can be transferred to a reconstruction of the vector-valued current $J$ by the expression, \cite[Sec.~20.2.4.]{Leweke2020},
\begin{equation}\label{eq:currentFromScalar}
J(x) = - \frac{4\pi}{\varrho_0^2} \sum_{\substack{n\in\mathbb{N} \\ \kappa_n \neq 0}}\sum_{j=1}^{2n+1} \frac{(2n+3)(2n+5)}{(n+1)\sqrt{n(2n+1)}} \kappa_n^{-2} \left(\sum_{k=1}^{\ell_{\mathrm{M}}} \alpha_k \frac{r^n}{s_k^{n+2}} \nu(y_k)\cdot\tilde{y}_{n,j}^{(1)}(\eta_k)\right) \tilde{y}_{n,j}^{(3)}(\xi).
\end{equation}
For the implementation, the addition theorem should be used in order to get rid of the summation over the spherical harmonics of order $j$. 
The particular calculation can be found in the appendix, \eqref{eq:proofA2}.
On the other hand, by means of the SVD-based method a direct relation between the vector-valued quantity $J$ and its measured effect is achieved. 
Using the vector spline method, $J$ can directly be reconstructed from the data. 
Before we compare these two approaches, we need to introduce the setup used for the numerics. 

For solving the inverse MEG and EEG problem, we use the three-shell model, which is the multiple-shell model for $L=3$. In addition, the radii of the shells and their conductivities are assumed to be
\begin{equation}
\begin{aligned}
  \varrho_0 &= \SI{0.071}{\metre}, & \varrho_1 &= \SI{0.072}{\metre},& \varrho_2 &= \SI{0.079}{\metre}, & \varrho_3 &= \SI{0.085}{\metre} , \\
  \sigma_0 &= \SI{0.330}{\siemens\per\metre}, & \sigma_1 &= \SI{1.000}{\siemens\per\metre}, & \sigma_2 &= \SI{0.042}{\siemens\per\metre}, & \sigma_3 &= \SI{0.330}{\siemens\per\metre}.
\end{aligned}
\end{equation}
These values are going back to \cite{Fokas2012} and are based on considerations from \cite{HE2005}. 
In addition, we use a set of real setting of sensor positions collected during a clinical trial of the CBU, Cambridge, UK.
They are plotted in relation to a ball modeling the cerebrum in Fig.~\ref{fig:sensorPosition}. 
Therein, we can see the irregular distribution of the measurement positions. 
In the MEG case, there is a huge data gap in the facial area. 
In the EEG case, almost all sensor positions are located in the upper hemisphere. 

\newlength{\figurewidth}
\begin{figure}
\setlength{\figurewidth}{0.38\textwidth}\centering
  \includegraphics[width=\figurewidth]{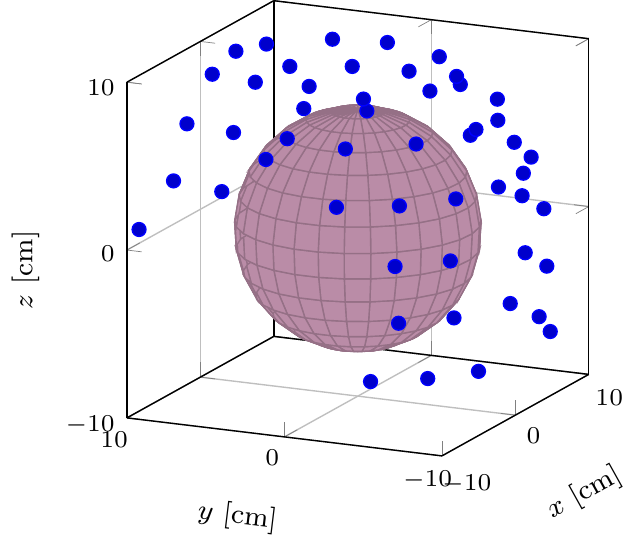}
  \includegraphics[width=\figurewidth]{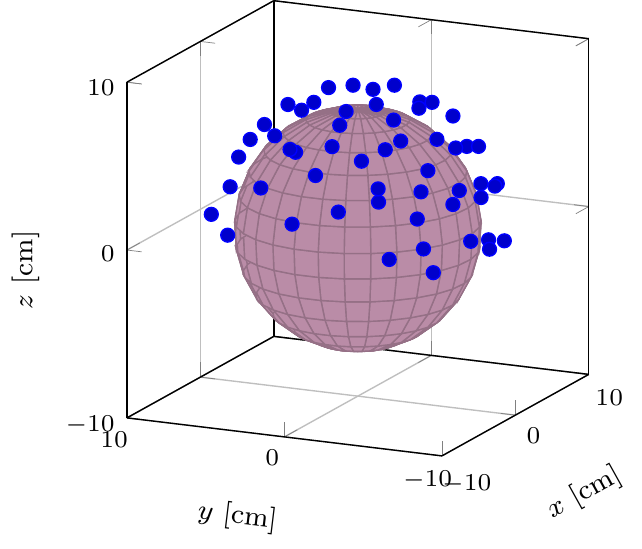}
\caption{Position of the MEG (left) and EEG (right) sensors around the cerebrum modeled by a ball with radius $\varrho_0$}\label{fig:sensorPosition}
\end{figure}
For our numerical tests, we create a realistic but synthetic test current by means of a linear combination of two splines. 
Then, the forward functional is applied to this spline for generating the data. 

In order to avoid the inverse crime, we consider the following aspects:
\begin{itemize}
  \item The splines and its corresponding RKHS used for generating the data must be different than those used for the inversion. In the case of generating the data, we use the sequence
  \begin{equation}\label{eq:sequenceSynthDataSpline}
    \kappa_{n,\text{data}}^{-2} \coloneqq \frac{h_{\text{data}}^{2n+2}}{(2n+5)(2n+1)}, \qquad n \in \mathbb{N}_0,
  \end{equation}
  for building the reproducing kernels. 
  For the free parameter, we choose $h_{\text{data}} = 0.8$. 
  Due to the Legendre series, the corresponding series representation of the reproducing kernel (see also Sec.~\ref{sec:ScalarSpline}) has a closed form based on the Legendre series, \cite[Eq.~(3.2.32)]{Freeden1998} , that is
  \begin{equation*}
    K(x,z) = \frac{1}{4\pi\varrho_0^3} \sum_{n=0}^\infty h_{\text{data}}^{2n+2} \left(\frac{r|z|}{\varrho_0^2}\right)^{n+1} P_n\left(\xi\cdot\frac{z}{|z|}\right)=  \frac{1}{4\pi\varrho_0^4} \frac{hr}{\left| \frac{h_{\text{data}}}{\varrho_0}x - \frac{\varrho_0}{h_{\text{data}} |z|^3} z\right|} .
  \end{equation*}
  \item We do not calculate the forward functional by the SVD-based method stated in \eqref{eq:MEGscalarFunc2Kern}. 
  Instead, we consider the integral representation of the forward function from \eqref{eq:ScalarFunctionalMEG}. More precisely,
  {\begin{equation*}
    \begin{multlined}[0.8\textwidth]
      \phantom{adfsdaf}\mathcal{A}^l_{\mathrm{m},x} \mathcal{A}^k_{\mathrm{m},z} K(x,z)  = \\ 
      \phantom{adfsd}\mu_0^2  \int_{B_{\varrho_0} \times B_{\varrho_0}} \Delta_x \left(|x| \Delta_z \left( |z| K(x,z)\right)\right) \left(\nu(y_l)\cdot\left(\nabla_y K_{\mathrm{m}}(z,y)\right|_{y = y_l}\right) \left(\nu(y_k) \cdot \left(\nabla_y K_{\mathrm{m}}(x,y) \middle)\right|_{y = y_k}\right) \, \mathrm{d}(x,z).\end{multlined}
    \end{equation*}}
    Note that we interchanged the Laplacian with the integration. 
    This is valid, due to the integrability and smoothness of the kernel $K$, see \cite[Cor.~16.3.]{Bauer2001} and \cite[Thm.~7.13]{Leweke2018}.
    In addition, by means of the aid of Mathematica \cite{Mathematica10} we are able to state the identity
  \begin{equation}\label{eq:closedLaplace}
    \phantom{adfsd}\Delta_z \left( |z| K(x,z)\right) = \frac{h}{2\pi\varrho_0^2} |x| \left(3\frac{\varrho_0^2}{h^2} - 4 x \cdot z + \left(\frac{h|x||z|}{\varrho_0}\right)^2\right)\left(\frac{\varrho_0^2}{h^2} - 2x\cdot z + \left(\frac{h|x||z|}{\varrho_0}\right)^2\right)^{-3/2}.
  \end{equation}
  \item The Laplacian with respect to $x$ is calculated numerically by a finite difference method of accuracy order \num{8}, \cite[Tab.~1]{Fornberg1988}. 
  \item The closed representation of the integral kernels gradient is used, \cite[Thm.~15.15]{Leweke2018},
  \begin{align}
    \phantom{adh} 4\pi \nabla_y K_{\mathrm{m}}(x,y) \coloneqq& \nabla_y \sum_{k=1}^\infty \frac{r^k}{s^{k+1} (k+1)} P_k(\xi \cdot \eta)\label{eq:repr_kernel} \\
    =& \begin{cases}
    \left( \frac{1}{sr((\xi\cdot\eta)^2-1)} \left[\left(\frac{s-x\cdot\eta}{|x-y|}-1\right)\xi - \left(\frac{\xi\cdot y-r}{|x-y|}- \xi\cdot\eta \right)\eta\right] + \frac{1}{s^2}\eta\right), & |\xi\cdot \eta| \neq 1, \\
    \left(\frac{1}{s(s-r)} + \frac{1}{s^2}\right)\eta, & \xi\cdot \eta = 1, \\
    \left(\frac{1}{s(s+r)} - \frac{1}{s^2}\right)\eta, & \xi\cdot \eta = -1. \\
    \end{cases}
  \end{align}
  \item The integration is performed by a quasi-Monte Carlo method, \cite{Dick2013}, with \num{4e6} integration points distributed via a quasi-random Kronecker sequence $t_n = \{n \gamma\}$, $n \in \{1,2,\dots,\num{4e6}\}$, and $\gamma = (1/\Phi_d^i)_{i = 1,\dots,d}$ based on the $d$-dimensional generalization of the golden ratio $\Phi_d$ as the solution of $x^{d+1} = x +1$, \cite{Roberts2018};
  \item We put additive Gaussian noise with $1$, $2$, $5$, and $10$ percent strength on the data.
\end{itemize}
On the one hand, these steps lead us to a numerical method for calculating the entries of the spline matrix. 
We will refer to this method as the numerical analysis method, since it contains numerical differentiation and integration.
On the other hand, we have the possibility to calculate the spline matrix based on the vector spherical harmonic expansion which is related to an SVD of the operator, that is \eqref{eq:scalarSplineMatrixEntry} combined with the addition theorem stated in \eqref{eq:scalarSplineMatrixEntryImpl}. 
This ansatz will be referred to as the SVD-based approach and yields to a series representation for each entry. 
These series of Legendre polynomials can be summed up efficiently by means of Clenshaw's algorithm. 
Based on Clenshaw's algorithm, we have implemented a highly vectorized and optimized code using MATLAB R2021a, \cite{MATLAB:2021}.
For determining the run time of \SI{0.4285(69)}{\second} in the case of \num{500} summands per series, we have run this algorithm \num{1000} times sequentially on a local machine. 
In contrast, this approach requires comprehensive knowledge of the forward operator in order to obtain the series representation. 
This theoretical effort is rewarded with a stable and very fast numerical method.
For the numerical analysis method, theoretical effort is required in order to gain the closed representations of the kernels. 
Afterwards the numerical differentiation can be operated fast by using vectorized code. 
However, each spline matrix entry requires a numerical integration over $B_{\varrho_0}\times B_{\varrho_0}$. 
In order to achieve an acceptable accuracy, \num{4e6} integration points are required for each integration over the three-dimensional ball $B_{\varrho_0}$. 
We also vectorized and parallelized the code and operated it on a node of the OMNI cluster (University of Siegen, 64 CPU cores, AMD EPYC 7452 at \SIrange{2.35}{3.35}{\giga\hertz}, \SI{256}{\giga\byte} memory) in about \num{18} days.
This slowness reveals the major disadvantage of this method, which comes along with the curse of dimensions.
However, the numerical integration may also offer the possibility to extend this method to more real-shaped brain geometries, which is not pursued within the scope of this paper. 

In Fig.~\ref{fig:splineMatrix}, the spline matrices achieved by the SVD-based method is visualized. The matrix is ill-conditioned with condition number \num{6.4053e+11} with respect to the $2$-norm. 
In addition, the element-wise relative derivation of the spline matrix achieved by the two approaches described above is plotted on a logarithmic scale. 
Even though there are several matrix entries with a higher point-wise deviation, the average of the relative deviation over all matrix elements is \SI{0.13}{\percent}. 
The higher point-wise deviation is related to entries closer to zero.
Hence, we conclude that the SVD-based method as well as the numerical analysis based method yield comparable results with respect to the accuracy. 

Recall that the data generated by the linear combination of two splines consists of the linear combination of the corresponding two rows of the spline matrix. 
For generating the data, we calculate the spline matrix by means of the numerical analysis method described before.
For the inversion of the data, we use the SVD-based method due to its immense speed-up to generate the spline matrix. 
The used setting is described in the next section.

For a better comparison between the results obtained from the scalar and the vector spline method, we use the same data set for both inversions. 
\begin{figure} 
\setlength{\figurewidth}{0.32\textwidth}\centering
\includegraphics[width=\figurewidth]{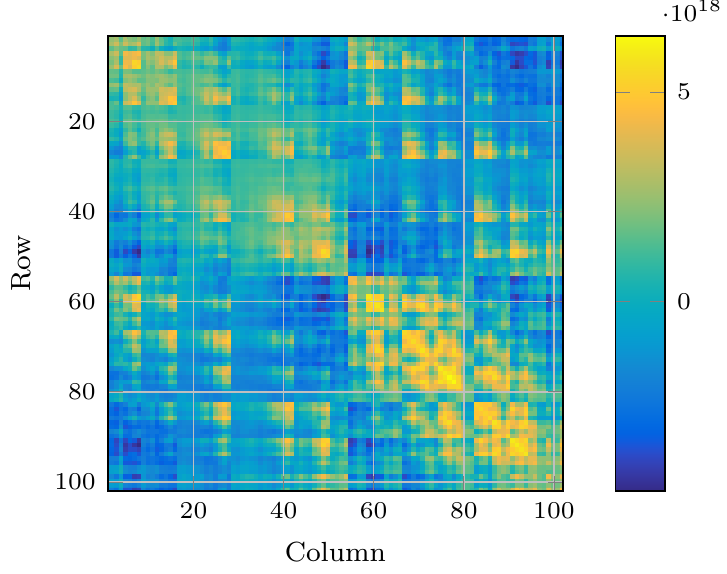} 
\includegraphics[width=1.07\figurewidth]{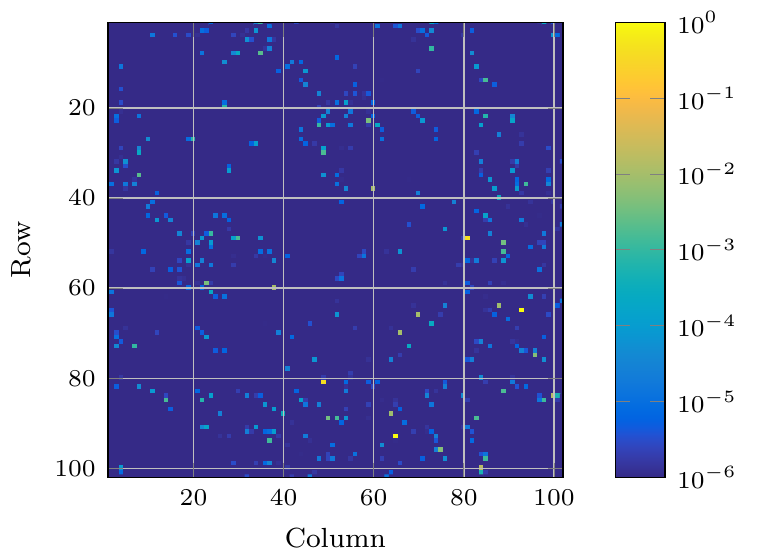}
\caption{Matrix of the scalar MEG splines with $h_{\text{data}}=0.8$ and sequence \eqref{eq:sequenceSynthDataSpline}, calculated by means of the SVD, 
\eqref{eq:scalarSplineMatrixEntry}, (left) and its logarithmic plot, entry-wise relative difference to the matrix obtained by the numerical method (right)} \label{fig:splineMatrix}
\end{figure}

\subsection{Spline Based Inversion of the MEG Data}
As stated before, a linear and continuous functional $\mathcal{A}^k_\mathrm{m}$ connecting the measured data $g_k$, $k=1,\dots,\ell_{\mathrm{M}}$, with the radial component $A^{(1)}$ from the vector field $a$ of the neuronal current Helmholtz decomposition \eqref{eq:Helmholtz} is known, \eqref{eq:MEGscalarFunc2Kern}.
In order to solve the problem
\begin{equation*}
  \mathcal{A}^k_\mathrm{m} A^{(1)} = g_k, \qquad k=1,\dots, \ell_{\mathrm{M}},
\end{equation*}
by means of the reproduced kernel based spline methods, we need to choose the symbols of the reproducing kernel. 
We set 
\begin{equation}\label{eq:symbolSpline}
    \kappa_n^{-2} \coloneqq \frac{h^{n}}{n}, \qquad n \in \{1,2,\dots,200\},\, h \in \{0.85, 0.9, 0.95, 0.99\}.
  \end{equation}

The RKHS spline can now be determined as in \eqref{eq:DefiScalarSpline}. 
A visualization of splines with only one and two non-vanishing coefficients $\alpha_k$ can be found in Fig.~\ref{fig:reproSpline}. 
Therein the localization of the spline in an area around its node can be seen. 
When the free parameter $h$ tends closer to $1$, the reproducing kernel becomes more localized and the amplitude rises.
\begin{figure}
  \setlength{\figurewidth}{0.33\textwidth}\centering
  \begin{subfigure}[t]{0.31\textwidth}
    \includegraphics[width=\figurewidth]{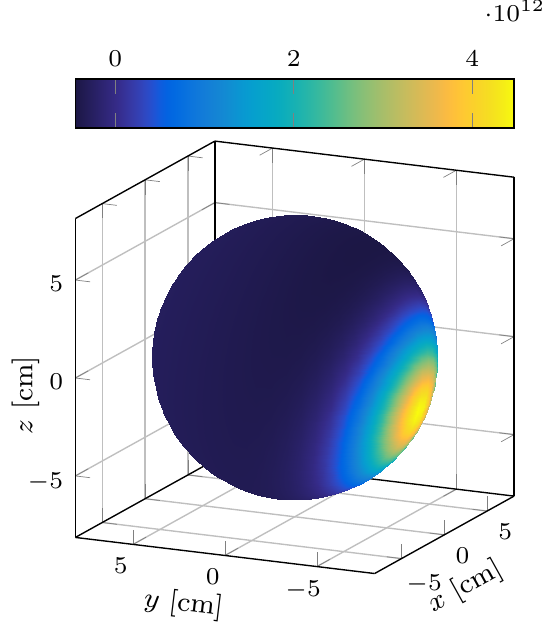}  %
    \subcaption{Scalar RKHS spline}
  \end{subfigure}\hfill
  \begin{subfigure}[t]{0.31\textwidth}
    \includegraphics[width=\figurewidth]{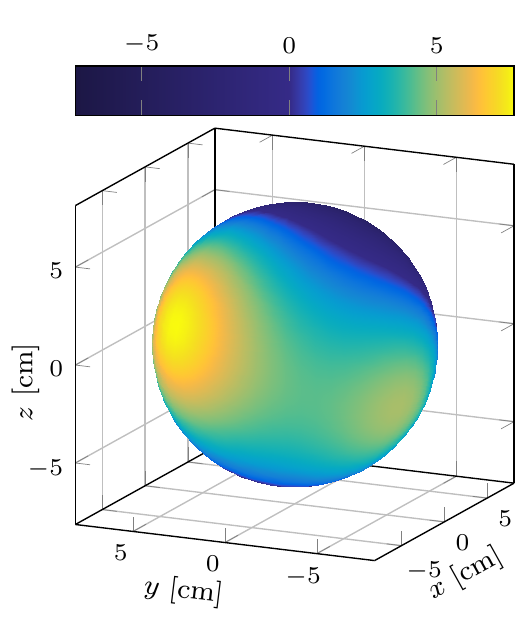}
    \subcaption{Exact solution of spline based scalar test case for $A^{(1)}$}\label{fig:exactSolutionA}
  \end{subfigure}\hfill
  \begin{subfigure}[t]{0.31\textwidth}
    \includegraphics[width=\figurewidth]{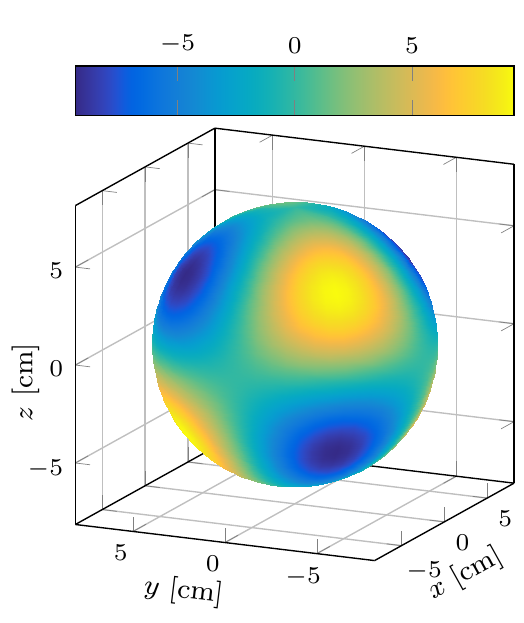}
    \subcaption{Exact solution of ONB based scalar test case for $A^{(1)}$}\label{fig:exactSolutionB}
  \end{subfigure}\hfill
  \caption{Radially symmetric RKHS spline generated by the sequence \eqref{eq:symbolSpline} with $h=0.95$ (left), 
  a linear combination of two splines with $h=0.8$ and sequence \eqref{eq:sequenceSynthDataSpline}, 
  which serves as the solution of $A^{(1)}$ in the spline based synthetic test case (middle), and the ONB based test case solution $A^{(1)}$ (right). 
  All plotted on a sphere inside the cerebrum with radius $0.99\varrho_0$.} \label{fig:reproSpline}
\end{figure}

For each data set, we solve the regularized linear equation system, \eqref{eq:TikReg},  for \num{500} different regularization parameters $\lambda$. 
The regularization parameters are logarithmically uniformly distributed and weighted with the absolute maximal value of the spline matrix. Hence, they are in the rage \numrange{1e-25}{1e-9}. 
In order to pick the best parameter, five parameter choice methods are applied afterwards, which are the automatic and the manual L-curve method (LCM), the discrepancy principle (DP), 
the quasi-optimality criterion (QOC), and the generalized cross validation (GCV), see \cite{Gutting2017,Bauer2011,Bauer2015} and the references therein.
For the resulting up to five `best' parameters, the normalized root mean square error (NRMSE) of the difference to the exact solution is calculated on the plotting grid. 
Then the parameter choice method resulting in the smallest NRMSE is chosen. 
The L-curve for the spline based synthetic test case with \SI{5}{\percent} noise is plotted in Fig.~\ref{fig:LCurve}.

\begin{figure}[h]
  \centering
  \includegraphics{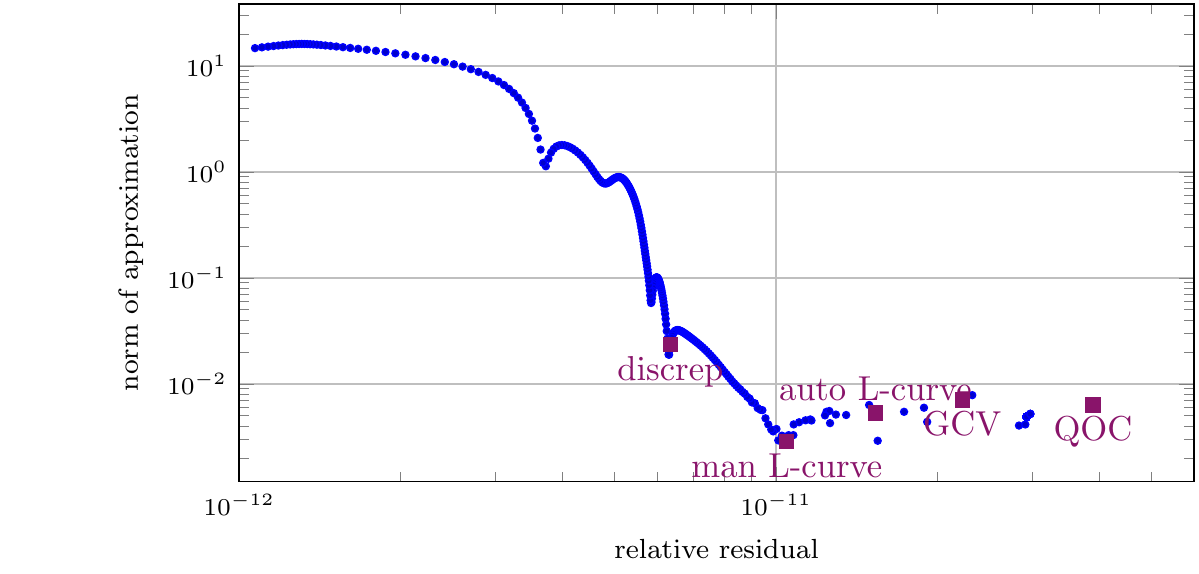}
  \caption{Double logarithmic plot of the relative residual against the norm of approximation depending on the regularization parameter. Regularization parameters preferred by the five parameter choice methods are highlighted.}
  \label{fig:LCurve}
\end{figure}

For each of the five different noise levels (\SI{0}{\percent},\,\SI{1}{\percent},\,\SI{2}{\percent},\,\SI{5}{\percent},\,\SI{10}{\percent}), 
we calculated the spline matrix with four different parameters $h$, see \eqref{eq:symbolSpline}. 
Among these \num{20} inversions, the best regularization parameter came only once from the GCV. 
In the other cases, the L-curve method provided the best regularization parameter. 
With low noise levels (up to \SI{2}{\percent}) the automatic L-curve method yields good results.
With higher noise levels, the a-priori range of the regularization parameters gets more important. 
If the range is chosen to be wide, the system of linear equations becomes easily over-regularized and the norm of the solution is so small (magnitude of \num{1e-12} or less) 
such that the automatic L-curve chooses regularization parameters leading to a relative data misfit of \num{1e2} or higher. 
This behavior can be avoided by restricting the parameter range a-priori or using the manual L-curve method.  

\begin{table}
\begin{tabular}{r r r r r r}\toprule
 noise level  & 0\%               & 1\%               & 2\%               & 5\%              & 10\% \\\midrule
$h$             & \num{0.8500}      & \num{0.8500}      & \num{0.9900}      &\num{0.8500}      & \num{0.8500}\\[1ex]
rel. residual & \num{4.2642e-08}& \num{0.0081}    & \num{0.0169}    &\num{0.0455}    & \num{0.0834}\\
rel. NRMSE  $A^{(1)}$  & \num{5.3310e-04} & \num{0.0194}    & \num{0.0328}    &\num{0.0472}    & \num{0.0599} \\[1ex]
rel. NRMSE $J$   & \num{5.3411e-04} & \num{0.0192}    & \num{0.0371}    &\num{0.0474}    & \num{0.0591} \\[1ex]
$\lambda$     & \num{4.4054e-17}& \num{1.3395e-11}& \num{4.2240e-11}&\num{5.2497e-11}& \num{5.7572e-11}\\
pcm           & LCM man        & LCM auto        & LCM auto        & LCM man        & LCM man\\\midrule
$h$             & \num{0.3771}      & \num{0.3771}      & \num{0.3771}      &\num{0.3771}      & \num{0.3771}\\[1ex]
rel. residual & \num{4.5957e-04}& \num{0.0080}    & \num{0.0166}    &\num{0.0451}    & \num{0.0874}\\
rel. NRMSE $J$   & \num{0.0051} & \num{0.0140}    & \num{0.0240}    &\num{0.0416}    & \num{0.0455} \\[1ex]
$\lambda$     & \num{3.6990e-15}& \num{1.4782e-15}& \num{1.4782e-15} &\num{2.5292e-14}& \num{1.4782e-14}\\
pcm           & LCM auto        & LCM auto        & LCM auto        & LCM auto        & LCM man\\\bottomrule
\end{tabular}
\caption{Numerical results for the MEG synthetic spline based test case achieved by the scalar spline method (top) and the vector spline method (bottom)}
\label{tab:scalarMEGResults}
\end{table}

Numerical results for the synthetic MEG data inversion such as the relative residual as well as the relative NRMSE with respect to different noise levels are listed in Tab.~\ref{tab:scalarMEGResults}.
For each noise level, we selected among the four values of $h$, the one resulting in the best numerical result with respect to the relative NRMSE. 
We can also see, that $h=0.85$ is chosen most frequently. 
A possible explanation is that the corresponding width of the spline location fits the best to the distances of the measurement positions and is the closest to the synthetic test case value $h_{\text{data}}$.
Besides the chosen $h$, the relative data misfit or residual, respectively, is stated in Tab.~\ref{tab:scalarMEGResults}. 
Even with increasing noise level, the relative residual is often more than slightly below the noise level. 

Taking the relative NRMSE and visualizations of the reconstructions of $A^{(1)}$ into account, Fig.~\ref{fig:scalarMEG02}, one can see that the approximations are not over-fitted.
In contrast, the active regions are covered well by the reconstruction.
In the case of lower noise levels, the amplitudes of the current are matching.
Only in the case of higher noise levels they become a little bit blurred and faded, which goes back to the required stronger regularization.
Since deviations to the exact solution are hard to see for lower noise levels, we have a closer look at them in Fig.~\ref{fig:scalarMEG510}.
Near the active regions, the deviations are the largest, which was anticipated.
In addition, larger differences can be seen in the lower part of the ball. 
This behavior is directly related to the irregular distribution of measurement points, see Fig.~\ref{fig:sensorPosition}, since there is only few information of the lower part of the ball and in the facial area. 
The artifactual pattern in the deviation plots leads back to the structure of the splines.  
All in all, the reconstructions of $A^{(1)}$ yield good and stable results with respect to increasing noise level. 

\begin{figure}
  \setlength{\figurewidth}{0.33\textwidth}\centering  
  \begin{subfigure}[t]{0.31\textwidth}
\includegraphics[width=\figurewidth]{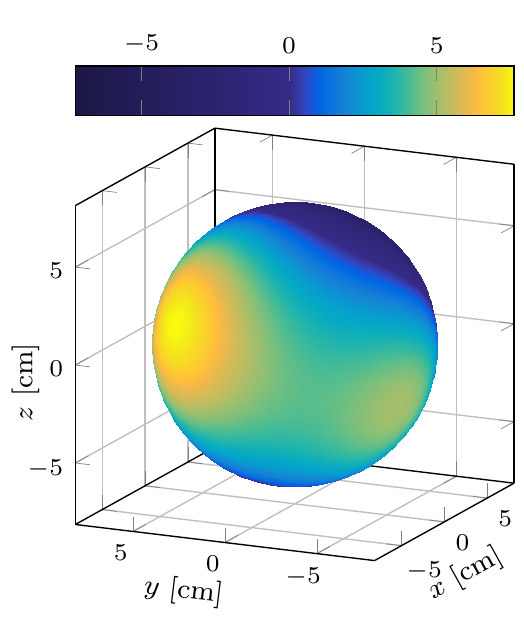}  %
    \subcaption{\SI{0}{\percent} noise on the data}
  \end{subfigure}\hfill
  \begin{subfigure}[t]{0.31\textwidth}
\includegraphics[width=\figurewidth]{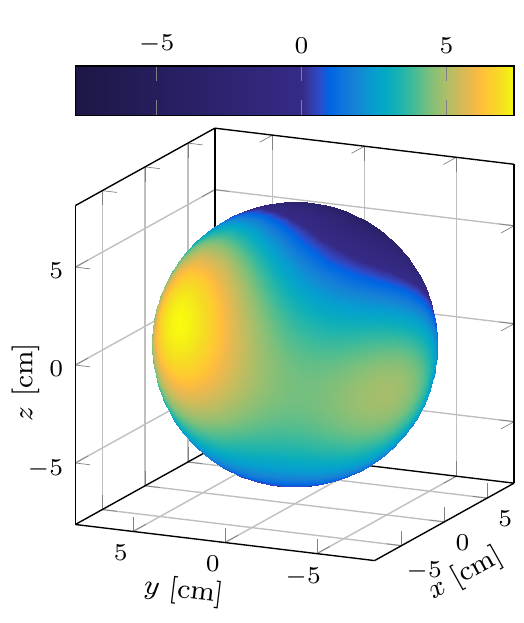} 
    \subcaption{\SI{5}{\percent} noise on the data}
  \end{subfigure}\hfill
  \begin{subfigure}[t]{0.31\textwidth}
\includegraphics[width=\figurewidth]{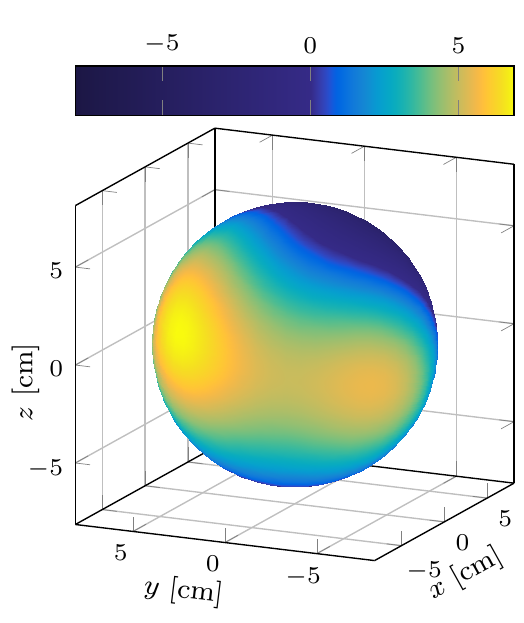} 
    \subcaption{\SI{10}{\percent} noise on the data}
  \end{subfigure}
\caption{Reconstruction of the neuronal current component $A^{(1)}$ in the synthetic spline based scalar MEG test case (\autoref{fig:exactSolutionA}) from \SI{0}{\percent} (left), \SI{5}{\percent} (middle), and \SI{10}{\percent} (right) noisy data} \label{fig:scalarMEG02} 
\end{figure}

\begin{figure}  
\setlength{\figurewidth}{0.33\textwidth}\centering  
  \begin{subfigure}[t]{0.31\textwidth}
\includegraphics[width=\figurewidth]{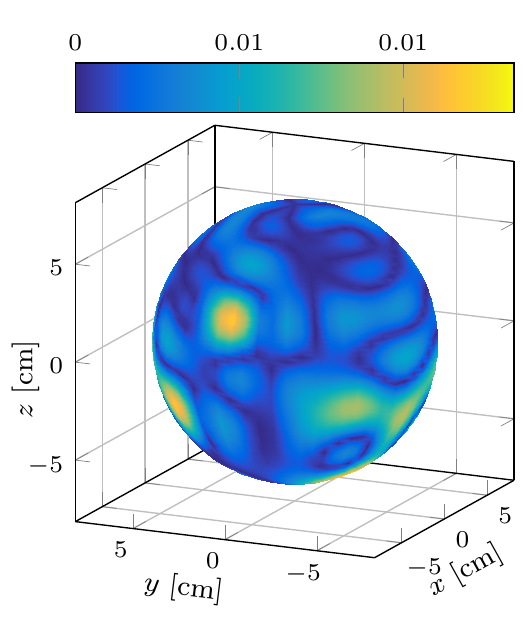} %
    \subcaption{\SI{0}{\percent} noise on the data}
  \end{subfigure}\hfill
  \begin{subfigure}[t]{0.31\textwidth}
\includegraphics[width=\figurewidth]{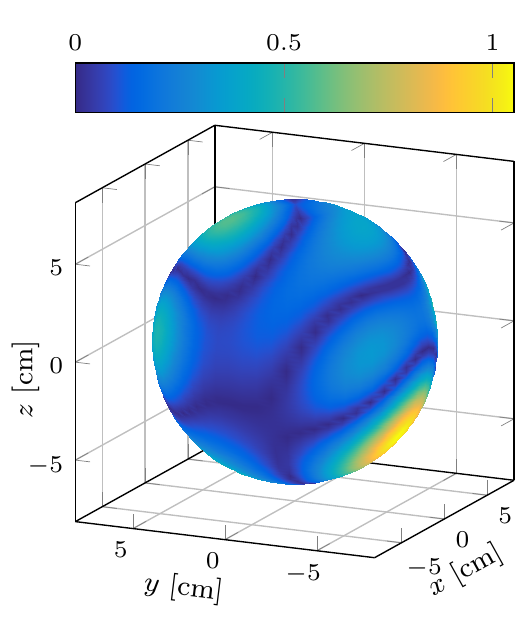}
    \subcaption{\SI{5}{\percent} noise on the data}
  \end{subfigure}\hfill
  \begin{subfigure}[t]{0.31\textwidth}
\includegraphics[width=\figurewidth]{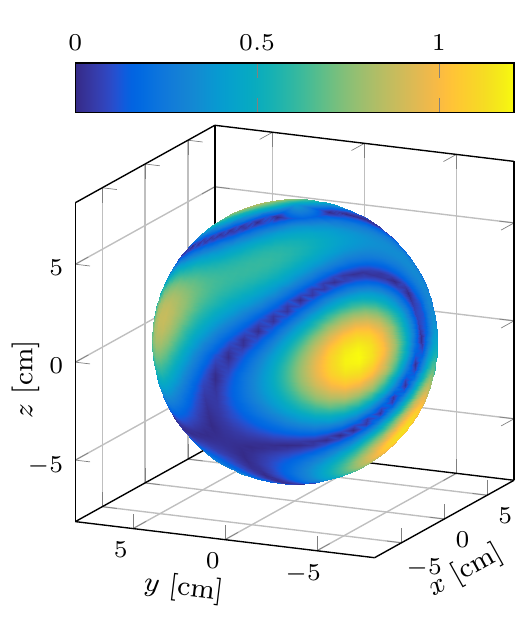}
    \subcaption{\SI{10}{\percent} noise on the data}
  \end{subfigure}
\caption{Absolute deviation of the neuronal current component $A^{(1)}$ in the synthetic spline based scalar MEG test case from \SI{0}{\percent} (left), \SI{5}{\percent} (middle), \SI{10}{\percent} (right) noisy data}\label{fig:scalarMEG510} 
\end{figure}

Before we discuss the corresponding vector-valued reconstruction, we verify the accuracy and stability of the numerical method by means of the inversion of synthetic data which is not generated by a spline method at all.  
For this purpose, we use the exact solution $A^{(1)}(r\xi) = r\xi \mapsto 0.1 G_3(r)Y_{3,6}(\xi)$, which we refer to as the ONB based test case solution and is plotted in Fig.~\ref{fig:exactSolutionB}.
The angular part of the function is given by the a spherical harmonic of degree \num{3} and order \num{6}, whereas the formula for the radial part is stated in \eqref{eq:GnBasis}.  
For this particular test case, we have also chosen $h=0.85$, since it yields the best results throughout our numerical tests. 
Since the non-noisy and low-noisy reconstructions visually coincide with the exact solution, we have only plotted the deviations in Fig.~\ref{fig:onbscalarMEG0210}.
Recall that the absolute maximum values of the exact solutions  are approximately \num{9.5}. 

The approximation recovers the active regions and the amplitude well, even with \SI{10}{\percent} noise on the data, the reconstruction does not seem to be blurry or fade out. 
This is also reflected by the relative residual which is significantly smaller than the noise level despite resulting in an satisfying small relative NRMSE. 
The particular values are stated in the captions of Fig.~\ref{fig:onbscalarMEG0210}.
Besides, there are artifacts based on the spline structure which can be seen in the deviation plots.
Unattached from the noise level, the areas most difficult to reconstruct are located in the data gap area.
This is especially revealed in the non-noisy case taking into account that the exact solution is rotationally symmetric. 
Besides these small artifacts in the data gap area, the used spline method can handle the irregularly distributed data situation very well. 

\begin{figure}
\setlength{\figurewidth}{0.33\textwidth}\centering
  \begin{subfigure}[c]{0.31\textwidth}
    \includegraphics[width=\figurewidth]{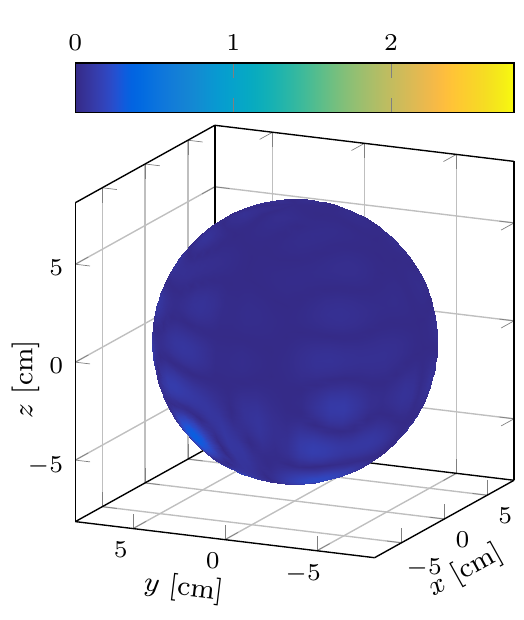}
  \subcaption{noise level: \SI{0}{\percent}\\
  rel. residual: \num{1.0092e-06}\\ 
  rel. NRMSE: \num{0.0076}}
  \end{subfigure}\hfill
  \begin{subfigure}[c]{0.31\textwidth}
    \includegraphics[width=\figurewidth]{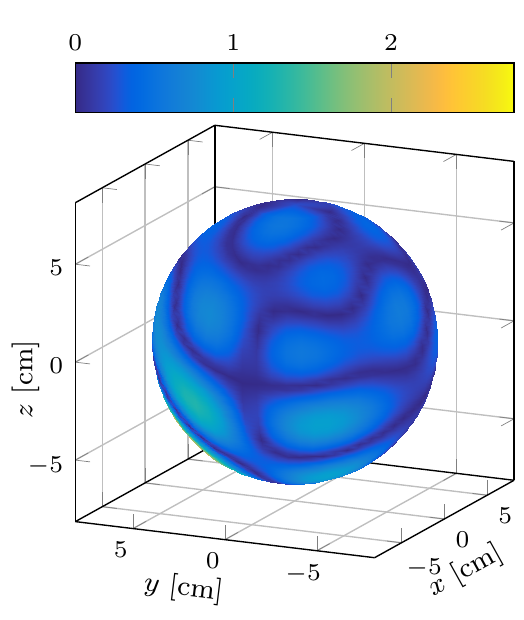}
  \subcaption{noise level: \SI{5}{\percent}\\
  rel. residual: \num{0.0313}\\ 
  rel. NRMSE: \num{0.0623}}
  \end{subfigure}\hfill
  \begin{subfigure}[c]{0.31\textwidth}
   \includegraphics[width=\figurewidth]{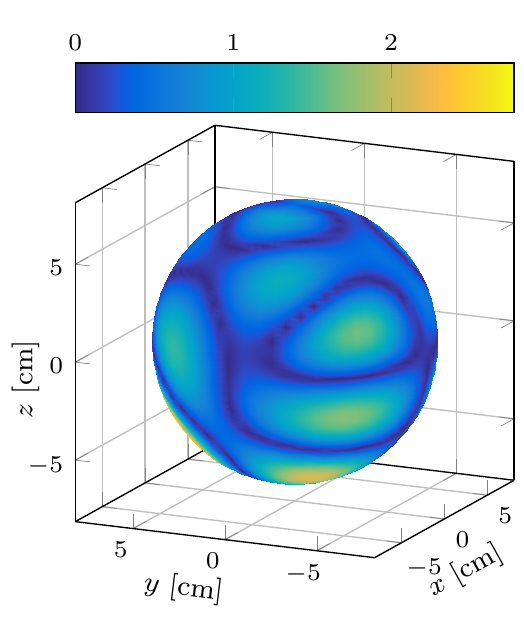} 
\subcaption{noise level: \SI{10}{\percent}\\
rel. residual: \num{0.0619}\\ 
rel. NRMSE: \num{0.0891}}
\end{subfigure}\hfill
\caption{Absolute deviation of the reconstruction of the neuronal current component $A^{(1)}$ in the ONB based synthetic scalar MEG test case to the exact solution $r\xi \mapsto 0.1 G_3(r)Y_{3,6}(\xi)$ depending on the noise levels.
For a better comparability, the colorbars were chosen equally.} \label{fig:onbscalarMEG0210}
\end{figure}

However, we are mainly interested in the approximation and visualization of the neuronal current, hence in the vector-valued quantity. 
In order to achieve this, we first transfer the scalar spline solution to the neuronal current by means of \eqref{eq:currentFromScalar}.
The exact solutions of the current $J$ in the spline based as well as the ONB based test case are plotted in Fig.~\ref{fig:vectorExactSolution}.
Besides the absolute values of the currents, which are visualized via surface plots, the direction of the currents are depicted via arrows. 
\begin{figure}
\setlength{\figurewidth}{0.35\textwidth}\centering
\includegraphics[width=\figurewidth]{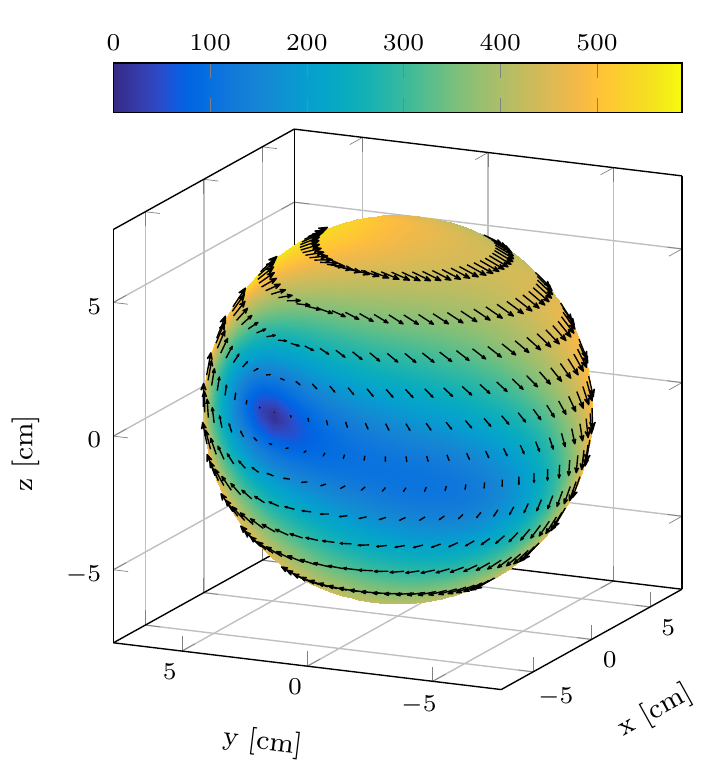}
\includegraphics[width=\figurewidth]{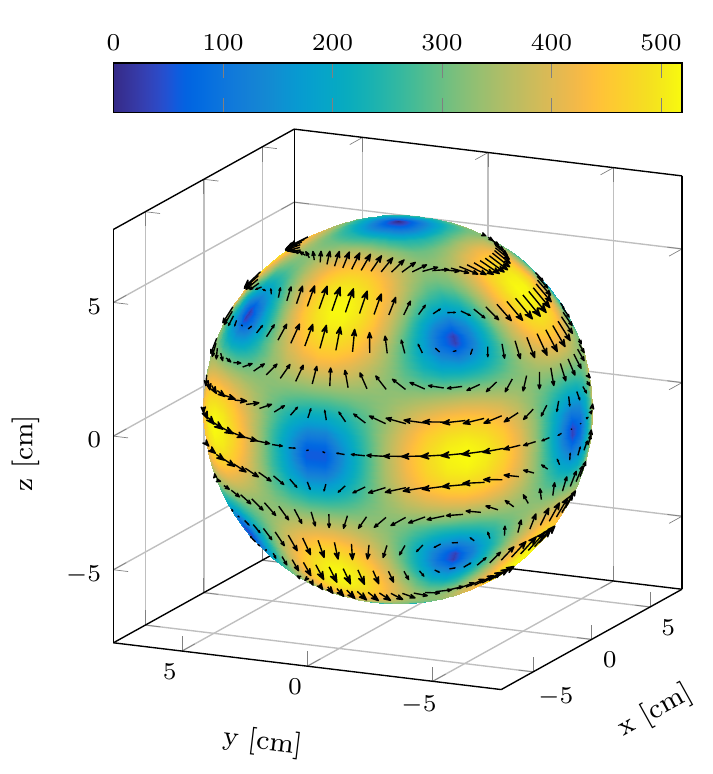}
\caption{Exact solution of current $J$ for the spline based synthetic test case (left) and the ONB based test case (right)}\label{fig:vectorExactSolution}
\end{figure}

We start with a look at the ONB based synthetic test case. 
In the non-noisy case, the quality of the vector-valued current reconstruction is totally comparable to the quality of the scalar one, Fig.~\ref{fig:currentonb}. 
Also in the low-noisy cases we observe that the main deviations to the exact solution are located in the data gap area and the remaining approximation fits the exact solution very well with respect of localization and amplitude. 
In the case of \SI{10}{\percent} noise on the data additional reconstruction errors occur besides the data gap and the reconstruction appears to be muted.
This is also manifested in the values of the relative NRMSE. 
Compared to the situation of the scalar approximation, see Fig.~\ref{fig:onbscalarMEG0210}, the relative NRMSE increases due to the transformation.
\begin{figure}\hspace*{-0.17\textwidth}
\setlength{\figurewidth}{0.55\textwidth}\centering
\begin{subfigure}[c]{0.31\textwidth}
  \includegraphics[width=\figurewidth]{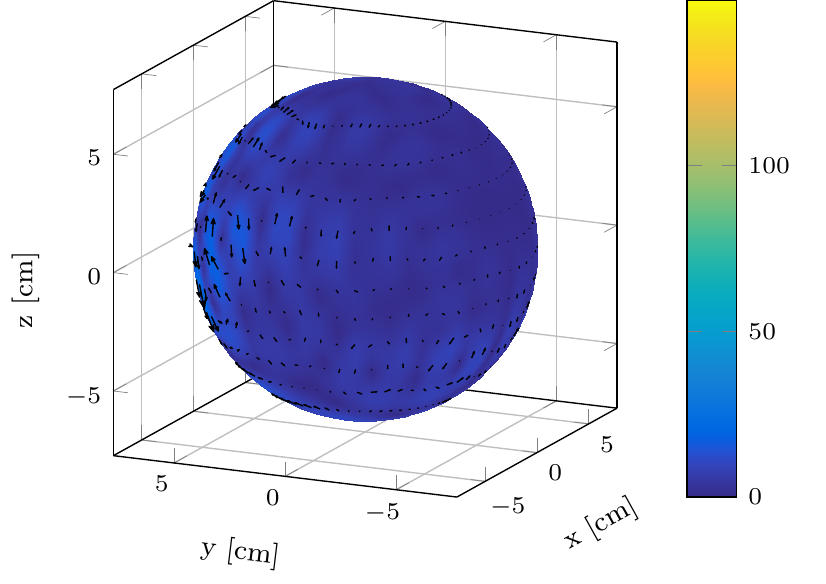} 
\subcaption{Deviation from non-noisy data \\ 
rel. NRMSE: \num{0.0154}}
\end{subfigure}\hspace*{0.17\textwidth}
\begin{subfigure}[c]{0.31\textwidth}
  \includegraphics[width=\figurewidth]{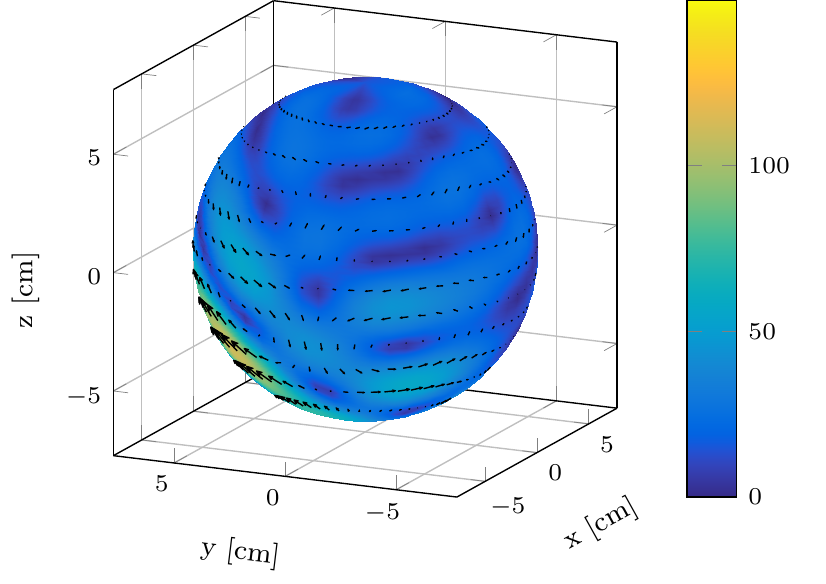}
\subcaption{Deviation from \SI{5}{\percent} noisy data\\ rel. NRMSE: \num{0.0787}}
\end{subfigure}

\hspace*{-0.17\textwidth}
\begin{subfigure}[t]{0.31\textwidth}
  \includegraphics[width=\figurewidth]{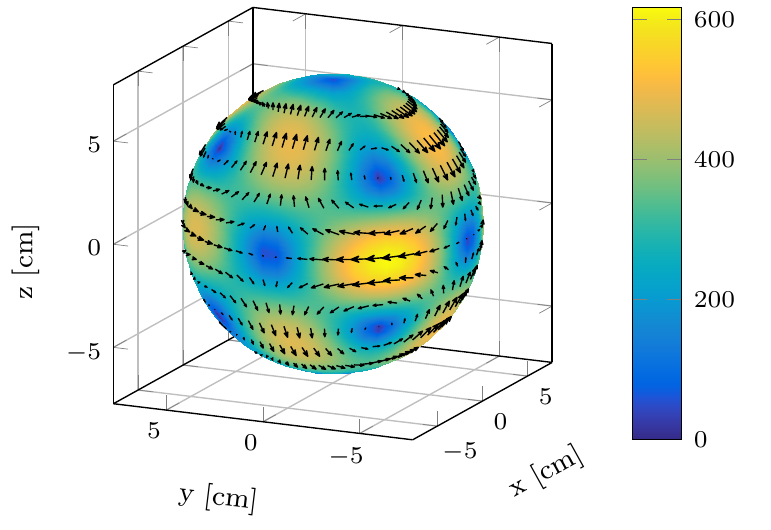}
\subcaption{Reconstruction from \SI{10}{\percent} noise}
\end{subfigure}\hspace*{0.17\textwidth}
\begin{subfigure}[t]{0.31\textwidth}
  \includegraphics[width=\figurewidth]{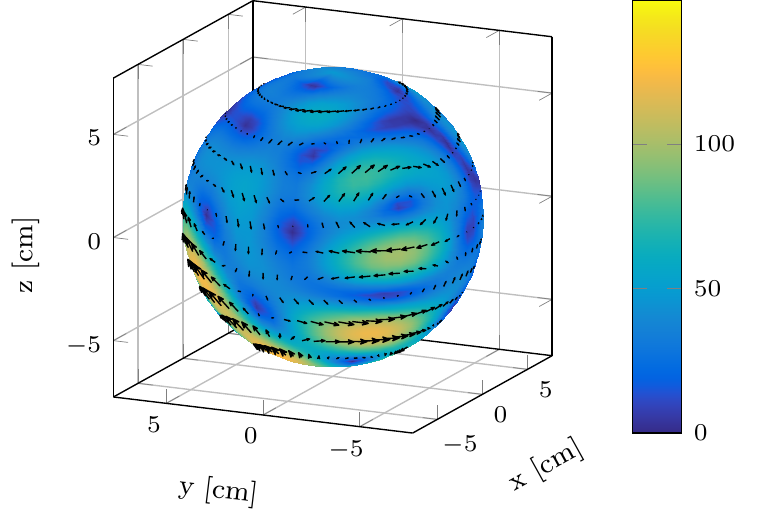} 
\subcaption{Deviation from \SI{10}{\percent} noisy data\\ rel. NRMSE: \num{0.1033}}
\end{subfigure}
\caption{Numerical results for the ONB based synthetic test case achieved from the transformation of the scalar reconstruction from Fig.~\ref{fig:onbscalarMEG0210}.
For a better comparability, the colorbars of the deviation plots were chosen equally.}\label{fig:currentonb}
\end{figure}

The relative NRMSE of the transferred vector-valued approximation of $J$ for the spline based synthetic test case is also stated in Tab.~\ref{tab:scalarMEGResults} (top). 
In this particular numerical experiment the relative NRMSE of the scalar reconstruction and the transferred one are quite similar. 
However, they are outperformed by the relative NRMSE of the direct vector-valued reconstruction of $J$ via the vector spline method, whereas the relative residual remains significantly lower than the noise level.
This gain of quality in the reconstruction can also be seen in the deviation plots, Fig.~\ref{fig:vectorMEGspline}.
In the first row, the inversion of data with noise level \SI{2}{\percent} is performed and with \SI{10}{\percent} noise on the data in the second row. 
In the left column the differences of the transferred scalar reconstruction to the exact solution are plotted. 
In comparison to the reconstructions of the direct vector spline method plotted in the middle column, we see significantly higher deviations spread over the entire ball.
In the case of the vector spline reconstruction from \SI{2}{\percent}, the highest deviations by far are located in the data gap area.
In the right column, the vector spline reconstructions are plotted, which are stable with increasing noise level.
Due to the construction of the vector reproducing kernels, the vector splines are significantly more localized than the scalar ones for higher parameter $h$. 
Based on our numerical investigation, the adaption $h = 0.85^6$ for the free parameter yielded reasonable results in order to avoid an overfitting of the data.  

\begin{figure}
\setlength{\figurewidth}{0.33\textwidth}\centering
\begin{subfigure}[t]{0.31\textwidth}
  \includegraphics[width=\figurewidth]{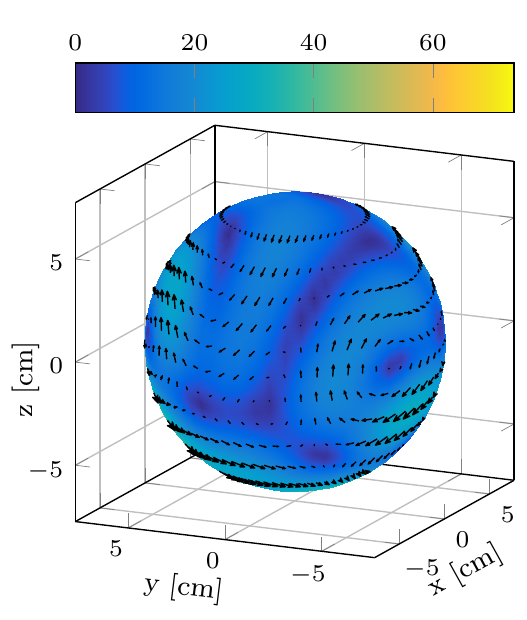} 
  \subcaption{Difference of converted scalar spline reconstruction (\SI{2}{\percent} noise)}\label{fig:vecA}
\end{subfigure}
\begin{subfigure}[t]{0.31\textwidth}
  \includegraphics[width=\figurewidth]{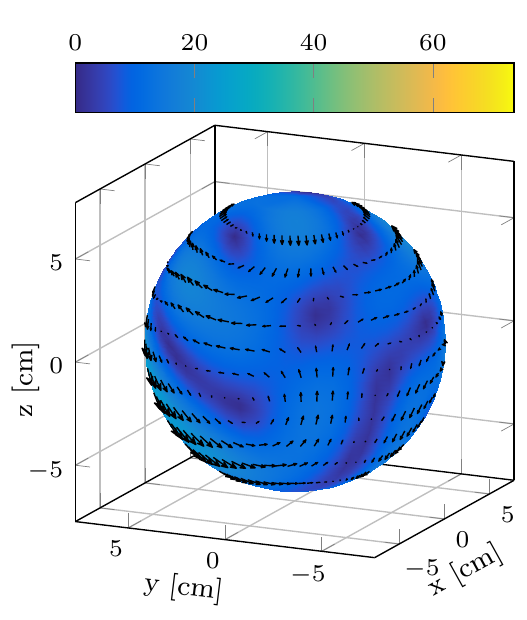}
  \subcaption{Difference of vector spline reconstruction (\SI{2}{\percent} noise)}\label{fig:vecB}
\end{subfigure}
\begin{subfigure}[t]{0.31\textwidth}
  \includegraphics[width=\figurewidth]{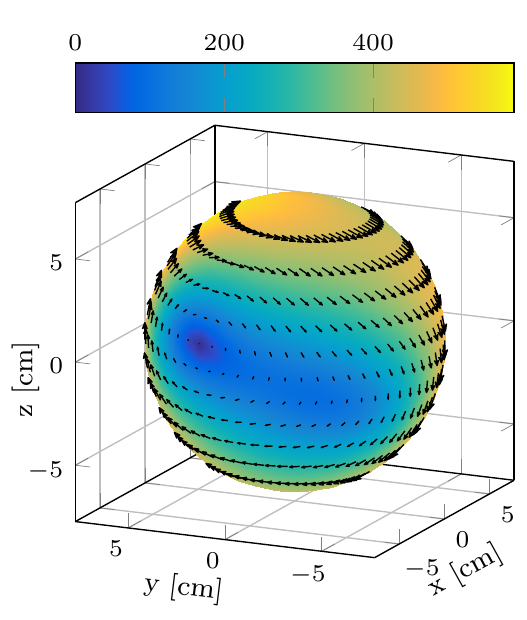}
  \subcaption{Reconstruction via vector splines (\SI{2}{\percent} noise)}
\end{subfigure}
\begin{subfigure}[t]{0.31\textwidth}
  \includegraphics[width=\figurewidth]{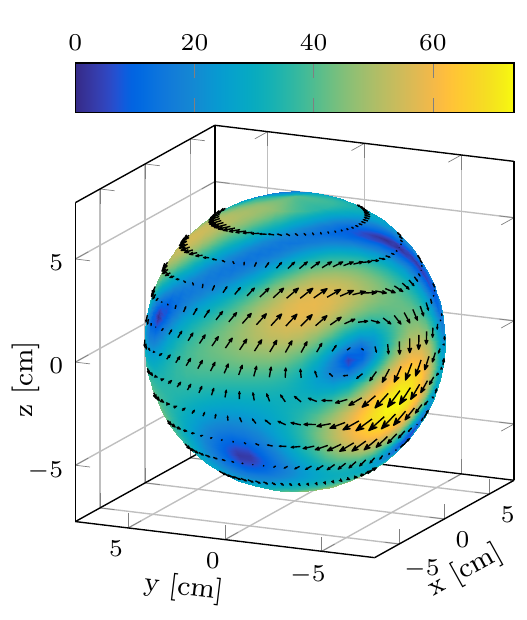} 
  \subcaption{Difference of converted scalar spline reconstruction (\SI{10}{\percent} noise)}\label{fig:vecD}
\end{subfigure}
\begin{subfigure}[t]{0.31\textwidth}
  \includegraphics[width=\figurewidth]{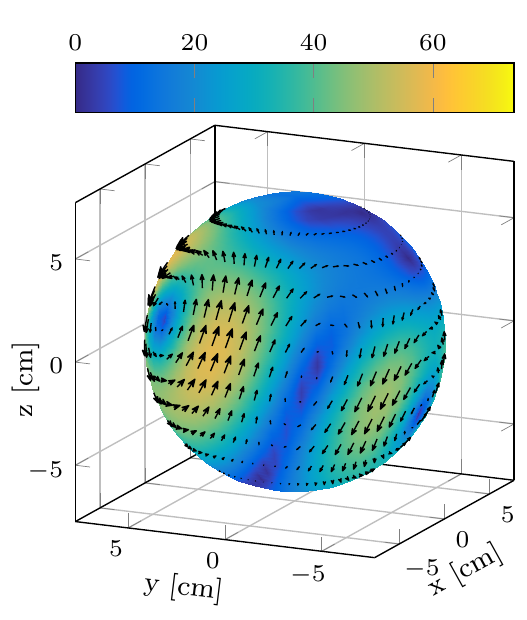}
  \subcaption{Difference of vector spline reconstruction (\SI{10}{\percent} noise)}
\end{subfigure}
\begin{subfigure}[t]{0.31\textwidth}
  \includegraphics[width=\figurewidth]{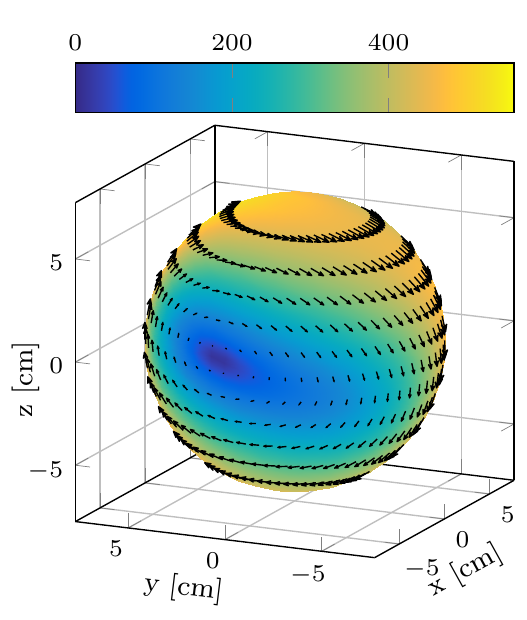}
  \subcaption{Reconstruction via vector splines (\SI{10}{\percent} noise)}
\end{subfigure}
\caption{Deviation of the exact current to the converted reconstruction of $J$ belonging to the scalar spline reconstruction of $A^{(1)}$ plotted in Fig.~\ref{fig:exactSolutionA} (left column), 
to the direct reconstruction by means of the vector spline method (middle column) and the reconstruction achieved by the vector splines (right column) depending on $2\%$ noise (top row) and $10\%$ noise (bottom row). For a better comparability, the colorbars in the left and middle column were chosen equally.}\label{fig:vectorMEGspline}
\end{figure}

We tested the vector spline method not only for the inverse MEG problem but also for the inverse EEG problem. 
In this setting, we also generated synthetic data by using a linear combination of two splines with $h_{\text{data}} = 0.64$.
Having the inverse crime in mind, we used the sequence $\kappa_{n,\text{data}}^{-2} \coloneqq n h^n$ for generating the data and $\kappa_n^{-2} \coloneqq h^n$ with $h=0.85$ for the inversion. 
As an additional obstacle, we located the test splines in the junction from the data gap area to the area covered by sensors. 
The exact solution is plotted in Fig.~\ref{fig:ExactEEG}.

\begin{table}
\begin{tabular}{r r r r r r}\toprule
 noise level  & 0\%             & 1\%             & 2\%             & 5\%            & 10\% \\\midrule
rel. residual & \num{3.0514e-09}& \num{0.0018}    & \num{0.0268}    &\num{0.0415}    & \num{0.0635}\\
rel. NRMSE $J$   & \num{0.0133} & \num{0.0271}    & \num{0.0427}    &\num{0.0512}    & \num{0.0869} \\[1ex]
$\lambda$     & \num{0.0233}& \num{23.2964}& \num{232.9644}&\num{232.9644}& \num{186.3716}\\\bottomrule
\end{tabular}
\caption{Numerical results for the EEG synthetic test case achieved by the vector spline method}\label{tab:ResultsEEG}
\end{table}

Even though the synthetic setting is more difficult to solve as in the inverse MEG test case, the results are comparably satisfying with respect to accuracy and stability during increasing noise level, Tab.~\ref{tab:ResultsEEG}.
Especially within small noise levels, one can see in the deviation plots Fig.~\ref{fig:vectorEEGspline} (B) to (E) that the reconstruction on the sensor covered hemisphere is more accurate than in the data gap area, which was expected. 
For \SI{10}{\percent} noise on the data this effect is smoothed away, since a higher regularization was necessary to handle the noise. 
However, even for \SI{10}{\percent} the reconstruction covers the active regions very satisfyingly and is not over-fitting the data, see Fig.~\ref{fig:ReconstEEG10} .

\begin{figure}
\setlength{\figurewidth}{0.33\textwidth}\centering
\begin{subfigure}[t]{0.31\textwidth}
  \includegraphics[width=\figurewidth]{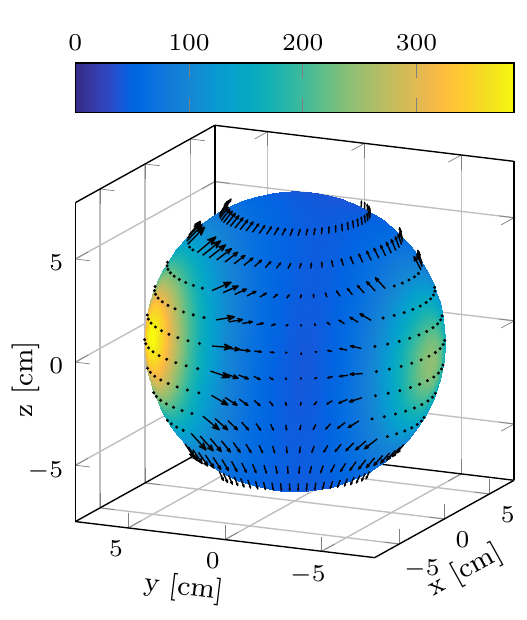}
  \subcaption{Exact spline based test current}\label{fig:ExactEEG}
\end{subfigure}
\begin{subfigure}[t]{0.31\textwidth}
  \includegraphics[width=\figurewidth]{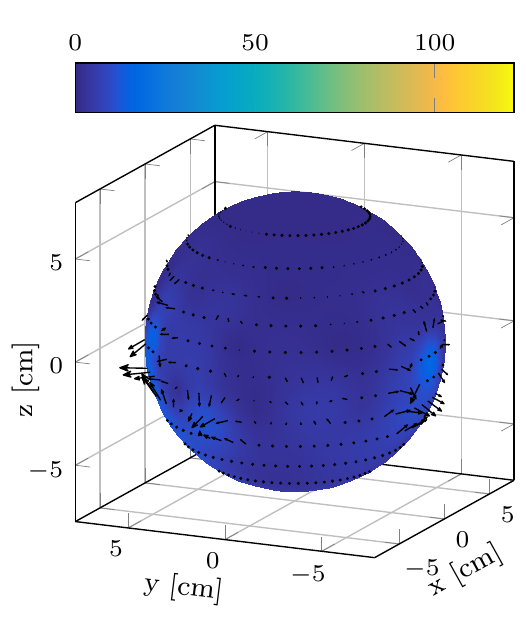}
  \subcaption{Deviation (\SI{0}{\percent} noise)}
\end{subfigure}
\begin{subfigure}[t]{0.31\textwidth}
  \includegraphics[width=\figurewidth]{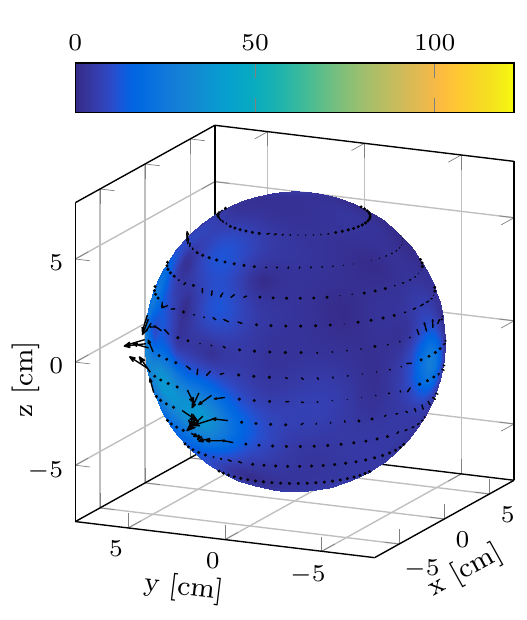}
  \subcaption{Deviation (\SI{1}{\percent} noise)}
\end{subfigure}
\begin{subfigure}[t]{0.31\textwidth}
  \includegraphics[width=\figurewidth]{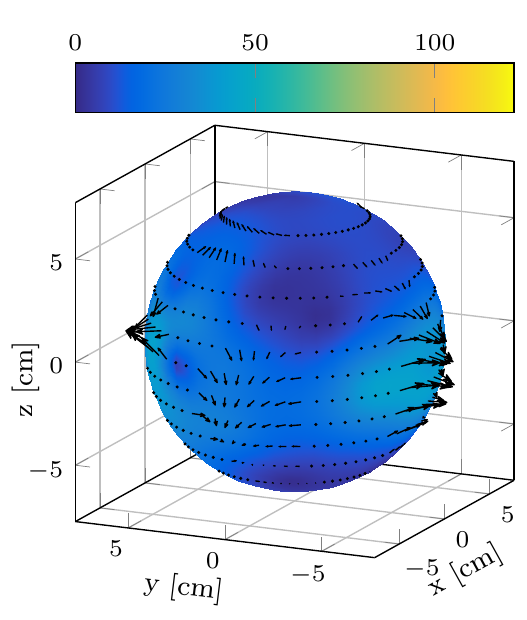}
  \subcaption{Deviation (\SI{5}{\percent} noise)}
\end{subfigure}
\begin{subfigure}[t]{0.31\textwidth}
  \includegraphics[width=\figurewidth]{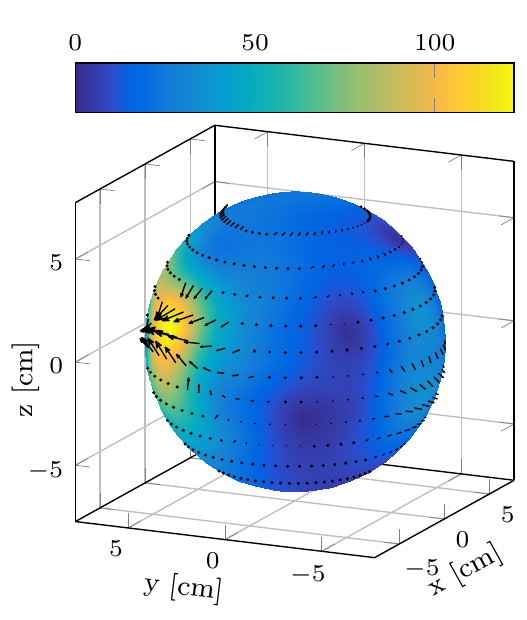}
  \subcaption{Deviation (\SI{10}{\percent} noise)}
\end{subfigure}
\begin{subfigure}[t]{0.31\textwidth}
  \includegraphics[width=\figurewidth]{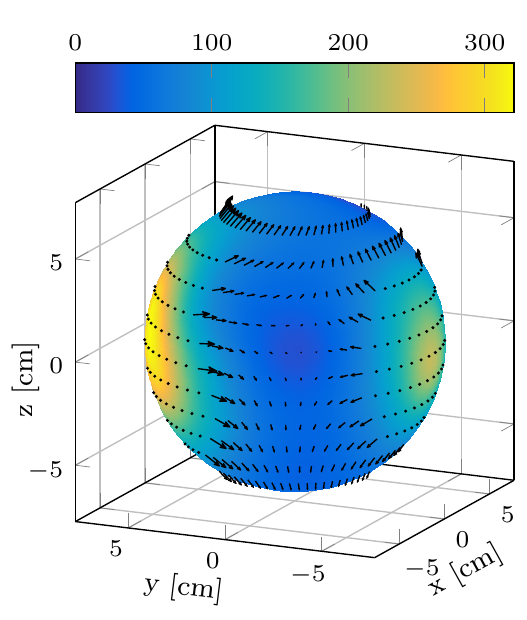}
  \subcaption{Reconstruction  (\SI{10}{\percent} noise)}\label{fig:ReconstEEG10}
\end{subfigure}
\caption{Absolute deviations and reconstruction of neuronal current $J$ from synthetic spline based EEG data with \SI{0}{\percent}, \SI{5}{\percent}, and \SI{10}{\percent} noise by using the vector spline method}\label{fig:vectorEEGspline}
\end{figure}

Within our synthetic test cases for the inverse MEG and EEG problem, we have seen that
\begin{itemize}
   \item the scalar spline method as well as the vector spline method is a  stable method for the solution of ill-posed inverse problems;
   \item the spline methods are fast: the average of \num{700} runs of the vector spline method including parameter choice methods and plotting for the inverse MEG problem yield \SI{0.5654(261)}{\second} and \SI{0.4285(69)}{\second} for the scalar spline method.
   \item both methods are able to handle irregularly distributed and noisy data;
   \item both methods keep the relative residual significantly below the noise level without over-fitting the reconstruction;
   \item the reconstructions coincides well with the exact solution even in the presence of higher noise levels;
   \item most deviations are located in the data gap area which was expected due to the lack of information in this area;
   \item the active regions are covered well and the amplitudes are reproduced accurately in the non-noisy and low-noisy case (up to \SI{5}{\percent}) and very satisfactorily in the case of \SI{10}{\percent} noise on the data;
   \item during the transition from the scalar reconstruction to the vector-valued one, reconstruction errors can be propagated.
 \end{itemize} 
Summarizing, we can say that even though the scalar spline method yields very good results during the synthetic test case, the vector method should be preferred if one is interested in a reconstruction of the vector-valued current.
Due to the vector spline method, quality of the reconstruction can be gained which manifests in the relative NRMSE as well in the deviation plots. 
Especially with increasing noise level, the vector spline method is more robust than the scalar spline method for solving the vector-valued inverse problem. 

Due to very good numerical results throughout our synthetic test cases, we test the numerical method with real data. 
To the knowledge of the author, it is the first time that a reproducing kernel based spline method is applied to real magneto-electroencephalography data, even though it yielded promising results for the scalar spline approach in \cite{Fokas2012}.

In order to compare the reconstruction of real data with former approaches, we use the same set of real data as in \cite{Leweke2018}. 
Therein, the regularized orthogonal functional matching pursuit (ROFMP) algorithm was used for the inversion and yielded good and reasonable results. 
The final results are also shown in the right columns of Fig.~\ref{fig:realEEG}  and Fig.~\ref{fig:realMEG} for the sake of comparison. 
For generating the data, a human participant wearing an EEG sensor cap was placed into an Elekta Neuromag\textsuperscript{\textregistered}, \cite{Elekta2005}, MEG device at the MRC Cognition and Brain Sciences Unit, Cambridge, UK.
During the measurements, a visual stimulus in form of a checkered pattern is presented to the participant in the right visual hemi-field.
After a delay of approximately \SI{80}{\milli\second}, the brain activity increases, presumably originating from visual areas in the back of the brain. 
For the inversion of both data types, we chose the measurement corresponding to \SI{89}{\milli\second}.
For the visualization of the reconstruction, we plot the neuronal current onto a sphere inside the cerebrum and show two views of the sphere, a front view and a back view.
Since the optical nerve fibers associated to the nasal side of the retinas cross each other in the optic chiasm, the brain activity should be maximal at the contralateral visual cortex, see \cite{Kandel2013}.
More precisely, we expect the main activity in the left visual cortex (i.e. contralateral to the side of visual stimulation) which is located at the back of the brain. 
More precisely, the brain activity should be bipolar in the case of the surface EEG recordings, that is a positive brain activity at the back results and a smaller negative brain activity at the front.

In the real data situation, no exact solution is known. 
Hence, we cannot use the NRMSE as a parameter choice method, such that the approximation values substitute here the deviations values.
As also seen during the synthetic test cases, the automatic L-curve method was chosen by our criterion for determining the regularization parameter. 

In the case of the reconstruction from the EEG measurements, the approximated neuronal current serves all the expectations, Fig.~\ref{fig:realEEG}. 
The order of magnitude of the activity is around \SI{4.5}{\nano\ampere\metre^{-2}}, which is quite realistic. 
In addition, the reconstructed activity is mainly located in the area of the left visual cortex with a smaller activity in the right visual cortex, Fig.~\ref{fig:EEGback}. 
The current is pointing inside the cerebrum.
Besides the reconstruction, we plotted a scatter plot of the absolute values of the measurements. 

\begin{figure}
\setlength{\figurewidth}{0.33\textwidth}\centering
\begin{subfigure}[t]{0.33\textwidth}
  \includegraphics[width=\figurewidth]{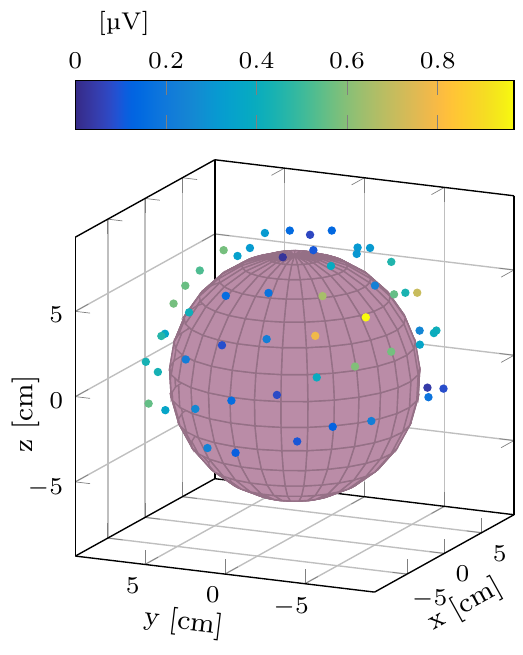}
  \subcaption{Back view of electric potential measurements absolute values}\label{fig:EEGscatterback}
\end{subfigure}
\begin{subfigure}[t]{0.31\textwidth}
  \includegraphics[width=\figurewidth]{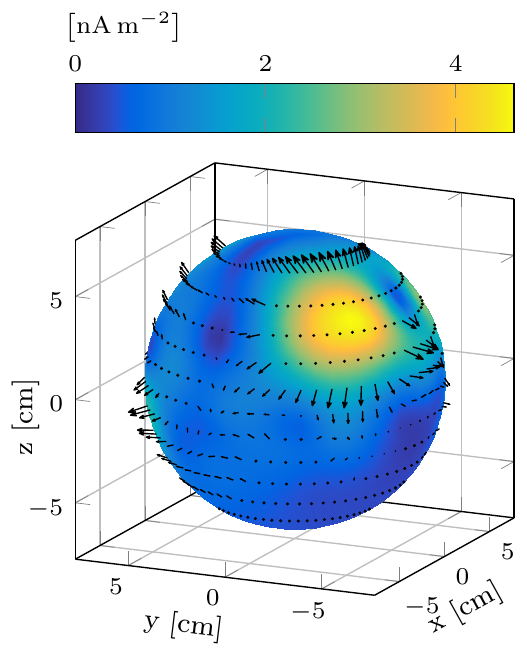}
  \subcaption{Back view of vector spline reconstruction}\label{fig:EEGback}
\end{subfigure}\hspace*{0.02\textwidth}
\begin{subfigure}[t]{0.31\textwidth}
  \includegraphics[width=\figurewidth]{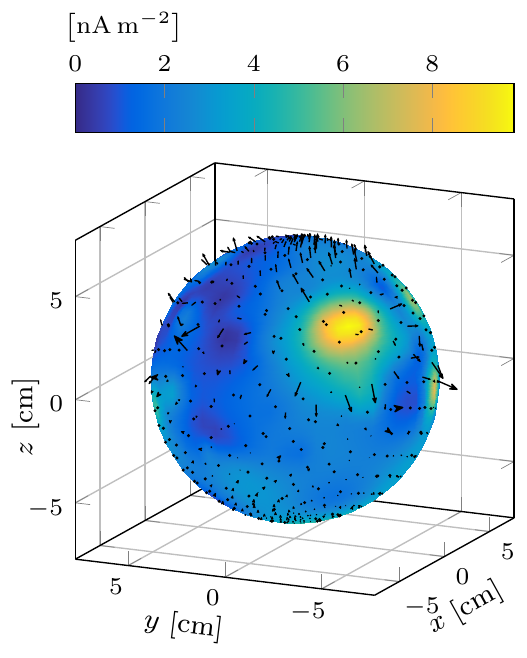}
  \subcaption{Back view of ROFMP reconstruction}\label{fig:EEGbackRFMP}
\end{subfigure}
\begin{subfigure}[t]{0.33\textwidth}
  \includegraphics[width=\figurewidth]{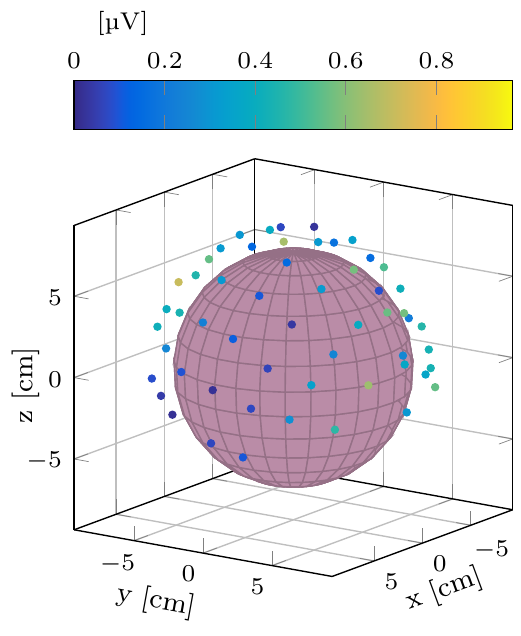}
  \subcaption{Front view of electric potential measurements absolute values}\label{fig:EEGscatterfront}
\end{subfigure}
\begin{subfigure}[t]{0.31\textwidth}
  \includegraphics[width=\figurewidth]{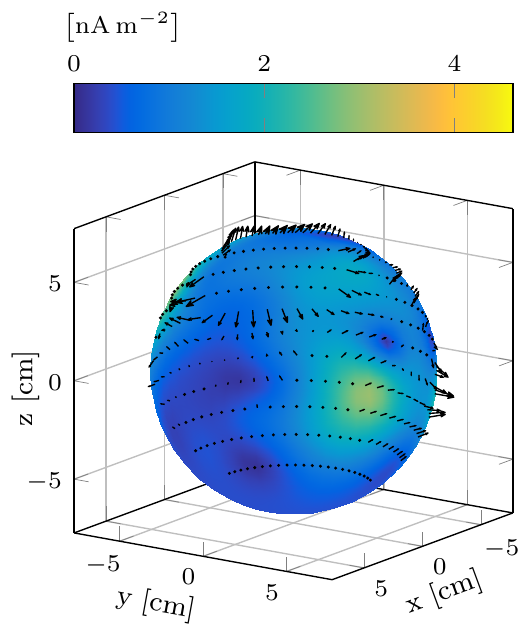}
  \subcaption{Front view of vector spline reconstruction}\label{fig:EEGfront}
\end{subfigure}\hspace*{0.02\textwidth}
\begin{subfigure}[t]{0.31\textwidth}
  \includegraphics[width=\figurewidth]{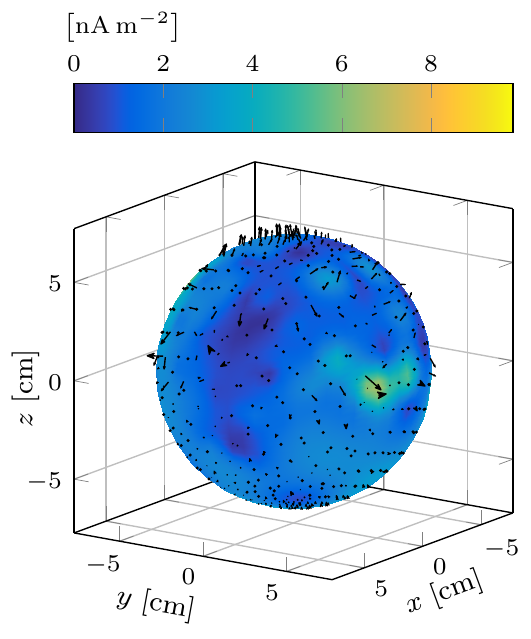}
  \subcaption{Front view of ROFMP reconstruction}\label{fig:EEGfrontRMP}
\end{subfigure}
\caption{Reconstruction of the neuronal current from real EEG electric potential measurements (left) on the scalp achieved via vector spline method (middle) and ROFMP (right)}\label{fig:realEEG}
\end{figure}

In the case of the inversion of the MEG measurements, we can see similar results, Fig.~\ref{fig:realMEG}.
The activity is mainly located in the left visual cortex and the order of magnitude seems to be realistic. 

In both cases, the vector splines method yields reasonable results.
Compared to the results achieved by the ROFMP by former research plotted on the right column in Figures \ref{fig:realEEG} and \ref{fig:realMEG}, the results of the vector spline method seem to be less over-fitted on the one hand and smoother and less blurred and unstructured on the other hand.
In contrast to the ROFMP who had a major problem with artifacts in the area of the facial data gap during the inversion of the MEG data, Fig.~\ref{fig:MEGfrontROFMP}, the vector spline method handles the data gap very well, Fig.~\ref{fig:MEGfrontVec}.
\begin{figure}
\setlength{\figurewidth}{0.33\textwidth}\centering
\begin{subfigure}[t]{0.31\textwidth}
  \includegraphics[width=\figurewidth]{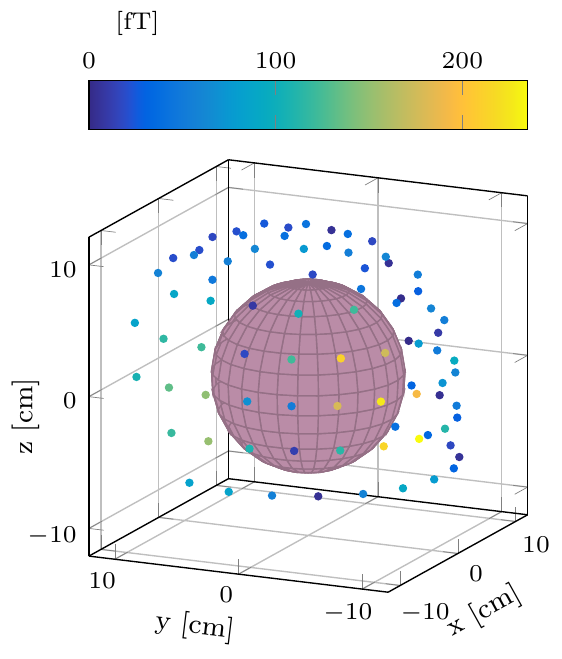}
    \subcaption{Back view of magnetic flux density measurements absolute values}\label{fig:MEGscatterback}
\end{subfigure}\hspace*{0.02\textwidth}
\begin{subfigure}[t]{0.31\textwidth}
  \includegraphics[width=0.98\figurewidth]{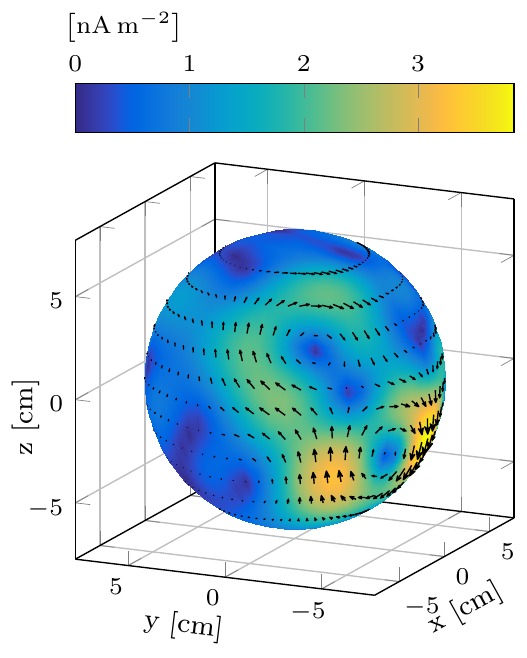}
  \subcaption{Back view of vector spline reconstruction}\label{fig:MEGback}
\end{subfigure}
\begin{subfigure}[t]{0.31\textwidth}
\includegraphics[width=\figurewidth]{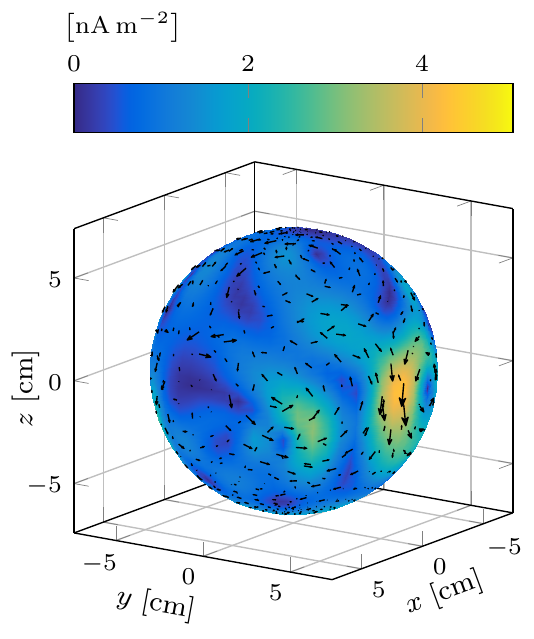}
   \subcaption{Back view of ROFMP reconstruction}\label{fig:MEGbackRFMP}
\end{subfigure}
\begin{subfigure}[t]{0.31\textwidth}
  \includegraphics[width=0.98\figurewidth]{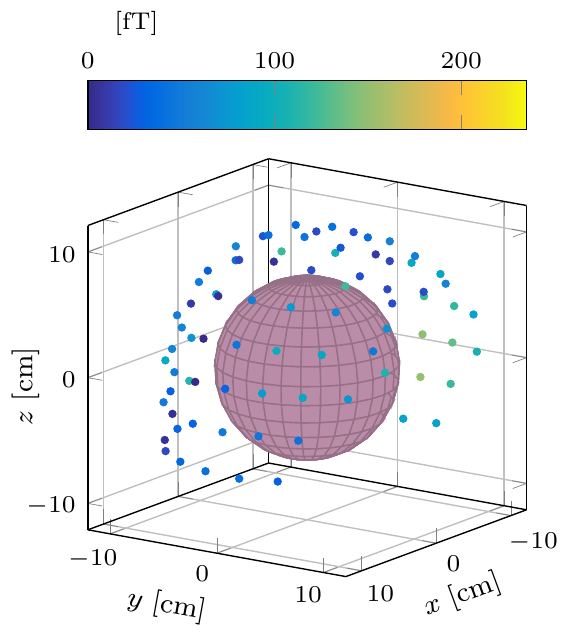}
  \subcaption{Front view of magnetic flux density measurements absolute values}\label{fig:MEGscatterfront}
\end{subfigure}\hspace*{0.02\textwidth}
\begin{subfigure}[t]{0.31\textwidth}
  \includegraphics[width=\figurewidth]{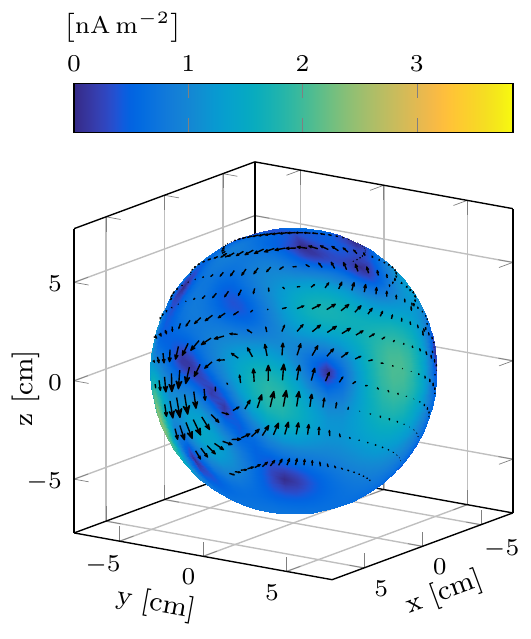}
  \subcaption{Front view of vector spline reconstruction}\label{fig:MEGfrontVec}
\end{subfigure}
\begin{subfigure}[t]{0.31\textwidth}
\includegraphics[width=\figurewidth]{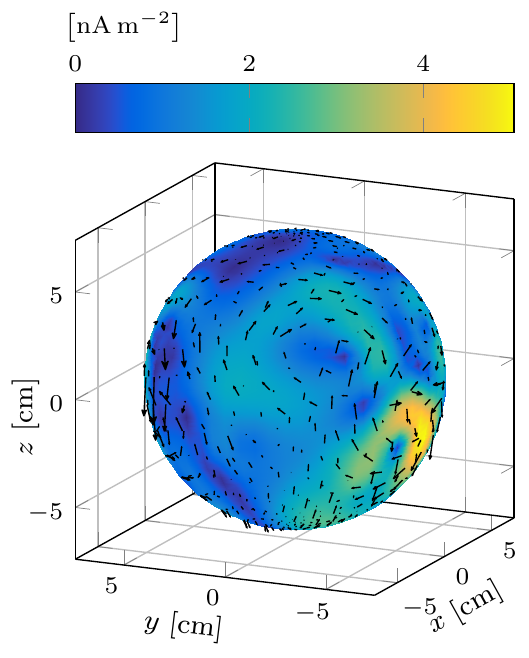}
   \subcaption{Front view of ROFMP reconstruction}\label{fig:MEGfrontROFMP}
\end{subfigure}
\caption{Reconstruction of the neuronal current from real MEG magnetic flux density measurements (left) on the scalp achieved via vector spline method (middle) and ROFMP (right)}\label{fig:realMEG}
\setlength{\figurewidth}{0.35\textwidth}\centering
\end{figure}

\section{Conclusion}\label{sec:conclude}

In order to solve functional inverse problems by means of an optimize-then-discretize approach, 
stable and robust numerical methods are required.
These need to handle difficulties caused by the ill-posedness of the problems.
More precisely, (infinite dimensional) null spaces of the related forward operators imply 
that the source contains silent parts.
Besides appropriate non-uniqueness constraints, the numerical method needs to be aware which parts of the source induce measurable effects and which do not. 
In addition, instabilities caused by decreasing singular values towards zero need to be faced
via regularization techniques. 

Within this paper we introduced a novel numerical method referred to as vector-valued spline method
in order to solve ill-posed functional inverse problems over the ball which serves these two requirements.
The vector-valued spline method is, likewise its scalar-valued relative, constructed for a particular application.

Here, we introduced the inverse MEG and EEG problem of reconstructing parts of the neuronal current from electro-magnetic quantities as a possible application.
The problem was analyzed in detail in \cite{Leweke2020}. 
The linear and continuous functionals mapping the sought source (neuronal current) onto the measurable effects (electric potential and magnetic flux density) were recapitulated
as well as related spherical orthonormal basis function. 
The vector-valued function space $\mathrm{L}_s(B_R,\mathbb{R}^3)$ was decomposed by means of these orthonormal basis 
functions into three orthogonal function spaces covering different (spherical) directions.
These directions conserve the complementarity of the operator null spaces. 
Afterwards, reproducing kernel Hilbert spaces were constructed for each direction while incorporating summability conditions 
for the existence of the occurring series representations.
Via the tensor product of the orthonormal basis functions, the reproducing kernels were built to 
fulfill a kind of reproducing property which justified the denotation.
Combining the tensor-valued reproducing kernel with an adequate operation of the linear functionals 
given by the application, the vector-valued spline functions were defined.

As it is typical for (spherical) approximation splines, the constructed vector-valued splines satisfy several useful properties, 
such as two minimum-norm properties. 
In addition, the Shannon Sampling Theorem as well as the Spline Approximation Theorem known from 
scalar-valued splines could be conserved in the vector-valued case. 
Here, the major advantage of the vector splines is revealed: The minimization of the regularized Tikhonov functional 
over the RKHS reduces to the resolution of a finite dimensional system of linear equations if the bounded linear functionals are linearly independent. 
In addition, we proved that the approximation spline converges towards the solution 
as the numbers of measurements increases. 

For our numerical tests, we used the minimum-norm condition to overcome the non-uniqueness.
Within this additional uniqueness constraint, only parts of the neuronal current which are in the 
orthogonal complement of the operator null space are reconstructed.
Within our synthetic test cases, we used a real sensor position distribution which comes along
with only few and irregularly distributed measurement positions. 
In addition, we constructed synthetic data sets caused by (global) orthonormal basis functions and (localized) splines.
In the spline case, we produced the data from splines and methods which differ from those used for the inversion in order to 
 avoid the inverse crime. 
In addition, for calculating the occurring forward functionals we used a second numerical method.
This one is grounded on numerical differentiation and integration instead of the singular value decomposition which is used in the inversion.
Besides we noised the data with additive Gaussian white noise up to $\SI{10}{\percent}$.

The intention of our numerical tests was to answer the following three questions:
\begin{itemize}
  \item Does the vector spline method produce a reasonable and correct reconstruction?
  \item Is the vector spline method stable with respect of increasing noise level?
  \item Does the vector spline method yield a benefit by avoiding the intermediate calculation of a scalar component function?
\end{itemize}

In the case of the inverse MEG and EEG problem, a scalar-valued formulation of the problem exists which was 
previously solved via a scalar spline approach, \cite{Fokas2012}. 
Hence, we constructed the MEG synthetic test case in such a way that we were able to 
compare the scalar spline method with the vector-valued approach. 
Unfortunately for the EEG case, it was shown in \cite{Leweke2020} that there exists no case where a neuronal current
can satisfy the minimum-norm condition and the conditions required in \cite{Fokas2012} simultaneously.

Since the vector spline approach is a regularization method, we used five different parameter choice methods for the 
determination of the regularization parameter. 
Namely, the automatic and manual L-curve, the discrepancy principle, the quasi-optimality criterion, and 
the generalized cross validation. 
The results achieved via these methods are compared by means of the normalized root mean square error, 
which could be calculated in the case of the synthetic tests.

Based on our numerical tests, we come to the following conclusion:
\begin{itemize}
  \item The spline method based on reproducing kernels is a fast method if the singular value decomposition is used for the calculation of the forward functionals.
  In this case, building the spline matrix requires \SI{0.4285(69)}{\second} compared to nearly \SI{18}{\day} in the case of the direct method based on numerical integration and differentiation.
  Afterwards, the inversion with \num{700} different regularization parameters including the evaluation of the five parameter-choice methods and the visualization of the reconstruction took \SI{0.5654(261)}{\second} (scalar spline) and \SI{0.4285(69)}{\second} (vector spline).
  \item The reconstructions belonging to the regularization parameter obtained by the L-curve method yielded the smallest NRMSE in most cases. 
  \item The scalar as well as the vector spline method are able to handle the challenging data situation and the irregularly distributed measurement positions. 
  \item Both spline methods are robust with respect to increasing noise levels: the relative residual is significantly below the noise level without over-fitting the reconstruction. In addition, the major deviations are located in the data gap area.
  \item During the transition of the scalar solution obtained by the scalar spline approach to a reconstruction of the vector-valued neuronal current, reconstruction errors are propagated. 
  Hence, a direct reconstruction of the vector-valued current produces more accurate and stable reconstructions than a transferred scalar reconstruction. 
\end{itemize}

After having verified the functionality of the vector spline method in several synthetic test cases,
we reconstructed the neuronal current from sets of real data.
On the reconstruction of the neuronal current from these data sets, several physiological expectations exist. 
These data sets were also inverted in \cite{Leweke2020} by means of the (orthogonal) regularized functional matching pursuit (RFMP) algorithm.
In both applications, the vector spline method yields reasonable results which satisfied physiological expectations.
Compared to former reconstructions of the RFMP, the vector spline method produced a solution which seems to be less over-fitted and avoids artifacts in the area of the facial data gap.

An extension of the vector spline method for other regularization types is a subject of current research. 
The presented method only allows cost functionals of the Tikhonov functionals which are based on a RKHS, or $\mathrm{L}_2(B_R,\mathbb{R}^3)$, respectively. 
A regularization with generalized total variation penalty terms is of great interest in our application, since particular stimuli produce highly localized brain activity.
Furthermore, our numerical tests suggest a dependency of the optimal free parameter $h$ to the mean distance of the sensor positions. 
This can be evaluated in future research.
Finally, the underlying multiple-shell model can be enhanced by passing over to more realistic brain geometries. 
In order to incorporate this enhancement into the vector-spline approach, an appropriate singular value decomposition needs to be found. 

\section*{Acknowledgements}
S. L. and V. M. gratefully acknowledge the support by the German Research Foundation (DFG), project MI 655/10-1 and MI 655/10-2.
They also appreciate the possibility to use the OMNI parallel computing cluster of the University of Siegen, Germany. 

Authors' contribution: The research was initially carried out for the PhD thesis of S.~L. and continued afterwards. 
O.~H. provided the data and contributed medical interpretation.  
V.~M. is the principal investigator and supervisor of the PhD thesis. 

\FloatBarrier
\newpage
\appendix

\section{Supplemantary Calculations for MEG and EEG Vector Splines} 

For the summability condition of the vector spline method, we need to calculate the values of $B_{m,n}^{(i)}$ for all $m\in\mathbb{N}_0$, $n \in \mathbb{N}$:
\begin{equation}\label{app:Bmn}
  \begin{aligned}
    B_{m,n}^{(i)} &= \sup_{x\in B_{\varrho_0}} \sum_{j=1}^{2n+1} \abs{h_{m,n,j}^{(i)}(x)} \\
    &=\sup_{x\in B_{\varrho_0}}  \left(\frac{4m+2t_n^{(i)}+3}{\varrho_0^3} \left(\frac{r}{\varrho_0}\right)^{2t_n^{(i)}} \left(P_{m}^{\left(0,t_n^{(i)}+1/2\right)}\left(2\frac{r^2}{\varrho_0^2}-1\right)\right)^2 \sum_{j=1}^{2n+1} \abs{\tilde{y}_{n,j}^{(i)}(\xi)}^2\right)\\
    &=\sup_{x\in B_{\varrho_0}}  \left(\frac{\left(4m+2t_n^{(i)}+3\right)(2n+1)}{4\pi\varrho_0^3} \left(\frac{r}{\varrho_0}\right)^{2t_n^{(i)}} \left(P_{m}^{\left(0,t_n^{(i)}+1/2\right)}\left(2\frac{r^2}{\varrho_0^2}-1\right)\right)^2\right)\\
    &=\frac{\left(4m+2t_n^{(i)}+3\right)(2n+1)}{4\pi\varrho_0^3} \sup_{r\in [0,\varrho_0]}  \left( \left(\frac{r}{\varrho_0}\right)^{2t_n^{(i)}} \left(P_{m}^{\left(0,t_n^{(i)}+1/2\right)}\left(2\frac{r^2}{\varrho_0^2}-1\right)\right)^2\right)\\
    &=\frac{\left(4m+2t_n^{(i)}+3\right)(2n+1)}{4\pi\varrho_0^3} \binom{m + t_n^{(i)} + 1/2}{m}^2
  \end{aligned}
\end{equation}
where we use the Addition Theorem, \ref{eq:AddThmAbs}, in the second step and the maximal values of the Jacobi polynomials, \cite[Ch. II.7]{Novikoff1938}, in the last step.

\section{Supplementary Calculations for Implementation}\label{sec:AppSupplImplement}

For an efficient implementation of the scalar as well as the vector MEG and EEG problem, some identities need to be calculated further. 
Especially, for building the spline matrices, fast and accurate computations of the occurring series are required, since manual summation can be quite slow and inaccurate. 
Hence, the main aim of this section is to get rid of the summations over $j$ and achieve some expressions based on Legendre polynomials instead. 
Afterwards, the Legendre series can be efficiently evaluated using the Clenshaw algorithm.

In the case of the scalar spline matrix, each entry contains the term
\begin{equation*}
  \sum_{j=1}^{2n+1} \left(\nu(y_k)\cdot\tilde{y}_{n,j}^{(1)}(\eta_k)\right) \left(\nu(y_l)\cdot \tilde{y}_{n,j}^{(1)}(\eta_l)\right) ,
\end{equation*}
see \eqref{eq:scalarSplineMatrixEntry}. 
Although, it consists of vector spherical harmonics, the addition theorem cannot be applied directly, due to the Euclidean inner product with the normal vector. 
By using \eqref{eq:DefiVecSph} and \eqref{eq:Defioi}, we split up the vector spherical harmonics summands. 
After using some vector calculus properties and product rules, we can finally apply the addition theorems and achieve
\begin{equation}\label{eq:scalarSplineMatrixEntryImpl}\begin{multlined}
\sum_{j=1}^{2n+1} \left(\nu(y_k)\cdot\tilde{y}_{n,j}^{(1)}(\eta_k)\right) \left(\nu(y_l)\cdot \tilde{y}_{n,j}^{(1)}(\eta_l)\right) = %
\frac{n+1}{4\pi} (\nu(y_k)\cdot \eta_k)(\nu(y_l)\cdot \eta_l) P_n(\eta_k\cdot\eta_l)\\
- \frac{1}{4\pi} (\nu(y_l)\cdot (\eta_l (\nu(y_k)\cdot \eta_l - (\eta_k\cdot\eta_l)\eta_k))P_n^\prime(\eta_k\cdot\eta_l) 
- \frac{1}{4\pi} (\nu(y_k)\cdot \eta_k) (\nu(y_l) \cdot (\eta_k-(\eta_l\cdot\eta_k)\eta_l)) P_n^\prime(\eta_l\cdot\eta_k) \\
+ \frac{1}{4\pi(n+1)}\nu(y_k)  \cdot [ \nu(y_l) \wedge (P_n^{\prime\prime}(\eta_k\cdot\eta_l)(1-(\eta_k\cdot\eta_l)^2) - (\eta_k\cdot\eta_l)P_n^{\prime}(\eta_k\cdot\eta_l))(\eta_l\wedge\eta_k) \\
   + P_n^{\prime}(\eta_k\cdot\eta_l)(\nu(y_l) - (\nu(y_l)\cdot\eta_k)\eta_k - (\nu(y_l)\cdot \eta_l)\eta_l + (\eta_k\cdot\eta_l)(\nu(y_l)\cdot \eta_k)\eta_l)\\
   + (\eta_k-(\eta_k\cdot\eta_l)\eta_l)P_n^{\prime\prime}(\eta_k\cdot\eta_l)(\eta_l\cdot\nu(\eta_l) - (\nu(y_l)\cdot\eta_k)(\eta_k\cdot\eta_l)) ].
\end{multlined}
\end{equation}
This formula can be efficiently implemented via Clenshaw's algorithm.

\begin{proof}[Proof of Thm.~\ref{thm:continuousVecFunctionals}]
  The linearity is clear, due to the linearity of the integration and differentiation.
  For all $f \in \mathscr{H}^{(3)}$ with $\norm[\mathscr{H}^{(3)}]{f} =1$ the estimate
  \begin{align*}
    \abs{\mathcal{A}_{\mathrm{M}}^{k} f}
      &= \mu_0 \abs{\sum_{n=1}^\infty \sum_{j=1}^{2n+1}  \sqrt{\frac{n \varrho_0}{(2n+1)(2n+3)}}  \scalar[\mathrm{L}^2(B_{\varrho_0},\mathbb{R}^3)]{f}{\tilde{g}^{(3)}_{0,n,j}(\varrho_0;\cdot)} \left(\frac{\varrho_0}{s_k}\right)^{n+1}\frac{1}{s_k} \nu(y_{k}) \cdot \tilde{y}_{n,j}^{(1)}(\eta_k) } \\
      &\leq \mu_0 \left(\sum_{n=1}^\infty \sum_{j=1}^{2n+1} \left(\kappa_{n}^{(3)}\right)^{-2} \frac{n }{(2n+1)(2n+3)\varrho_0}   \left(\frac{\varrho_0}{\varrho_L}\right)^{2n+2} \abs{\tilde{y}_{n,j}^{(1)}(\eta_k)}^2\right)^{1/2} \\
      & \qquad \times \left(\sum_{n=1}^\infty \sum_{j=1}^{2n+1} \left(\kappa_{n}^{(3)}\right)^2 \scalar[\mathrm{L}^2(B_{\varrho_0},\mathbb{R}^3)]{f}{\tilde{g}^{(3)}_{0,n,j}(\varrho_0;\cdot)} \right)^{1/2}  \\
      &\leq \frac{\mu_0}{\sqrt{4\pi\varrho_0}} \left(\sum_{n=1}^\infty \left(\kappa_{n}^{(3)}\right)^{-2} \frac{n}{2n+3} \left(\frac{\varrho_0}{\varrho_L}\right)^{2n+2} \right)^{1/2} \norm[\mathscr{H}^{(3)}]{f} < \infty
  \end{align*}
  holds true. Here, we used the addition theorem for Edmonds vector spherical harmonics, see \eqref{eq:AddThmAbs}, in the last step.
  In the EEG case, we obtain analogously
  \begin{align*}
    \abs{\mathcal{A}_{\mathrm{E}}^k f} &= \abs{\sum_{n=1}^\infty \sum_{j=1}^{2n+1} \frac{1}{\sqrt{n\varrho_0}} \beta^{(L)}_n \scalar[\mathrm{L}^2(B_{\varrho_0},\mathbb{R}^3)]{f}{\tilde{g}_{0,n,j}^{(i)}(\varrho_0;\cdot)} \left({(n+1)} \left(\frac{s_k}{\varrho_L}\right)^{2n+1} + {n} \right) \left(\frac{\varrho_0}{s_k}\right)^{n+1} Y_{n,j}(\eta_k)} \\
    &\leq \frac{1}{\sqrt{4\pi\varrho_0}} \left(\sum_{n=1}^\infty \left(\kappa_{n}^{(2)}\right)^{-2} \frac{(2n+1)^3}{n} \left(\beta^{(L)}_n\right)^2  \left(\frac{\varrho_0}{\varrho_{L-1}}\right)^{2n+2} \right)^{1/2} \norm[\mathscr{H}^{(2)}]{f}  \\
    &\leq \frac{C}{\sqrt{4\pi\varrho_0}} \left(\sum_{n=1}^\infty \left(\kappa_{n}^{(2)}\right)^{-2} \frac{(2n+1)^2}{n} \left(\frac{\varrho_0}{\varrho_{L-1}}\right)^{2n+2} \right)^{1/2} \norm[\mathscr{H}^{(2)}]{f} < \infty.
  \end{align*}
  In the last step, we used the fact that $\sup_{n\in\mathbb{N}} (2n+1) \left(\beta^{(L)}_n\right)^2 \leq C$ with a constant $C > 0$, see \cite[Cor.~4.3]{Leweke2018}.
\end{proof}

\begin{proof}[Proof of Lem.~\ref{lem:reprVecRKHSKernels}]
  For the proof, we use the precise representation of the orthonormal basis function and the series representation of the forward solution of $J$ in \eqref{eq:funcVector_MEG}. 
  Thus, in the MEG case we achieve
  \begin{align}
    &\phantom{=\ }\mathcal{A}_{\mathrm{M}}^{k}  \mathfrak{k}^{(3)}(\cdot,x) \\
    &= \sum_{n=1}^\infty \sum_{j=1}^{2n+1} \left(\kappa_{n}^{(3)}\right)^{-2} \tilde{g}_{0,n,j}^{(3)}(\varrho_0;x)  \mathcal{A}_{\mathrm{M}}^k \tilde{g}_{0,n,j}^{(3)}(\varrho_0;\cdot) \notag\\
    &= -\mu_0 \sum_{n=1}^\infty \sum_{j=1}^{2n+1} \left(\kappa_{n}^{(3)}\right)^{-2} \tilde{g}_{0,n,j}^{(3)}(\varrho_0;x) \sqrt{\frac{n \varrho_0}{(2n+1)(2n+3)}} \left(\frac{\varrho_0}{s_k}\right)^{n+1}\frac{1}{s_k} \nu(y_{k}) \cdot \tilde{y}_{n,j}^{(1)}(\eta_k) \notag\\
    &= -\mu_0 \sum_{n=1}^\infty \sum_{j=1}^{2n+1} \left(\kappa_{n}^{(3)}\right)^{-2} \frac{1}{\varrho_0} \left(\frac{r}{\varrho_0}\right)^{n}\tilde{y}_{n,j}^{(3)}(\xi) \sqrt{\frac{n}{2n+1}} \left(\frac{\varrho_0}{s_k}\right)^{n+1}\frac{1}{s_k} \nu(y_{k}) \cdot \tilde{y}_{n,j}^{(1)}(\eta_k) \notag\\
    &= -\mu_0 \sum_{n=1}^\infty \sum_{j=1}^{2n+1} \left(\kappa_{n}^{(3)}\right)^{-2}  \sqrt{\frac{n}{2n+1}}\left(\nu(y_{k}) \cdot \tilde{y}_{n,j}^{(1)}(\eta_k)\right)  \frac{r^n}{s_k^{n+2}} \tilde{y}_{n,j}^{(3)}(\xi). \notag
  \intertext{In order to get rid of the summation over $j$, the addition theorem needs to be applied. For this purpose, we split again the vector spherical harmonics and use some vector calculus and the product rule. Hence,}
   &\phantom{=\ }\mathcal{A}_{\mathrm{M}}^{k}  \mathfrak{k}^{(3)}(\cdot,x) \\
   &= -\mu_0 \sum_{n=1}^\infty \left(\kappa_{n}^{(3)}\right)^{-2} \frac{r^n}{s_k^{n+2}} \frac{\sqrt{n}}{2n+1} \notag\\
   &\phantom{=} \qquad \times
   \left(\sqrt{n+1}(\nu(y_k) \cdot \eta_k) \left(\sum_{j=1}^{2n+1} Y_{n,j}(\eta_k)\tilde{y}_{n,j}^{(3)}(\xi)\right) - \frac{1}{\sqrt{n+1}}\sum_{j=1}^{2n+1} (\nu(y_k)\cdot\nabla^*_{\eta_k} Y_{n,j}(\eta_k))\tilde{y}_{n,j}^{(3)}(\xi)\right) \notag\\
   &= -\frac{\mu_0}{4\pi} \sum_{n=1}^\infty \left(\kappa_{n}^{(3)}\right)^{-2} \frac{r^n}{s_k^{n+2}} \left((\nu(y_k) \cdot \eta_k) (\xi \wedge \eta_k) P_n^\prime(\xi \cdot \eta_k) \phantom{\frac{1}{{n+1}}}\right. \label{eq:proofA2}\\
   &\phantom{=} \qquad \times \left. -\frac{1}{{n+1}}\left( (\xi \wedge \eta_k) P_n^{\prime\prime}(\xi \cdot \eta_k)\left(\xi\cdot(\nu(y_k)-(\nu(y_k)\cdot\eta_k)\eta_k)\right) + P_n^\prime(\xi \cdot \eta_k) \xi \wedge (\nu(y_k) - (\nu(y_k)\cdot\eta_k)\eta_k) \right)\right) \notag
  \end{align}
  Similarly, but  with the addition theorem in the last step we obtain for the EEG case the representation
  \begin{align*}
    &\phantom{=\ } \mathcal{A}_{\mathrm{E}}^k \mathfrak{k}^{(2)}(\cdot,x) \\
    &= \sum_{n=1}^\infty \sum_{j=1}^{2n+1} \left(\kappa_{n}^{(2)}\right)^{-2} \tilde{g}_{0,n,j}^{(2)}(\varrho_0;x)  \mathcal{A}_{\mathrm{E}}^k \tilde{g}_{0,n,j}^{(2)}(\varrho_0;\cdot) \\
    &= \sum_{n=1}^\infty \sum_{j=1}^{2n+1} \left(\kappa_{n}^{(2)}\right)^{-2} \tilde{g}_{0,n,j}^{(2)}(\varrho_0;x)  \frac{1}{\sqrt{n\varrho_0}} \beta^{(L)}_n \left({(n+1)} \left(\frac{s_k}{\varrho_L}\right)^{2n+1} + {n} \right) \left(\frac{\varrho_0}{s_k}\right)^{n+1} Y_{n,j}(\eta_k) \\
    &= \sum_{n=1}^\infty \sum_{j=1}^{2n+1} \left(\kappa_{n}^{(2)}\right)^{-2} \sqrt{\frac{2n+1}{n\varrho_0^4}} \frac{r^{n-1}}{\varrho_0^{n-1}}  \beta^{(L)}_n \left({(n+1)} \left(\frac{s_k}{\varrho_L}\right)^{2n+1} + {n} \right) \left(\frac{\varrho_0}{s_k}\right)^{n+1} \tilde{y}_{n,j}^{(2)}(\xi) Y_{n,j}(\eta_k) \\
    &= \sum_{n=1}^\infty \sum_{j=1}^{2n+1} \left(\kappa_{n}^{(2)}\right)^{-2} \sqrt{\frac{2n+1}{n}} \frac{r^{n-1}}{s_k^{n+1}} \beta^{(L)}_n \left({(n+1)} \left(\frac{s_k}{\varrho_L}\right)^{2n+1} + {n} \right) \tilde{y}_{n,j}^{(2)}(\xi) Y_{n,j}(\eta_k) \\
    &= \frac{1}{4\pi}\sum_{n=1}^\infty \left(\kappa_{n}^{(2)}\right)^{-2} \frac{(2n+1)^{3/2}}{\sqrt{n}} \frac{r^{n-1}}{s_k^{n+1}} \beta^{(L)}_n \left({(n+1)} \left(\frac{s_k}{\varrho_L}\right)^{2n+1} + {n} \right) \tilde{p}_{n}^{(2)}(\xi;\eta_k).
  \end{align*}
\end{proof}

\begin{proof}[Proof of Thm.~\ref{thm:splineMatrixVector}]
  By means of the first transformation of the latter calculation, we can determine each entry of the MEG spline matrix by
  \begin{align*}
    &\phantom{=\ } \mathcal{A}_{\mathrm{M}}^l \mathcal{A}_{\mathrm{M}}^k \left(\mathfrak{k}^{(3)}(\cdot, \cdot)\right)\\
    &= \sum_{n=1}^\infty \sum_{j=1}^{2n+1} (\kappa_{n}^{(3)})^{-2}\left(\mathcal{A}_{\mathrm{M}}^l\tilde{g}_{0,n,j}^{(3)}(\varrho_0;\cdot)\right) \left(\mathcal{A}_{\mathrm{M}}^k \tilde{g}_{0,n,j}^{(3)}(\varrho_0;\cdot)\right) \\
    &= -\mu_0 
    \sum_{n=1}^\infty \sum_{j=1}^{2n+1} (\kappa_{n}^{(3)})^{-2}\left(\mathcal{A}_{\mathrm{M}}^l\tilde{g}_{0,n,j}^{(3)}(\varrho_0;\cdot)\right) 
    \sqrt{\frac{n \varrho_0}{(2n+1)(2n+3)}} \left(\frac{\varrho_0}{s_k}\right)^{n+1}\frac{1}{s_k} \nu(y_{k}) \cdot \tilde{y}_{n,j}^{(1)}(\eta_k) \\
    &= \frac{\mu_0^2}{\varrho_0} \sum_{n=1}^\infty \sum_{j=1}^{2n+1} (\kappa_{n}^{(3)})^{-2} 
    {\frac{n}{(2n+1)(2n+3)}} \left(\frac{\varrho_0^2}{s_ls_k}\right)^{n+2} \left(\nu(y_{l}) \cdot \tilde{y}_{n,j}^{(1)}(\eta_l)\right)\left( \nu(y_{k}) \cdot \tilde{y}_{n,j}^{(1)}(\eta_k)\right).
  \end{align*}
  For an efficient implementation, the summation over $j$ needs to be replaced by the expression stated in \eqref{eq:scalarSplineMatrixEntryImpl}.
  Consequently, an entry of the EEG spline matrix occurring in \eqref{eq:SplineMatrixVec} has for all $x$, $z \in B_{R}$ the representation
  \begin{align*}
    \mathcal{A}_{\mathrm{E}}^l \mathcal{A}_{\mathrm{E}}^k \left(\mathfrak{k}^{(2)}(\cdot, \cdot)\right) &= \sum_{n=1}^\infty \sum_{j=1}^{2n+1} (\kappa_{n}^{(2)})^{-2}\left(\mathcal{A}_{\mathrm{E}}^l\tilde{g}_{0,n,j}^{(2)}(\varrho_0;\cdot)\right) \left(\mathcal{A}_{\mathrm{E}}^k \tilde{g}_{0,n,j}^{(2)}(\varrho_0;\cdot)\right) \\
    &= \sum_{n=1}^\infty \left(\kappa_{n}^{(2)}\right)^{-2}  \sum_{j=1}^{2n+1} \frac{\left(\beta^{(L)}_n\right)^2}{{n\varrho_0}} \left(\frac{\varrho_0^2}{s_ls_k}\right)^{n+1}   \left({(n+1)} \left(\frac{s_l}{\varrho_L}\right)^{2n+1} + {n} \right) \\
    &\phantom{=} \times \left({(n+1)} \left(\frac{s_k}{\varrho_L}\right)^{2n+1} + {n} \right) Y_{n,j}(\eta_l) Y_{n,j}(\eta_k) \\
     &= \frac{1}{4\pi}\sum_{n=1}^\infty \left(\kappa_{n}^{(2)}\right)^{-2}  \frac{2n+1}{{n\varrho_0}} \left(\beta^{(L)}_n\right)^2 \left(\frac{\varrho_0^2}{s_ls_k}\right)^{n+1}   \left({(n+1)} \left(\frac{s_l}{\varrho_L}\right)^{2n+1} + {n} \right) \\
    &\phantom{=} \times \left({(n+1)} \left(\frac{s_k}{\varrho_L}\right)^{2n+1} + {n} \right) P_{n}(\eta_l\cdot\eta_k) . 
  \end{align*}
  Note that the addition theorem, see \eqref{eq:AdditionThm}, is used in the last step. 
\end{proof}

\printbibliography

@Article{Akram2011,
  Title                    = {{A study of differential operators for complete orthonormal systems on a {3D} ball}},
  Author                   = {Akram, M. and Amina, I. and Michel, V.},
  Journal                  = {Int. J. Pure Appl. Math.},
  Year                     = {2011},
  Pages                    = {489--506},
  Volume                   = {73}
}

@Article{Amirbekyan2008,
  Title                    = {{Splines on the three-dimensional ball and their application to seismic body wave tomography}},
  Author                   = {Amirbekyan, A. and Michel, V.},
  Journal                  = {Inverse Probl.},
  Year                     = {2008},
  Pages                    = {015022},
  Volume                   = {24}
}

@incollection{Bauer2015,
  doi = {10.1007/978-3-642-54551-1_99},
  url = {https://doi.org/10.1007/978-3-642-54551-1_99},
  Edition                  = {2},
  Publisher                = {Springer},
  Year                     = {2015},
  Address                  = {Berlin, Heidelberg},
  pages = {1713--1774},
  author = {Bauer, F. and Gutting, M. and Lukas, M. A.},
  title = {Evaluation of parameter choice methods for regularization of ill-posed problems in geomathematics},
  booktitle = {Handbook of Geomathematics},
  Editor = {Freeden, W. and Nashed, M. Z. and Sonar, T.}
}

@article{Bauer2011,
  doi = {10.1016/j.matcom.2011.01.016},
  url = {https://doi.org/10.1016/j.matcom.2011.01.016},
  year  = {2011},
  month = may,
  publisher = {Elsevier {BV}},
  volume = {81},
  number = {9},
  pages = {1795--1841},
  author = {Bauer, F. and Lukas, M. A.},
  title = {Comparing parameter choice methods for regularization of ill-posed problems},
  journal = {Math. Comput. Simulat.}
}

@book{Bauer2001,
  title={Measure and Integration Theory},
  author={Bauer, H.},
  isbn={9783110866209},
  series={De Gruyter Studies in Mathematics},
  year={2001},
  publisher={De Gruyter},
  address = {Berlin}
}

@Phdthesis{Bayer2000,
  author = {Bayer, M.},
  title = {Geomagnetic Field Modelling From Satellite Data by First and Second Generation Vector Wavelets},
  Institution              = {University of Kaiserslautern, Department of Mathematics, Geomathematics Group},
  Year                     = {2000},
  address                  = {Shaker-Verlag Aachen}
}

@PhdThesis{Berkel2009,
  Title                    = {{Multiscale Methods for the Combined Inversion of Normal Mode and Gravity Variations}},
  Author                   = {Berkel, P.},
  Institution              = {University of Kaiserslautern, Department of Mathematics, Geomathematics Group},
  Address                  = {Shaker-Verlag Aachen},
  Year                     = {2009}
}

@article{Berkel2010b,
  doi = {10.1007/s13137-010-0007-5},
  Url = {https://doi.org/10.1007/s13137-010-0007-5},
  year  = {2010},
  month = aug,
  publisher = {Springer Nature},
  volume = {1},
  number = {2},
  pages = {167--204},
  author = {Berkel, P. and Fischer, D. and Michel, V.},
  title = {Spline multiresolution and numerical results for joint gravitation and normal-mode inversion with an outlook on sparse regularisation},
  journaltitle = {Int. J. Geomath.}
}

@Article{Berkel2010a,
  Title                    = {On mathematical aspects of a combined inversion of gravity and normal mode variations by a spline method},
  Author                   = {Berkel, P. and Michel, V.},
  Journal                  = {Math. Geosci.},
  Year                     = {2010},

  Month                    = aug,
  Number                   = {7},
  Pages                    = {795--816},
  Volume                   = {42},
  Doi                      = {10.1007/s11004-010-9297-2},
  Publisher                = {Springer Nature},
  Url                      = {https://doi.org/10.1007/s11004-010-9297-2}
}

@Phdthesis{Beth2000,
  author = {Beth, S.},
  title = {Multiscale Approximation by Vector Radial Basis Functions on the Sphere},
  Institution              = {University of Kaiserslautern, Department of Mathematics, Geomathematics Group},
  Year                     = {2000},
  address                  = {Shaker-Verlag Aachen}
}

@Incollection{Cohen2009,
  Title                    = {{Magnetoencephalography}},
  Author                   = {Cohen, D and Halgren, E},
  Editor                   = {Squire, L. R.},
  Pages                    = {615--622},
  Publisher                = {Academic Press},
  Year                     = {2009},

  Address                  = {Oxford},
  Volume                   = {5},

  Booktitle                = {Encyclopedia of Neuroscience}
}

@Book{Cook2002,
  Title                    = {{The Theory of the Electromagnetic Field}},
  Author                   = {Cook, D. M.},
  Publisher                = {Courier Dover Publications},
  Year                     = {2002},

  Address                  = {Mineola NY}
}

@Article{Dassios2013,
  Title                    = {The definite non-uniqueness results for deterministic {EEG} and {MEG} data},
  Author                   = {Dassios, G. and Fokas, A. S.},
  Journal                  = {Inverse Probl.},
  Year                     = {2013},

  Month                    = jun,
  Number                   = {6},
  Pages                    = {065012},
  Volume                   = {29},
  Doi                      = {10.1088/0266-5611/29/6/065012}
}

@Article{Dassios2009c,
  Title                    = {Electro-magneto-encephalography and fundamental solutions},
  Author                   = {Dassios, G. and Fokas, A. S.},
  Journal                  = {Q. Appl. Math.},
  Year                     = {2009},

  Month                    = may,
  Number                   = {4},
  Pages                    = {771--780},
  Volume                   = {67},
  Doi                      = {10.1090/s0033-569x-09-01144-7},
  Publisher                = {American Mathematical Society ({AMS})},
  Url                      = {https://doi.org/10.1090/s0033-569x-09-01144-7}
}

@book{Dick2013,
address = {Berlin, Heidelberg},
doi = {10.1007/978-3-642-41095-6},
editor = {Dick, J and Kuo, F Y and Peters, G W and Sloan, I H},
isbn = {978-3-642-41094-9},
publisher = {Springer Berlin Heidelberg},
series = {Springer Proceedings in Mathematics \& Statistics},
title = {{Monte Carlo and Quasi-Monte Carlo Methods 2012}},
volume = {65},
year = {2013}
}

@Book{Edmonds57,
  Title                    = {{Angular Momentum in Quantum Mechanics}},
  Author                   = {Edmonds, A. R.},
  Publisher                = {Princeton University Press},
  Year                     = {1957},

  Address                  = {Princeton}
}

@manual{Elekta2005,
  Title                    = {{Elekta Neuromag System Hardware Technical Manual}},

  Author                   = {{Elekta Neuromag}},
  Year                     = {2005},
  Month                    = sep,

  Subtitle                 = {Revision F},
  url = {http://imaging.mrc-cbu.cam.ac.uk/meg/VectorviewDescription?action=AttachFile&do=get&target=HardwareTechnical.pdf},
  urldate = {2018-06-29}
}

@Book{Engl1996,
  Title                    = {{Regularization of Inverse Problems}},
  Author                   = {Engl, H. W. and Hanke, M. and Neubauer, A.},
  Publisher                = {Kluwer Academic Publishers},
  Year                     = {1996},

  Address                  = {Dordrecht},


  Doi                      = {doi:10.1007/978-94-009-1740-8}
}

@Article{Fokas2009,
  Title                    = {Electro-magneto-encephalography for a three-shell model: distributed current in arbitrary, spherical and ellipsoidal geometries},
  Author                   = {Fokas, A. S.},
  Journal                  = {J. R. Soc. Interface},
  Year                     = {2009},
  Number                   = {34},
  Pages                    = {479--488},
  Volume                   = {6},
  Doi                      = {10.1098/rsif.2008.0309},
  Publisher                = {The Royal Society},
  Url                      = {https://doi.org/10.1098/rsif.2008.0309}
}

@Article{Fokas2012,
  Title                    = {{Electro-magneto-encephalography for the three-shell model: numerical implementation via splines for distributed current in spherical geometry}},
  Author                   = {Fokas, A. S. and Hauk, O. and Michel, V.},
  Journal                  = {Inverse Probl.},
  Year                     = {2012},
  Number                   = {3},
  Pages                    = {035009},
  Volume                   = {28}
}

@Article{Fokas2012a,
  Title                    = {Electro-magneto-encephalography for the three-shell model: minimal $\mathrm{L}^2$-norm in spherical geometry},
  Author                   = {Fokas, A. S. and Kurylev, Y.},
  Journal                  = {Inverse Probl.},
  Year                     = {2012},

  Month                    = feb,
  Number                   = {3},
  Pages                    = {035010},
  Volume                   = {28},


  Doi                      = {10.1088/0266-5611/28/3/035010},
  Publisher                = {{IOP} Publishing},
  Url                      = {https://doi.org/10.1088/0266-5611/28/3/035010}
}

@article{Fornberg1988,
author = {Fornberg, B},
doi = {10.2307/2008770},
issn = {00255718},
journal = {Mathematics of Computation},
month = {oct},
number = {184},
pages = {699},
title = {{Generation of finite difference formulas on arbitrarily spaced grids}},
volume = {51},
year = {1988}
}

@article{Freeden1981,
author = {Freeden, W},
journal = {Math. Meth. Appl. Sci.},
pages = {551--575},
title = {{On spherical spline interpolation and approximation}},
volume = {3},
year = {1981}
}

@article{Freeden1993,
  doi = {10.1002/mma.1670160302},
  Url = {https://doi.org/10.1002/mma.1670160302},
  year  = {1993},
  month = mar,
  publisher = {Wiley-Blackwell},
  volume = {16},
  number = {3},
  pages = {151--183},
  author = {Freeden, W. and Gervens, T.},
  title = {Vector spherical spline interpolation{\textemdash}basic theory and computational aspects},
  journaltitle = {Math. Methods. Appl. Sci.}
}

@Book{Freeden1998,
  Title                    = {{Constructive Approximation on the Sphere (With Applications to Geomathematics)}},
  Author                   = {Freeden, W. and Gervens, T. and Schreiner, M.},
  Publisher                = {Oxford University Press},
  Year                     = {1998},

  Address                  = {Oxford}
}

@article{Grech2008,
  doi = {10.1186/1743-0003-5-25},
  url = {https://doi.org/10.1186/1743-0003-5-25},
  year  = {2008},
  publisher = {Springer Nature},
  volume = {5},
  number = {1},
  pages = {article number 25},
  author = {Grech, R. and  Cassar, T. and Muscat, J. and Camilleri, K. P. and Fabri, S. G. and Zervakis, M. and Xanthopoulos, P. and  Sakkalis, V. and Vanrumste, B.},
  title = {Review on solving the inverse problem in {EEG} source analysis},
  journal = {J Neuroeng Rehabil}
}

@article{Gutting2017,
  doi = {10.3389/fams.2017.00010},
  url = {https://doi.org/10.3389/fams.2017.00010},
  year  = {2017},
  month = jun,
  publisher = {Frontiers Media {SA}},
  volume = {3},
  note = {Article 10},
  author = {Gutting, M. and Kretz, B. and Michel, V. and Telschow, R.},
  title = {Study on parameter choice methods for the {RFMP} with respect to downward continuation},
  journal = {Front. Appl. Math. Stat.},
  Number = {10}
}

@Article{Hamalainen1993,
  Title                    = {{Magneto\-ence\-phalo\-graphy--theory, instrumentation, and applications to noninvasive studies of the working human brain}},
  Author                   = {Hämäläinen, M. S. and Hari, R. and Ilmoniemi, R. J. and Knuutila, J. and Lounasmaa, O. V.},
  Journal                  = {Rev. Mod. Phys.},
  Year                     = {1993},

  Month                    = apr,
  Number                   = {2},
  Pages                    = {413--505},
  Volume                   = {65},
  Doi                      = {10.1103/RevModPhys.65.413}
}

@article{Hashemzadeh2020,
  doi = {10.1098/rsif.2019.0831},
  url = {https://doi.org/10.1098/rsif.2019.0831},
  year = {2020},
  month = feb,
  publisher = {The Royal Society},
  volume = {17},
  number = {163},
  pages = {20190831},
  author = {Parham Hashemzadeh and A. S. Fokas and C. B. Sch\"{o}nlieb},
  title = {A hybrid analytical{\textendash}numerical algorithm for determining the neuronal current via electroencephalography},
  journal = {Journal of the Royal Society Interface}
}

@Book{HE2005,
  Title                    = {Modeling and Imaging of Bioelectrical Activity: Principles and Applications},
  Editor                   = {He, B.},
  Publisher                = {Kluwer Academic Publishers},
  Year                     = {2005},

  Address                  = {New York},
  Series                   = {Bioelectric Engineering},
  Doi                      = {10.1007/978-0-387-49963-5},
  Url                      = {https://doi.org/10.1007/978-0-387-49963-5}
}

@book{Kandel2013,
    author = {Kandel, E. R. and Schwartz, J. H. and Jessell, T. M. and Siegelbaum, S. A. and Hudspeth, A. J.},
    edition = {5},
    publisher = {McGraw-Hill Medical},
    title = {Principles of Neural Science},
    address = {New York},
    year = {2013}
}

@article{Leweke2020,
  doi = {10.1088/1361-6420/ab291f},
  url = {https://doi.org/10.1088/1361-6420/ab291f},
  year = {2020},
  month = feb,
  publisher = {{IOP} Publishing},
  volume = {36},
  number = {3},
  pages = {035003},
  author = {Leweke, S and Michel, V and Fokas, A. S.},
  title = {Electro-magnetoencephalography for a spherical multiple-shell model: novel integral operators with singular-value decompositions},
  journal = {Inverse Probl.}
}

@Phdthesis{Leweke2018,
  author = {Leweke, S.},
  title = {The Inverse Magneto-electroencephalography Problem for the Spherical Multiple-shell Model--Theoretical Investigations and Numerical Aspects},
  Institution              = {University of Siegen, Department of Mathematics, Geomathematics Group},
  Year                     = {2018}, 
  url  = {https://nbn-resolving.org/urn:nbn:de:hbz:467-13967}
}

@book{MATLAB:2021,
year = {2021},
author = {MATLAB},
title = {version 9.10.0 (R2021a)},
publisher = {The MathWorks Inc.},
address = {Natick, Massachusetts}
}

@article{Mercer1909,
author = {Mercer, J},
issn = {0950-1207},
journal = {Proceedings of the Royal Society of London. Series A, Containing Papers of a Mathematical and Physical Character},
month = {nov},
number = {559},
pages = {69--70},
title = {{Functions of positive and negative type, and their connection with the theory of integral equations}},
volume = {83},
year = {1909}
}

@Book{Michel2013,
  Title                    = {{Lectures on Constructive Approximation. Fourier, Spline, and Wavelet Methods on the Real Line, the Sphere, and the Ball}},
  Author                   = {Michel, V.},
  Publisher                = {Birkhäuser},
  Year                     = {2013},

  Address                  = {New York}
}

@Incollection{Michel2015Geo,
  Title                    = {Tomography: problems and multiscale solutions},
  Author                   = {Michel, V.},
  Editor                   = {Freeden, W. and Nashed, M. Z. and Sonar, T.},
  Pages                    = {2087--2119},
  Edition                  = {2},
  Publisher                = {Springer},
  Year                     = {2015},
  Address                  = {Berlin, Heidelberg},
  Booktitle                = {Handbook of Geomathematics}
}

@Article{Michel2015,
  Title                    = {On the null space of a class of Fredholm integral equations of the first kind},
  Author                   = {Michel, V. and Orzlowski, S.},
  Journal                  = {J. Inverse Ill-Posed Probl.},
  Year                     = {2016},

  Month                    = jan,
  Number                   = {6},
  Volume                   = {24},
  Pages                    = {687--710},
  Doi                      = {10.1515/jiip-2015-0026},
  Publisher                = {Walter de Gruyter {GmbH}},
  Url                      = {https://doi.org/10.1515/jiip-2015-0026}
}

@Article{deMunck1988,
  Title                    = {The potential distribution in a layered anisotropic spheroidal volume conductor},
  Author                   = {de Munck, J. C.},
  Journal                  = {J. Appl. Phys.},
  Year                     = {1988},

  Month                    = jul,
  Number                   = {2},
  Pages                    = {464--470},
  Volume                   = {64},
  Doi                      = {10.1063/1.341983},
  Publisher                = {{AIP} Publishing},
  Url                      = {https://doi.org/10.1063/1.341983}
}

@Article{deMunck1993,
  Title                    = {A fast method to compute the potential in the multisphere model ({EEG} application)},
  Author                   = {de Munck, J. C. and Peters, M. J.},
  Journal                  = {{IEEE} Trans. Biomed. Eng.},
  Year                     = {1993},
  Number                   = {11},
  Pages                    = {1166--1174},
  Volume                   = {40},


  Doi                      = {10.1109/10.245635},
  Publisher                = {Institute of Electrical {and} Electronics Engineers ({IEEE})},
  Url                      = {https://doi.org/10.1109/10.245635}
}

@article{Nashed1974,
author = {Nashed, M. Z. and Wahba, G},
doi = {10.1137/0505095},
issn = {0036-1410},
journal = {SIAM Journal on Mathematical Analysis},
month = {nov},
number = {6},
pages = {974--987},
title = {Generalized inverses in reproducing kernel spaces: an approach to regularization of linear operator equations},
volume = {5},
year = {1974}
}

@Article{Novikoff1938,
  Title                    = {{Sur le probl{\`e}me inverse du potentiel}},
  Author                   = {Novikoff, P.},
  Journal                  = {C.R. Acad. Sci. U.R.S.S.},
  Year                     = {1938},
  Pages                    = {165--168},
  Volume                   = {18}
}

@Book{Plonsey69,
  Title                    = {{Biomagnetic Phenomena}},
  Author                   = {Plonsey, R. },
  Publisher                = {McGraw-Hill},
  Year                     = {1969},

  Address                  = {New York}
}

@Article{Plonsey1967,
  Title                    = {{Considerations of quasi-stationarity in electrophysiological systems}},
  Author                   = {Plonsey, R. and Heppner, D. B.},
  Journal                  = {Bull. Math. Biophys.},
  Year                     = {1967},
  Number                   = {4},
  Pages                    = {657--664},
  Volume                   = {29},
  publisher                = {Kluwer Academic Publishers},
  doi                      = {10.1007/BF02476917},
  Pmid                     = {5582145}
}

@online{Roberts2018,
  author = {Roberts, M},
  title = {The unreasonable effectiveness of quasirandom sequences},
  year = 2018,
  url = {http://extremelearning.com.au/unreasonable-effectiveness-of-quasirandom-sequences/#GeneralizingGoldenRatio},
  urldate = {2021-11-19}
}

@article{Schneider1996,
author = {Schneider, F},
doi = {10.1007/s001900050070},
issn = {09497714},
journal = {Journal of Geodesy},
number = {1},
pages = {2--15},
title = {{The solution of linear inverse problems in satellite geodesy by means of spherical spline approximation}},
volume = {71},
year = {1996}
}

@Book{Szego1975,
  Title                    = {{Orthogonal Polynomials}},
  Author                   = {Szegö, G.},
  Publisher                = {American Mathematical Society},
  Year                     = {1975},
  edition = {4},

  Address                  = {Providence, Rhode Island}
}

@article{Wahba1981,
author = {Wahba, G},
doi = {10.1137/0902002},
issn = {0196-5204},
journal = {SIAM Journal on Scientific and Statistical Computing},
month = {mar},
number = {1},
pages = {5--16},
title = {{Spline interpolation and smoothing on the sphere}},
url = {http://epubs.siam.org/doi/10.1137/0902002},
volume = {2},
year = {1981}
}

@Misc{Mathematica10,
  Title                    = {{Mathematica}},

  Author                   = {{Wolfram Research, Inc.}},
  Note                     = {Version 10.0},
  Year                     = {2014},

  Address                  = {Champaign, Illinois},
  Publisher                = {Wolfram Research, Inc.}
}

@Book{Yosida1988,
  Title                    = {Functional Analysis},
  Author                   = {Yosida, K.},
  Publisher                = {Springer},
  Year                     = {1980},
  edition = {6},

  Address                  = {Berlin, Heidelberg, New York},

  ISBN                     = {3540102108}
}

@article{Freeden1981b,
author = {Freeden, W.},
title = {On approximation by harmonic splines},
journal = {Manuscr. Geodaet.},
volume = {6},
pages = {193–-244},
year  = {1981}
}

@Book{Freeden2009,
  Title                    = {{Spherical Functions of Mathematical Geosciences. A Scalar, Vectorial and Tensorial Setup}},
  Author                   = {Freeden, W. and Schreiner, M.},
  Publisher                = {Springer},
  Year                     = {2009},

  Address                  = {Berlin}
}

\end{document}
